\documentclass[reqno,tbtags,a4paper,12pt]{amsproc}

\usepackage[T1,T2A]{fontenc}
\usepackage[utf8]{inputenc}

\usepackage[bbgreekl]{mathbbol}
\usepackage{amsfonts}

\DeclareSymbolFontAlphabet{\mathbb}{AMSb}
\DeclareSymbolFontAlphabet{\mathbbl}{bbold}

\usepackage{amsmath,amssymb,amsthm,amsfonts,amscd,array,graphicx,color}
\usepackage[in]{fullpage}
\usepackage{tikz-cd}

\newcommand{\T}{\mathcal{T}}
\newcommand{\CH}{\mathcal{H}}
\newcommand{\CF}{\mathcal{F}}
\newcommand{\CC}{\mathcal{C}}
\newcommand{\CP}{\mathcal{P}}
\newcommand{\B}{\mathcal{B}}
\newcommand{\F}{\mathbb{F}}
\newcommand{\Z}{\mathbb{Z}}
\newcommand{\R}{\mathbb{R}}

\newcommand{\hX}{\widehat{X}}
\newcommand{\hx}{\hat{x}}
\newcommand{\hy}{\hat{y}}

\newcommand{\ord}{\mathop{\mathrm{ord}}\nolimits}
\newcommand{\Hom}{\mathop{\mathrm{Hom}}\nolimits}

\newcommand{\im}{\mathop{\mathrm{im}}\nolimits}

\newcommand{\Sym}{\mathop{\mathrm{Sym}}\nolimits}
\newcommand{\Aut}{\mathop{\mathrm{Aut}}\nolimits}

\newcommand{\ba}{\mathbf{a}}
\newcommand{\bu}{\mathbf{u}}

\newtheorem{theorem}{Theorem}[section]

\newtheorem{propos}[theorem]{Proposition}
\newtheorem{cor}[theorem]{Corollary}
\newtheorem{lem}[theorem]{Lemma}
\newtheorem{question}[theorem]{Question}
\newtheorem{prob}[theorem]{Problem}
\theoremstyle{definition}
\newtheorem{constr}[theorem]{Construction}
\newtheorem{defin}[theorem]{Definition}
\newtheorem{remark}[theorem]{Remark}
\newtheorem{example}[theorem]{Example}

\author{V.\,M. Buchstaber}
\author{A.\,A. Gaifullin}
\author{A.\,P.~Veselov}

\thanks{The work of V.\,M.~Buchstaber and A.\,A.~Gaifullin was performed at the Steklov International Mathematical Center and supported by the Ministry of Science and Higher Education of the Russian Federation (agreement no. 075-15-2022-265).}

\address{\textnormal{\bf Victor M. Buchstaber}}

\address{Steklov Mathematical Institute of Russian Academy of Sciences, Moscow, Russia}

\address{Lomonosov Moscow State University, Russia}

\address{National Research University ``Higher School of Economics'', Moscow, Russia}

\email{buchstab@mi-ras.ru}

\address{}

\address{\textnormal{\bf Alexander A. Gaifullin}}

\address{Steklov Mathematical Institute of Russian Academy of Sciences, Moscow, Russia}

\address{Skolkovo Institute of Science and Technology, Russia}

\address{Lomonosov Moscow State University, Russia}

\address{Institute for Information Transmission Problems of the Russian Academy of Sciences (Kharkevich Institute), Moscow, Russia}

\email{agaif@mi-ras.ru}

\address{}

\address{\textnormal{\bf Alexander P. Veselov}}

\address{Loughborough University, UK}

\email{A.P.Veselov@lboro.ac.uk}


\title{Classification of involutive commutative two-valued groups}

\begin{document}

\begin{abstract}
A complete classification of finitely generated involutive commutative two-valued groups is obtained. Three series of such two-valued groups are constructed: principal, unipotent and special, and it is shown that any finitely generated involutive commutative two-valued group is isomorphic to a two-valued group belonging to one of these series.
A number of classification results are obtained for topological involutive commutative two-valued groups in the Hausdorff and locally compact cases.
The classification of algebraic involutive two-valued groups in the one-dimensional case is also discussed.
\end{abstract}

\maketitle

\tableofcontents

\section{Introduction}\label{section_formulation}

A multi-valued group is a generalization of an ordinary group such that the product of any pair of elements is a multiset, that is, an unordered set of elements (possibly with repetitions). The theory of multi-valued multiplications, where the cardinality of the result of multiplication of two elements for a given multi-valued group can take different values, has rich history dating back to the 19th century, see ~\cite{Lit85}.

In the present paper we study the $n$-valued groups. The key to the definition of an $n$-valued group is fixing the cardinality~$n$ of a multiset that is a product of two elements. The concept of $n$-valued group originates from the construction of V.M.~Buchstaber and S.P.~Novikov~\cite{BuNo71} in the theory of characteristic classes of vector bundles. The foundations of the algebraic theory of $n$-valued groups were laid in the works of V.M.~Buchstaber. This theory and its applications were developed by V.M.~Buchstaber, his students (A.N.~Kholodov, P.V.~Yagodovsky and others) and co-authors (E.~ Rees, A.P.~Veselov, M.I.~Monastyrsky, V.~Dragovic and others), see references in surveys~\cite{Buc06} and ~\cite{BuVe19}.

In this paper we will deal only with two-valued groups. We will denote by $\Sym^2(X)$ the second symmetric power of a set~$X$, that is, the set of all two-element multisets with elements from~$X$. We will use square brackets to enumerate elements of multisets.

\begin{defin}
\textit{Two-valued group} is a set~$X$ with a two-valued multiplication $*\colon X\times X\to \Sym^2(X)$, identity element $e\in X$, and the inverse $x\mapsto x^{-1}$ with the following properties.
\begin{itemize}
\item \textit{Associativity}: for any three elements $x,y,z\in X$ there is a coincidence of four-element multisets
$$
x*(y*z)=(x*y)*z.
$$
\item \textit{Strong identity}: for any element~$x$ there are equalities
$$
e*x=x*e=[x,x].
$$
\item \textit{Existence and uniqueness of the inverse}: for any element $x\in X$, each of the multisets $x*x^{-1}$ and $x^{-1}*x$ contains the identity~$e$; moreover, if the element $y$ is such that one of the multisets $x*y$ and~$y*x$ contains the identity, then $y = x^{-1}$.
\end{itemize}
A two-valued group is called \textit{commutative} if
$$
x*y=y*x
$$
for all $x,y\in X$.
\end{defin}

\begin{remark}
The original definition of a two-valued group (see ~\cite{Buc06}) did not require the uniqueness of the inverse element. This additional condition was used in ~\cite{BVEP96} and ~\cite{BuVe19}. For the present work this condition turned out to be important.
\end{remark}

The most important source of $n$-valued groups is given by the construction of coset multi-valued groups, see ~\cite[ Section ~6]{ Buc06}. Coset $n$-valued groups are constructed from a pair consisting of a single-valued group~$G$ and a finite subgroup $H$ of its automorphism group. At the same time, the theory of $n$-valued groups is not exhausted by coset groups. Firstly, different pairs~$(G,H)$ can lead to isomorphic $n$-valued groups; secondly, there is a rich class of $n$-valued groups that are not coset, see \cite{Buc06},~\cite{BVEP96}. We will need the following construction of the coset two-valued groups.

\begin{constr}[Coset two-valued group]
Let $G$ be an ordinary (single-valued) group and $\iota\colon G\to G$ be its involutive (that is, such that $\iota^2=\mathrm{id}$) automorphism. Then the quotient set $X=G/\iota$ is endowed with the structure of a two-valued group in such a way that
$$
\pi(g)*\pi(h)=\bigl[\pi(gh),\pi(g\iota(h))\bigr]
$$
for all $g,h\in G$, where $\pi\colon G\to X$ is the natural projection. The identity of this two-valued group is the image of the identity of the group~$G$, and the inverse element is defined by the formula $\pi(g)^{-1} = \pi(g^{-1})$.
In the case of abelian groups $G$ there is an important example of the involution, namely taking the inverse element (antipodal involution): $\iota_{\ba}(g)= g^{-1}.$
\end{constr}

One of the important features of the theory of $n$-valued groups is a high complexity of classification and enumeration problems. This is due to the fact that the number of isomorphism classes of $n$-valued groups on the set of $k$ elements grows extremely rapidly with the growth of $k$ and $n$. The same effect holds even if we consider only commutative $n$-valued groups. We are aware of attempts to solve the problem of enumerating isomorphism classes of $n$-valued groups for given $n$ and~$k$ using high-performance computing technology, which did not lead to significant results due to the computational complexity of the problem.

In this connection the problem of identifying such classes of $n$-valued groups, for which classification problems are affordable, is of natural interest. An important class of this type is the class of involutive two-valued groups introduced by V.M.~Buchstaber and A.P.~Veselov in~\cite{BuVe19} in connection with the Conway topograph.

\begin{defin}\label{defin_invol}
An element $x$ of a two-valued group~$X$ is called \textit{weak involution} if $x^{-1}=x$ or, equivalently, if the multiset $x*x$ contains the identity~$e$. An element $x\in X$ is called \textit{strong involution} if $x*x=[e,e]$. A two-valued group~$X$ is called \textit{involutive} if it consists of weak involutions, that is, $x^{-1}=x$ for all $x\in X$.
\end{defin}

\begin{remark}
Note that in ~\cite{BVEP96} the concept of an involutive $n$-valued group is used in a different sense, namely, as the involutivity of the operation of taking the inverse element $x\mapsto x^{-1}.$ It is clear that the class of involutive groups in the sense of ~\cite{BVEP96} contains the class of involutive groups in sense of work~\cite{BuVe19}. The main result of ~\cite{BVEP96} is the proof that the group algebras of the involutive $n$-valued groups introduced in it coincide with the so-called combinatorial algebras, see ~\cite{BaIt87}. In this paper we use the term `involutive two-valued group' in the sense of ~\cite{BuVe19}.
\end{remark}

The main goal of this paper is to classify involutive commutative two-valued groups. The problem of classifying such two-valued groups was posed in ~\cite{BuVe19}. In that paper it was conjectured that every involutive commutative two-valued group is isomorphic to a coset two-valued group of the form $A/\iota_{\ba}$, where $A$ is an abelian group and $\iota_{\ba}(g)=g^{-1}$ is an antipodal involution on this group. We will show that this conjecture is false even for finite two-valued groups.

We consider the classification problem for involutive commutative two-valued groups in the following classes:
\begin{itemize}
\item finitely generated two-valued groups (theorem ~\ref{theorem_main_fg}, the proof of which occupies sections~\ref{section_1gen}--\ref{section_final_proof}),
\item arbitrary two-valued groups without topology (Sec. ~\ref{section_non_fg}),
\item Hausdorff topological two-valued groups (Sec. ~\ref{section_top}),
\item locally compact Hausdorff two-valued groups (Sec. ~\ref{section_lc}),
\item algebraic two-valued groups (Sec. ~\ref{section_alg}).
\end{itemize}

A complete classification will be obtained in the case of finitely generated groups (Theorem ~\ref{theorem_main_fg}), as well as in the case of compact topological groups without small subgroups (Theorem ~\ref{theorem_compact}).

Most of the work will be devoted to the finitely generated case. After that we will discuss how our approach extends to other classes of two-valued groups and what results and new questions are obtained along the way.

Let us now formulate a classification theorem for finitely generated involutive commutative two-valued groups. We will show that, along with \textit{principal series}, which consists of coset two-valued groups of the form $A/\iota_{\ba}$, there are two more series of finitely generated involutive commutative two-valued groups; we call them \textit{unipotent} and \textit{special series}. Interestingly, the two-valued groups of the unipotent series are also obtained using the coset construction from abelian groups, but with respect to a non-antipodal involution. Two-valued groups from the special series are not coset, see Proposition ~\ref{propos_no_coset} below. In particular, our result gives new constructions of non-coset two-valued groups.

We now present the constructions of those series of involutive commutative two-valued groups.

Throughout this article, we will denote by~$C_n$ the cyclic group of order~$n$; including $C_{\infty}$, which is an infinite cyclic group. It will be convenient for us to use the multiplicative notation for all groups, including abelian ones.

\subsection*{1. Principal series} Let $A$ be a (single-valued) abelian group. Consider \textit{antipodal involution} $\iota_{\ba}\colon A\to A$ defined by the formula $\iota_{\ba}(a)=a^{-1}$. Then the coset two-valued group
\begin{equation*}
X^{\ba}_A=A/\iota_{\ba}
\end{equation*}
is an involutive commutative group.

If $A$ is a finitely generated abelian group, then, as it is well known, it can be represented uniquely up to isomorphism as a product $C_{d_1}\times C_{d_2}\times \cdots\times C_{d_k}$, where $2\le d_i\le \infty$ and $d_i$ divides $d_{i+1}$ for all $i=1,\dots, k-1$ (see, for example, \cite{Vinberg}). Introduce the notation
$$
X_{d_1,\ldots,d_k}^{\ba}=X_{C_{d_1}\times \cdots\times C_{d_k}}^{\ba}.
$$
It can be directly verified that if the group~$A$ is finite, then the number of elements in the two-valued group~$X^{\ba}_A$ is equal to
$$
\frac12(d_1d_2\cdots d_k+2^r),
$$
where $r$ is the number of even numbers among numbers~$d_1,\ldots,d_k$.

\begin{remark}
The simplest but important case of involutive two-valued groups of the principal series are the groups $X_{m\times 2}^{\ba}=C_2^m/\iota_{\ba}$ (where $m\times 2$ denotes the sequence $2,\ldots,2$ of $m$ twos). The two-valued group $X_{m\times 2}^{\ba}$ is the so-called \textit{double} of the group~$C_2^m$. This means that $X_{m\times 2}^{\ba}$ coincides with ~$C_2^m$ as sets, and the operation in~$X_{m\times 2}^{\ba}$ is obtained by doubling the operation in~$C_2^m$, that is, $x*y=[xy,xy]$. Note that the double can be defined for any group~$G$. The result is always a two-valued group, which is involutive in the sense of ~\cite{BVEP96}. It is commutative and involutive in the sense of ~\cite{BuVe19} only if $G$ is a vector space over a field of two elements~$\F_2$.
\end{remark}

\begin{remark}
It is natural to consider two-valued groups~$X_{\infty,\ldots,\infty}^{\ba}$ as \textit{free} involutive commutative two-valued groups. The two-valued groups $X_{\infty}^{\ba}=C_{\infty}/\iota_{\ba}$ and $X_{\infty,\infty}^{\ba}=(C_{\infty}\times C_{\infty})/\iota_{\ba}$ are studied in detail in ~\cite{BuVe19}, where they were called \textit{Buchstaber-Novikov two-valued group} and \textit{Conway two-valued group} and denoted by ~$\mathbb{X}_1$ and ~$\mathbb{X}_2$ respectively.
\end{remark}

\subsection*{2. Unipotent series} We need the notion of the \textit{Boolean group}. A group is called \textit{Boolean} if each non-identity element of it has order~$2$, see ~\cite[ Chapter ~1]{ GiHa09}. It is easy to show that any such group is abelian. Thus, a Boolean group is a vector space over a field of two elements~$\F_2$ with an operation written in multiplicative form. For a Boolean group~$V$ we will denote by~$\dim V$ its dimension as a vector space.

Let $V$ be a Boolean group.
Consider \textit{unipotent involution}
$\iota_{\bu}\colon V\times V\to V\times V$, determined by the formula
\begin{equation}
\label{unip}
\iota_{\bu}(a,b)=(a,ab).
\end{equation}
Then the coset two-valued group
\begin{equation*}
X^{\bu}_V=(V\times V)/\iota_{\bu}
\end{equation*}
is an involutive commutative two-valued group.

\begin{remark}
Using the automorphism of the group~$V\times V$ defined by the formula $(a,b)\mapsto (ab,b)$, the involution~$\iota_{\bu}$ is translated into the permutation involution $$\iota_{\bf t}: (a,b)\mapsto (b,a),$$ but for our purposes it is more convenient to consider the involution in the unipotent form (\ref{unip}).
The equivalence of these two involutions is explained by the well-known isomorphism
$$
GL(2, \F_2)\cong S_3
$$
with the permutation group of a three-element set in which all elements of order two are conjugated.
\end{remark}

It can be directly verified that if $\dim V=n<\infty$, that is, $V\cong C_2^n$, then the two-valued group
$$
X^{\bu}_n=X^{\bu}_{C_2^n}
$$
consists of $2^{2n-1}+2^{n-1}$ elements.

If $n=1$, then $X_1^{\bu}=X^{\bu}_{C_2}$ is a three-element group~$\{e,v,x\}$ with multiplication table $v*v=[e,e]$, $v*x=[x,x]$, $x*x=[e,v]$. It is easy to see that $X^{\bu}_{C_2}\cong X^{\ba}_{C_4}$.

When $n=2$ there is a less obvious isomorphism of $10$-element two-valued groups $X^{\bu}_2\cong X^{\ba}_{4,4}$, that is, $X^{\bu}_{C_2\times C_2}\cong X^{\ba}_{C_4\times C_4}$ (see Example~\ref{ex_iso} and section ~\ref{section_124}).

Thus, when $\dim V\le 2$ the two-valued groups $X^{\bu}_V$ are contained in the principal series of examples. Nevertheless, as we will see below, for $\dim V\ge 3$, the two-valued groups~$X^{\bu}_V$ are not isomorphic to any two-valued groups of the principal series.

\begin{remark}
It is easy to see that a coset two-valued group $G/\iota$ is involutive if and only if for any  element $g\in G$ one of the two equalities $g^{-1}=g$ or $g^{-1}=\iota(g)$ holds. The first possibility is realised for two-valued groups of the unipotent series, the second for two-valued groups of the principal series.
\end{remark}

\subsection*{3. Special series}

V.\,M.~Buchstaber and E.~Rees proposed a way of constructing  two-valued groups by adding an additional element to a usual (single-valued) group, see Example~(6) in Section~3 of the paper~\cite{BR-97}. The special series of involutive commutative two-valued groups is obtained by applying this construction to Boolean groups.

Let $V$ be a Boolean group with identity~$e$. We add one more element to the set~$V$, which we denote by ~$s$, and define a commutative two-valued operation on the set $Y_V=V\cup\{s\}$ by the formulae
\begin{align*}
x*y&=[xy,xy],&&x,y\in V,\ x\ne y,\\
x*x&=[e,s],&&x\in V,\ x\ne e,\\
s*x&=[x,x],&&x\in V,\ x\ne e,\\
e*x&=[x,x],&&x\in V\cup\{s\},\\
s*s&=[e,e].&&
\end{align*}
It can be directly verified that the introduced operation is well defined, associative and transforms~$Y_V$ into an involutive commutative two-valued group.

If $\dim V=n<\infty$, that is, $V\cong C_2^n$, then the two-valued group
$$Y_n=Y_{C_2^n}$$ consists of the $2^n+1$ elements.

For $n=1$, the two-valued group~$Y_1$ consists of the three elements $e$, $s$ and~$x$ and has the multiplication table $s*s=[e,e]$, $s*x=[x,x]$, $x*x=[e,s]$. Thus, $Y_1\cong  X^{\ba}_{C_4} \cong X^{\bu}_{C_2}$.
Nevertheless, as we shall see below  (Proposition~\ref{propos_special_yn}), when $\dim V\ge 2$ the two-valued group~$Y_V$ is not coset and, in particular, it is not isomorphic to any of the two-valued groups of the principal and unipotent series. (One can also deduce that~$Y_V$ is not coset when $\dim V\ge 2$ from Proposition~3.1 in~\cite{BR-97}.)

Let us give one more interpretation of the two-valued group~$Y_V$. The set $V\setminus\{e\}$ can be naturally considered as the projectivization~$\mathbb{P}(V)$ of a vector space~$V$. Then $Y_V$ is the projective space~$\mathbb{P}(V)$ with two added elements~$e$ and~$s$. The multiplication rule is given as follows:
\begin{enumerate}
\item $e$ is the identity, i.e.  $e*x=[x,x]$ for all $x\in Y_V$,
\item $x*y=[z,z]$ if $x$ and~$y$ are two different points of the projective space~$\mathbb{P}(V)$ and $z$ is the third point on the line passing through~$x$ and~ $y$,
\item $x*x=[e,s]$ and $s*x=[x,x]$ if $x\in\mathbb{P}(V)$,
\item $s*s=[e,e]$.
\end{enumerate}

\subsection*{4. Product with Boolean group}
If $X$ is an involutive commutative two-valued group and $W$ is a Boolean group, then the direct product $X\times W$ is again endowed with the structure of an involutive commutative two-valued group, the multiplication in which is defined by the formula $(x_1,w_1)*(x_2,w_2)=(x_1*x_2,w_1w_2)$. More precisely, if $x_1*x_2=[z_1,z_2]$, then
$$
(x_1,w_1)*(x_2,w_2)=\bigl[(z_1,w_1w_2),(z_2,w_1w_2)\bigr].
$$
If we take involutive commutative groups of the principal series as $X$, then this construction does not lead to new examples, since for any abelian group~$A$ and any Boolean group~$W$ there is a canonical isomorphism
\begin{equation}\label{eq_iso_a}
X^{\ba}_A\times W\cong X^{\ba}_{A\times W}.
\end{equation}
However, if we take two-valued groups of the unipotent and special series as~$X$, then we obtain new examples of involutive two-valued groups.

\smallskip

We are now ready to state a classification theorem for finitely generated involutive commutative two-valued groups.

\begin{theorem}
\label{theorem_main_fg}
Every finitely generated involutive commutative two-valued group is isomorphic to one of the following two-valued groups:
\begin{itemize}
\item[\textnormal(1)] $X^{\ba}_{d_1,\ldots,d_k}$, where $2\le d_i\le\infty$ and $d_1|d_2|\cdots|d_k$;
\item[\textnormal(2)] $X^{\bu}_n\times C_2^m$, where $n\ge 1$, $m\ge 0$;
\item[\textnormal(3)] $Y_n\times C_2^m$, where $n\ge 1$, $m\ge 0$.
\end{itemize}
For each $m\ge 0$ there are isomorphisms
\begin{gather}
X^{\ba}_{m\times 2,4}\cong X^{\bu}_1\times C_2^m\cong Y_1\times C_2^m,\label{eq_main_iso1}\\
X^{\ba}_{m\times 2,4,4}\cong X^{\bu}_2\times C_2^m,\label{eq_main_iso2}
\end{gather}
where $m\times 2$ denotes the sequence $2,\ldots,2$ of $m$ twos. The indicated isomorphisms exhaust all isomorphisms between the listed two-valued groups.
\end{theorem}

\begin{example}\label{ex_iso}
Let us consider $10$-element two-valued groups~$X^{\ba}_{4,4}=C_4^2/\iota_{\ba}$ and~$X^{\bu}_2=(C_2^2\times C_2^2)/\iota_{\bu}$ and present an explicit isomorphism between them. Denote by~$a_1$ and~$a_2$ the generators of the factors in the group $C_4^2$ and denote by~$b_1$ and~$b_2$ the generators of the factors in the group $C_2^2$. Define the mapping $f\colon C_4^2\to C_2^2\times C_2^2$ by the formula
\begin{equation*}
f(a_1^ka_2^l)=\left(b_1^kb_2^l,b_1^{[k/2]}b_2^{[l/2]}\right),\qquad 0\le k\le 3,\ 0\le l\le 3,
\end{equation*}
where $[r]$ denotes the integer part of the number~$r$. The mapping~$f$ is not a homomorphism. However, it can be directly verified that for any elements $a,c\in C_4^2$ the element~$f(ac)$ coincides with one of the $4$ elements $f(a)f(c)$, $f(a)\iota_{\bu}\bigl(f(c)\bigr)$, $\iota_{\bu}\bigl(f(a)\bigr)f(c)$, and~$\iota_{\bu}\bigl(f(a)\bigr)\iota_{\bu}\bigl(f(c)\bigr)$; moreover, there is a coincidence of multisets
\begin{multline*}
\bigl[
f(ac), f(ac^{-1}), f(a^{-1}c), f(a^{-1}c^{-1})
\bigr]
={}\\
\left[
f(a)f(c), f(a)\iota_{\bu}\bigl(f(c)\bigr), \iota_{\bu}\bigl(f(a)\bigr)f(c),\iota_{\bu}\bigl(f(a)\bigr)\iota_{\bu}\bigl(f(c)\bigr)
\right].
\end{multline*}
Therefore, the map~$f$ induces a well-defined isomorphism $X^{\ba}_{4,4}\cong X^{\bu}_2$. Together with the isomorphism~\eqref{eq_iso_a}, this gives the isomorphism~\eqref{eq_main_iso2}.
\end{example}

\begin{example}\label{ex_Y2}
Consider the $5$-element group~$Y_2=Y_{C_2^2}$, which is the first group in the special series not contained in the principal series. Let $x$, $y$ and~$z$ be the three nontrivial elements in~$C_2^2$. Then the multiplication table in the two-valued group~$Y_2$ has the form
\begin{gather*}
x*x=y*y=z*z=[e,s],\quad s*s=[e,e],\\
\begin{aligned}
x*y&=[z,z],& y*z&=[x,x],& z*x&=[y,y],\\
s*x&=[x,x],& s*y&=[y,y],& s*z&=[z,z].
\end{aligned}
\end{gather*}
This two-valued group is related to the non-abelian $8$-element group of quaternions $$Q_8=\{\pm 1,\pm i,\pm j,\pm k\}\subset\mathbb{H},$$ where $\mathbb{H}$ is the algebra of quaternions. On the group $Q_8$ there is the conjugation anti-automorphism ~$\sigma$, $\sigma(gh)=\sigma(h)\sigma(g)$, leaving elements~$1$ and~$-1$ fixed and permuting $i\leftrightarrow -i$, $j\leftrightarrow -j$, $k\leftrightarrow -k$. The construction of a coset two-valued group by an involutive automorphism, as a rule, does not carry over to the case of an anti-automorphism. Nevertheless, the anti-automorphism~$\sigma$ of the group~$Q_8$ has the following special property: for any two elements $g,h\in Q_8$ the four-element multiset $[gh,g\sigma(h),\sigma(g)h,\sigma(g)\sigma(h)]$ is a doubling of some two-element multiset. This property easily implies that the $5$-element quotient set $Q_8/\sigma$ has a well-defined structure of a two-valued group. Since the multiplication of quaternion units is commutative up to sign, the two-valued group~$Q_8/\sigma$ is commutative. It is easy to see that it is isomorphic to the two-valued group~$Y_2$ with the isomorphism $1\mapsto e$, $-1\mapsto s$, $\pm i\mapsto x$, $\pm j\mapsto y$ and~$\pm k\mapsto z$.
\end{example}

\begin{example}\label{ex_Y3}
Similarly to the previous example, there is a connection between the $9$-element two-valued group~$Y_3$ and the $16$-element set $$O_{16}=\{\pm 1,\pm e_0,\pm e_1,\ldots,\pm e_6\}\subset\mathbb{O},$$ where $\mathbb{O}$ is the algebra of octonions and $1,e_0,\ldots,e_6$ is its standard basis. Since the multiplication of octonions is not associative, the set~$O_{16}$ is not even a group, but only a so-called \textit{Moufang loop}, see ~\cite[\S 7.1]{ CoSm09}. However, conjugation defines an involutive anti-automorphism of this algebraic structure. It can be directly verified that again for any two elements $g,h\in O_{16}$ the four-element multiset $[gh,g\sigma(h),\sigma(g)h,\sigma(g)\sigma(h)]$ is a doubling of some two-element multiset. This again implies that on the $9$-element quotient set $O_{16}/\sigma$ there is a well-defined two-valued multiplication. Although the multiplication in $O_{16}$ is neither associative nor commutative, it is associative and commutative up to sign. Using this, one can directly check that the resulting two-valued multiplication by $O_{16}/\sigma$ is associative and commutative and defines the structure of a commutative two-valued group on this set. Again, it can be directly verified that this two-valued group is isomorphic to the two-valued group~$Y_3$.
\end{example}

The work is organized as follows. In section ~\ref{section_1gen} we classify single-generated involutive two-valued groups. In particular, we show that all such two-valued groups are commutative. Sections ~\ref{section_ord2}--\ref{section_final_proof} are devoted to the classification of involutive commutative two-valued groups in the finitely generated case. Sections~\ref{section_non_fg}--\ref{section_lc} contain a number of classification results in the topological case. In section ~\ref{section_alg} the case of algebraic two-valued groups is considered. In the final section~\ref{section_conclusion} we discuss some open questions.

\section{Single-generated involutive two-valued groups}\label{section_1gen}

The classification problem of single-generated usual groups is very simple: the answer is given by the infinite and all possible finite cyclic groups. In the case of $n$-valued groups, the problem of classifying single-generated groups (even commutative two-valued groups) turned out to be difficult and is currently far from a complete solution.

A rich source of single-generated commutative $n$-valued groups is given by multivalued multiplication on the set of irreducible representations of a finite group, see ~\cite[ Sections ~10,~11]{ Buc06}. More precisely, the following theorem holds: \textit{If a finite group $G$ has a faithful irreducible representation~$\rho$, then the multivalued group on the set of its irreducible representations is single-generated}. The proof is based on Burnside's theorem (see ~\cite{CuRe62}), which states that in this case any irreducible representation of the group~$G$ is contained in a decomposition of some tensor power of the representation~$\rho$.

Note also that the problem of single-generated multivalued groups is closely related to the problem of integrability of multivalued dynamical systems with discrete time, see ~\cite{Ves91}, ~\cite{BuVe96}. Nevertheless, in the class of involutive two-valued groups considered in this paper, the problem of classifying single-generated groups can be completely solved even without the commutativity assumption. The answer turns out to be as simple as for the usual groups (the case of free single-generated two-valued groups was considered in \cite{BuVe19} in connection with the Conway topograph).

\begin{theorem}\label{theorem_1gen}
Any single-generated involutive two-valued group is isomorphic to a coset group of the form $X^{\ba}_C=C/\iota_{\ba}$, where $C$ is a cyclic group and $\iota_{\ba}\colon C\to C$ is the antipodal involution. In particular, every such two-valued group is commutative.
\end{theorem}

The proof will be divided into several Lemmas, which will be useful to us in their own right.

\begin{lem}\label{lem_basic}
Let $X$ be an involutive two-valued group and elements $x,y,z\in X$ be such that $z$ belongs to the multiset~$x*y$. Then $x$ belongs to the multiset~$y*z$.
\end{lem}

\begin{proof}
Let $x*y=[z,z']$. Then, due to the involutivity of the two-valued group~$X$, the multiset $(x*y)*z=[z*z,z'*z]$ contains the identity~$e$. On the other hand, due to the associativity, this multiset can be written as $x*(y*z)$. From the fact that it contains~$e$ and involutivity of the two-valued group~$X$, it follows that $y*z$ contains~$x$.
\end{proof}

\begin{lem}\label{lem_sequence}
Let $X$ be an involutive two-valued group. Then for any element $x\in X$ there exists a unique sequence $x^0,x^1,x^2,\ldots$ of elements of the two-valued group~$X$ such that $x^0=e$, $x^1=x$ and
\begin{equation}\label{eq_1}
x*x^k=x^k*x=[x^{k-1},x^{k+1}]
\end{equation}
for all~$k\ge 1$.
\end{lem}

\begin{proof}
Let us show that the formula ~\eqref{eq_1} together with the initial conditions $x^0=e$, $x^1=x$ defines a well-defined recurrent rule for constructing uniquely the sequence $\{x^k\}_{k\ge 0}$. Assume that the elements of $x^0,x^1,\ldots,x^n$, where $n\ge 1$, are already defined such that equalities ~\eqref{eq_1} hold for $1\le k\le n-1$, and show that~$x^{n+1}$ can be uniquely defined such that equalities~\eqref{eq_1} hold for $k=n$.

From the equalities ~\eqref{eq_1} under $1\le k\le n-1$, in particular, it follows that the element~$x$ commutes with each of the elements of $x^0,x^1,\ldots,x^{n-1}$. Let us show that $x$ commutes with~$x^n$. This is obvious for $n=1$, so suppose $n\ge 2$. We have
\begin{align*}
x*x*x^{n-1}&=x*[x^{n-2},x^n]=[x*x^{n-2},x*x^n],\\
x^{n-1}*x*x&=[x^{n-2},x^n]*x=[x^{n-2}*x,x^n*x].
\end{align*}
Since $x$ commutes with~$x^{n-1}$, then $x*x*x^{n-1}=x^{n-1}*x*x$. Also, $x*x^{n-2}=x^{n-2}*x$. Hence $x*x^n=x^n*x$.

Now, applying Lemma~\ref{lem_basic} for the triple $(x^{n-1},x,x^n)$, we get that the element $x^{n-1}$ lies in the multiset~$x*x^n=x^n*x$. Therefore, we can and should take as~$x^{n+1}$ the second element of this multiset.
\end{proof}

For $k<0$, by definition, we set $x^k=x^{-k}$. Then the equality ~\eqref{eq_1} will be satisfied for all integers~$k$. It is natural to call the element~$x^k$ as \textit{$k$-th power} of the element~$x$.

\begin{lem}\label{lem_sequence_mult}
Let $X$ be an involutive two-valued group, $x$ be its arbitrary element. Then
\begin{align}\label{eq_mult1}
x^k*x^l&=[x^{k+l},x^{k-l}],\\
\label{eq_mult2}
(x^k)^l&=x^{kl}
\end{align}
for all integers~$k$ and~$l$.
\end{lem}

\begin{proof}
Due to $x^n=x^{-n}$, it suffices to prove the equality ~\eqref{eq_mult1} for non-negative~$k$ and~$l$. Let us prove it when $0\le k\le l$; case $k\ge l$ is completely analogous. Since $x^0=e$, ~\eqref{eq_mult1} is true for $k=0$. For $k=1$ this equality becomes formula~\eqref{eq_1}, which was used to construct the sequence $\{x^k\}$; thus it is also true. Let us prove the equality ~\eqref{eq_mult1} with $2\le k\le l$ by induction on~$k$. By the inductive hypothesis
$$
x^{k-1}*x^l=[x^{k+l-1},x^{k-l-1}].
$$
Multiplying both sides of this equation from the left by~$x$ and applying the formula~\eqref{eq_1}, we get
$$
x*x^{k-1}*x^l=[x^{k+l},x^{k+l-2},x^{k-l},x^{k-l-2}].
$$
On the other hand, using the inductive assumption and formula ~\eqref{eq_1} again, we obtain
$$
x*x^{k-1}*x^l=[x^{k-2}*x^l,x^k*x^l] = [x^{k+l-2},x^{k-l-2},x^k*x^l],
$$
whence it follows that $x^k*x^l = [x^{k+l},x^{k-l}]$.

From the formula ~\eqref{eq_mult1}, in particular, it follows that the sequence of elements $x^{kl}$, where $l = 0,1,2,\ldots$, satisfies the recurrence relation
$$
x^k*x^{kl}=[x^{k(l-1)},x^{k(l+1)}]
$$
and hence $(x^{k})^l = x^{kl}$.
\end{proof}

\begin{cor}
Every single-generated involutive two-valued group is commutative.
\end{cor}

\begin{lem}\label{lem_sequence_null}
Let $X$ be an involutive two-valued group and $x$ be its arbitrary element. Then the subset $K\subseteq\Z$ consisting of all integers~$k$ such that $x^{k}=e$ is a subgroup. If $k\in K$, then $x^{\pm l+k}=x^{l}$ for all $l\in\Z$. Conversely, if $x^l=x^m$, then either $l-m\in K$ or $l+m\in K$.
\end{lem}

\begin{proof}
Obviously, $0\in K$ and $-k\in K$ are for all $k\in K$. Since $e*x^l=[x^l,x^l]$, it immediately follows from the formula ~\eqref{eq_mult1} that $x^{\pm l+k}=x^{l}$ for all $k\in K$ and $l\in \Z$. In particular, if $l$ also lies in~$K$, we get that $x^{k+l}=e$, so $k+l\in K$. Thus $K\subseteq\Z$ is a subgroup.

Suppose now that $x^l=x^m$. By involutivity, the multiset $x^l*x^m=[x^{l+m},x^{l-m}]$ contains~$e$. Hence, either $x^{l+m}=e$ or $x^{l-m}=e$, so either $l+m\in K$ or $l-m\in K$.
\end{proof}

The smallest positive integer~$k$ such that $x^k=e$ will be called \textit{order} of element~$x$ and denoted by~$\ord x$. If $x^k\ne e$ for all positive integers~$k$, then we will say that the order of element~$x$ is equal to infinity. It is obvious that $\ord x=1$ if and only if $x=e$. In what follows, the elements of orders~$2$ and~$3$ will be of particular interest. It easily follows from the definition and Lemma ~\ref{lem_sequence_null} that a non-identity element~$x$ has order~$2$ if and only if $x*x=[e,e]$, and has order~$3$ if and only if $x*x=[e,x]$. In particular, the elements of orders $1$ and ~$2$ are exactly \textit{strong involutions}, see the definition~\ref{defin_invol}.

The following Corollary follows from Lemma ~\ref{lem_sequence_null}.

\begin{cor}
If $x$ is an element of an involutive two-valued group, then $x^l=x^m$ if and only if at least one of the numbers $l-m$ and~$l+m$ is divisible by $\ord x$. In particular, $x^l=e$ then and only when $l$ is divisible by $\ord x$.
\end{cor}

\begin{proof}[Proof of Theorem~\ref{theorem_1gen}]
Let $X$ be a single-generated involutive two-valued group with generator~$x$. According to Lemma~\ref{lem_sequence_mult}, for any $k,l\ge 0$, the equality~\eqref{eq_mult1} holds. A subset of the set~$X$ consisting of all powers~$x^k$ is closed under multiplication and contains~$x.$ Since the two-valued group~$X$ is generated by the element~$x$, we get that the powers $x^k$, $k\in\Z$ exhaust all elements of~$X$.

According to Lemma~\ref{lem_sequence_null}, the subset $K\subseteq\Z$ consisting of all integers~$k$ such that $x^k=e$ is a subgroup, and $x^{\pm l+k}=x^{l}$ for all $k\in K$ and $l\in\Z$. Therefore, the correspondence $k\mapsto x^k$ gives a well-defined surjective map $\varphi\colon(\Z/K)/\iota \to X$. From formula ~\eqref{eq_mult1} it follows that the map~$\varphi$ is a homomorphism of two-valued groups, and from the last statement of Lemma ~\ref{lem_sequence_null} it follows that the map~$\varphi$ is injective. Thus, $\varphi$ is an isomorphism of two-valued groups.
\end{proof}

\section{Subgroup of the elements of order 2}\label{section_ord2}

Let $X$ be an involutive commutative two-valued group.
Let us introduce a partially defined single-valued operation of multiplication~$\cdot$ on $X$ such that
\begin{enumerate}
\item $x\cdot y=z$ if $x*y=[z,z]$
\item The expression $x\cdot y$ is undefined if the multiset $x*y$ consists of two different elements.
\end{enumerate}
It immediately follows from the definition that $x\cdot y =y\cdot x$ and $x\cdot e=x$. Also, $x\cdot x=e$ if $\ord x\le 2$, and $x\cdot x$ is undefined if $\ord x> 2$.

\begin{lem}\label{lem_order2}
If at least one of the two elements $x$ and $y$ has order~$2$, then the multiset $x*y$ consists of two identical elements, and thus the expression $x\cdot y$ is defined.
\end{lem}

\begin{proof}
Let $\ord x=2$ and $z\in x*y$. By Lemma ~\ref{lem_basic}, we have $y\in x*z$, that is, $x*z=[y,y']$ for some~$y'$. Then
\begin{align*}
(x*x)*z&=[e,e]*z=[z,z,z,z],\\
x*(x*z)&=x*[y,y']=[x*y,x*y'],
\end{align*}
so $x*y=[z,z]$.
\end{proof}

\begin{propos}\label{propos_bullet}
Let $X$ be an involutive commutative two-valued group and $V\subseteq X$ be its subset consisting of the identity~$e$ and all the elements of order~$2$. Then
\begin{itemize}
\item[\textnormal{(a)}] the set $V$ is a Boolean group with respect to the operation~$\cdot$\,;
\item[\textnormal{(b)}] operation $\cdot$ defines the action of the group~$V$ on the set $X$; moreover, if $x*y=[z_1,z_2]$, then
\begin{equation}\label{eq_action}
(v\cdot x)*y=[v\cdot z_1,v\cdot z_2]
\end{equation}
for all $v\in V$;
\item[\textnormal{(c)}] if $v\in V\setminus\{e\}$ and $x\in X$, then $v\cdot x=x$ if and only if $x^2=v$; in particular, the equality $v\cdot x=x$ can take place only if $\ord x=4$.
\end{itemize}
\end{propos}

\begin{proof}
Statements~(a) and~(b) immediately follow from Lemma ~\ref{lem_order2}, as well as the associativity, commutativity and involutivity of the two-valued group~$X$.

Let us prove statement~(c). If $x^2=v$, then $\ord x=4$, then $x^3=x$, hence $v*x=x^2*x=[x^3,x]=[x,x]$, so $v\cdot x=x$.

Conversely, suppose $v\cdot x=x$, that is, $v*x=[x,x]$. Then it follows from Lemma~\ref{lem_basic} that the multiset~$x*x$ contains the element~$v$. However, the fact that the group~$X$ is involutive implies that this multiset also contains the element~$e$. Since $v\ne e$, we get that $x*x=[e,v]$, that is, $x^2=v$.
\end{proof}

The two-valued group~$X$ splits into orbits with respect to the action~$\cdot$ of the group~$V$ on it. These orbits we will call \textit{$V$-orbits}; the $V$-orbit containing the element~$x$ will be denoted by~$Vx$.

\begin{remark}
In what follows, we will usually write $xy$ instead of ~$x\cdot y$, which should not lead to confusion.
\end{remark}

\section{Special and non-special two-valued groups}

\begin{defin}
A pair $(x,y)$ consisting of elements of an involutive commutative two-valued group will be called \textit{special} if $\ord x>2$, $\ord y>2$, but the multiset $x*y$ consists of two identical elements. An involutive commutative two-valued group will be called \textit{special} if it contains at least one special pair of elements, and \textit{non-special} otherwise.
\end{defin}

\begin{remark}
In ~\cite{BuVe19}, non-special two-valued groups were called \text{strongly two-valued}.
\end{remark}

Thus, if a two-valued group is non-special, then the value~$x\cdot y$ is defined if and only if one of the elements~$x$ and~$y$ has order~$\le 2$.

\begin{propos}\label{propos_special_yn}
All two-valued groups of the form $X_A^{\ba}$ and $X_V^{\bu}\times W$ are non-special.  Two-valued groups ~$Y_V\times W$ are special when $\dim V\ge 2$.
\end{propos}

\begin{proof}
Let us prove the non-speciality of the two-valued group~$X_A^{\ba}$. Let $\pi\colon A\to X_A^{\ba}$ be the quotient map by antipodal involution. Consider elements~$x,y\in X_A^{\ba}$. If $a,b\in A$ are elements such that $x=\pi(a)$ and~$y=\pi(b)$, then $x*y=[\pi(ab),\pi(ab^{-1})]$. Assume that the orders of elements ~$x$ and~$y$ are greater than~$2$. Then the orders of the elements~$a$ and~$b$ are also greater than~$2$. Hence, the elements $ab$ and~$ab^{-1}$ are not equal and not inverse to each other. Therefore, the multiset $x*y$ consists of two different elements. Thus, the two-valued group~$X_A^{\ba}$ is non-special.

Let us prove the non-speciality of the two-valued group $X_V^{\bu}=(V\times V)/\iota_{\bu}$, where $\iota_{\bu}(a,b)=(a,ab)$. Let $\pi\colon V\times V\to X_V^{\bu}$ be the quotient map by involution~$\iota_{\bu}$. It can be directly verified that the order of element $\pi(a,b)$ of a two-valued group~$X_V^{\bu}$ is equal to~$4$ if $a\ne e$, and is equal to~$2$, if $a=e$ and ~$b\ne e$. Thus, any pair of elements of orders greater than~$2$ has the form $x=\pi(a_1,b_1)$, $y=\pi(a_2,b_2)$, where $a_1\ne e$ and $a_2\ne e$. Then $x*y=[\pi(a_1a_2,b_1b_2),\pi(a_1a_2,a_1b_1b_2)]$. Since the element $a_1b_1b_2$ is neither~$b_1b_2$ nor $a_1a_2b_1b_2$, then the elements of the multiset~$x*y$ are different. Thus, the two-valued group~$X_V^{\bu}$ is non-special. It is easy to see that this implies the non-speciality of all two-valued groups~$X_V^{\bu}\times W.$

Now consider the two-valued group~$Y_V$, where $\dim V\ge 2$. Let $x,y\in V$ be two distinct non-identity elements. Then the orders of the elements~$x$ and~$y$ in the two-valued group~$Y_V$ are equal to~$4$, since $x*x=y*y=[e,s]$ and $s$ is an element of order $2$. Moreover, $x*y=[xy,xy]$, hence $(x,y)$ is a special pair. Thus, the two-valued group~$Y_V$ is special, and hence all two-valued groups~$Y_V\times W$ are special.
\end{proof}

\begin{propos}\label{propos_no_coset}
Let $X$ be a special involutive commutative two-valued group. Then $X$ is not coset, that is, $X$ is not isomorphic to any two-valued group of the form $G/\iota$, where $G$ is usual group and $\iota$ is an automorphism of ~$G$ such that $\iota^2=\mathrm{id}$.
\end{propos}

\begin{proof}
Assume the opposite: let $X\cong G/\iota$. Denote by $\pi\colon G\to X$ the quotient map by automorphism~$\iota$. Let $(x,y)$ be a special pair in~$X$ and let $g,h\in G$ be elements such that $\pi(g)=x$ and $\pi(h)=y$.

We have $x*x=\bigl[\pi(g^2),\pi\bigl(g\iota(g)\bigr)\bigr]$. Since $\ord x>2$, the multiset $x*x$ must consist of two different elements, so it follows that $g^2\ne g\iota(g)$ and hence $\iota(g)\ne g$. Similarly, $\iota(h)\ne h$.

A multiset $x*y = \bigl[\pi(gh),\pi\bigl(g\iota(h)\bigr)\bigr]$ consists of two identical elements. Hence, either $g\iota(h)=gh$ or $g\iota(h)=\iota(gh)=\iota(g)\iota(h)$, which implies that either $\iota(h)=h$ or $\iota(g)=g$. As it has been proven, neither of these is possible, which completes the proof.
\end{proof}

\begin{remark}
Note that, generally speaking, commutative two-valued groups can be obtained from non-commutative single-valued groups using the coset construction. For example, as shown in ~\cite{Buc06} (see also ~\cite{BuVe19}), the two-valued Buchstaber--Novikov group $X^{\ba}_{C_{\infty}}=C_{\infty}/\iota_{\ba}$ is isomorphic to the coset group $(C_2*C_2)/\sigma$, where $*$ denotes the free product and $\sigma$ is an automorphism, interchanging the generators of factors~$C_2$.We emphasize that Proposition~\ref{propos_no_coset} asserts that a special involutive commutative two-valued group cannot be a coset two-valued group of any (not necessarily commutative) group~$G$ by its automorphism. Nevertheless, we recall that a special two-valued group~$Y_2$ (respectively, $Y_3$) can be obtained using an analogue of the coset construction for the anti-automorphism of a non-commutative group~$Q_8$ (respectively, a non-commutative and non-associative Moufang loop~$O_{16}$), see examples~\ref{ex_Y2} and ~\ref{ex_Y3}.
\end{remark}

\section{Quotient groups of two-valued groups}

The general question of the existence of quotient groups of two-valued groups requires a separate study. We give a positive answer to this question in our case of the involutive commutative two-valued groups.

\subsection{Existence of quotient groups}
A subset $Y$ of an involutive commutative two-valued group~$X$ will be called \textit{subgroup} if for any two elements of $y_1,y_2\in Y$ both elements of the multiset~$y_1*y_2$ lie in~$Y$.

To the subgroup $Y$ we assign an equivalence relation~$\sim$ on the two-valued group~$X$ defined by the rule: $x\sim x'$ if and only if there exists an element $y\in Y$ such that $x'\in x*y$.

\begin{lem}\label{lem_equiv}
\textnormal{(a)} The introduced relation~$\sim$ is indeed an equivalence relation.

\textnormal{(b)} $x\sim e$ if and only if $x\in Y$.

\textnormal{(c)} If $x_1*x_2=[z_1,z_2]$, $x_1'*x_2'=[z_1',z_2']$, $x_1\sim x_1'$ and~$x_2\sim x_2'$, then $[z_1,z_2]\sim [z_1',z_2']$, that is, either $z_1\sim z_1'$ and~$z_2\sim z_2'$, or $z_1\sim z_2'$ and~$z_2\sim z_1'$.
\end{lem}

\begin{proof}
The relation~$\sim$ is reflexive since $x\in x*e$ is for all $x$. The symmetry of the relation~$\sim$ immediately follows from Lemma~\ref{lem_basic}. Let us prove its transitivity. Let $x_1\sim x_2$ and $x_2\sim x_3$. Then there are elements of $y_1, y_2\in Y$ such that $x_2\in x_1*y_1$ and $x_3\in x_2*y_2$, which immediately implies that $x_3\in x_1*y_1*y_2$. On the other hand, $y_1*y_2=[y_3,y_4]$ for some elements of $y_3,y_4\in Y$, so
$$
x_1*y_1*y_2=[x_1*y_3,x_1*y_4].
$$
Therefore, the element~$x_3$ belongs to one of the multisets~$x_1*y_3$ and $x_1*y_4$. Thus, $x_1\sim x_3$, which completes the proof that~$\sim$ is an equivalence relation.

If $x\in Y$, then $e\in x*x$ implies $x\sim e$. Conversely, if $x\sim e$, then there is an element $y\in Y$ such that $e\in x*y$, so from involutivity it follows that $x=y\in Y$.

Due to the commutativity of the two-valued group~$X$ and the fact that $\sim$ is an equivalence relation, it suffices to prove the statement~(c) in the case of $x_2=x_2'$. If $x_1\sim x_1'$ then there exists $y\in Y$ such that $x_1'\in x_1*y$. Then $x_1*y=[x_1',p]$ for some~$p$ and we have
$$
[z_1*y, z_2*y]=(x_1*x_2)*y=(x_1*y)*x_2=[x_1'*x_2,p*x_2]=[z_1',z_2',p*x_2].
$$
Swapping, if necessary, the elements~$z_1$ and~$z_2$, we can assume that $z_1'$ belongs to the multiset~$z_1*y$; then $z_1\sim z_1'$. If, in addition, the element~$z_2'$ lies in the multiset~$z_2*y$, then $z_2\sim z_2'$ and, thus, the claim follows. It remains for us to consider the case when $z_1*y=[z_1',z_2']$. Then $z_1'\sim z_1\sim z_2'$. Lemma ~\ref{lem_basic} implies that $x_1'*y=[x_1,q]$ for some~$q$. Then
$$
[z_1'*y, z_2'*y]=(x_1'*x_2)*y=(x_1'*y)*x_2=[x_1*x_2,q*x_2]=[z_1,z_2,q*x_2],
$$
which implies that $z_2$ belongs to one of the multisets~$z_1'*y$ and~$z_2'*y$. Hence, all four elements~$z_1$, $z_2$, $z_1'$, and~$z_2'$ are pairwise equivalent to each other, which implies the claim.
\end{proof}

Lemma ~\ref{lem_equiv} implies that
the set $X/Y=X/{\sim}$ is well defined and
two-valued multiplication in the group~$X$ induces a well-defined two-valued multiplication on the set~$X/Y$.

\begin{lem}
The introduced multiplication turns the set~$X/Y$ into an involutive commutative two-valued group.
\end{lem}
\begin{proof}
The commutativity and associativity of two-valued multiplication in~$X/Y$ immediately follows from the commutativity and associativity of multiplication in~$X$. It is also obvious that the equivalence class of identity, which, according to statement~(b) of Lemma~\ref{lem_equiv}, coincides with the subgroup~$Y$, is the (strong) identity element of the introduced two-valued multiplication; we will also denote this identity by~$e$. 

Let us prove the involutivity (and at the same time the existence and uniqueness of the inverse). The fact that $e\in x*x$ for all $x\in X$ immediately implies that $e\in z*z$ for all $z\in X/Y$. Therefore, it is only necessary to prove that if $z_1$ and~$z_2$ are two different elements from~$X/Y$, then $e\notin z_1*z_2$. To do this, we need to show that if $x_1$ and~$x_2$ are representatives of two different equivalence classes~$z_1$ and~$z_2$, then the multiset $x_1*x_2$ does not contain elements from~$Y$. Assume that this is not the case, i.e. the element $y\in Y$ lies in the multiset~$x_1*x_2$. Then, by the Lemma ~\ref{lem_basic}, $x_2\in x_1*y$, it means that $x_1\sim x_2$, and hence $z_1=z_2$, which is not true. The contradiction completes the proof of the Lemma.
\end{proof}

The involutive commutative two-valued group~$X/Y$ defined in this way will be called \textit{quotient group} of the two-valued group~$X$ by its subgroup~$Y$. It follows from part~(b) of Lemma ~\ref{lem_equiv} that the kernel of the natural projection $\pi\colon X\to X/Y$ indeed coincides with the subgroup~$Y$. (The \textit{kernel} of a homomorphism of two-valued groups, as in the case of the ordinary groups, is defined as the preimage of the identity.)

\begin{remark}
Let $V\subseteq X$ be a subset consisting of all elements of order~$\le 2$. By Proposition~\ref{propos_bullet}, the set $V$ is a Boolean group, which acts on the two-valued group $X$. Hence if $W\subseteq V$ is a subgroup, then the notation $X/W$ can be given two meanings: (i) the quotient group of a two-valued group~$X$ by its subgroup~$W$; (ii) the space of orbits of the action of the group~$W$ on~$X$, where $W$ is considered as a single-valued group with respect to the operation~$\cdot$\, introduced in section ~\ref{section_ord2}. It is easy to see, however, that these two sets coincide.
\end{remark}

\subsection{Homomorphism Theorem}
The fundamental theorem on the image of a homomorphism carries over to the case of involutive commutative two-valued groups.

\begin{propos}
Let $f\colon X\to Z$ be a homomorphism of involutive commutative two-valued groups. Then its kernel~$\ker f$ and image $\im f$ are subgroups of two-valued groups~$X$ and~$Z$, respectively, and $f$ induces an isomorphism $X/\ker f\cong\im f$.
\end{propos}

\begin{proof}
Obviously, $\im f$ is a subgroup. If $y_1,y_2\in\ker f$ and $y_1*y_2=[z_1,z_2]$, then
$$
[f(z_1),f(z_2)]=f(y_1)*f(y_2)=e*e=[e,e],
$$
so $f(z_1)=f(z_2)=e$. Hence, $\ker f$ is also a subgroup. Let~$\sim$ be the defined above equivalence relation on~$X$ corresponding to this subgroup. To complete the proof of the Proposition, it suffices to prove that the inverse images of the elements $z\in \im f$ are precisely the equivalence classes with respect to~$\sim$. Let us prove this. If $x_1\sim x_2$, then there exists $y\in\ker f$ such that $x_2\in x_1*y$. Then $$f(x_2)\in f(x_1)*f(y)=f(x_1)*e=[f(x_1),f(x_1)],$$ so $f(x_2)=f(x_1)$. Conversely, let $x_1$ and~$x_2$ be two elements of~$X$ such that $f(x_1)=f(x_2)$. Let $x_1*x_2=[p,q]$. Then the multiset $[f(p),f(q)]=f(x_1)*f(x_2)=f(x_1)*f(x_1)$ must contain the identity due to involutivity. Hence, at least one of the elements~$p$ and~$q$ lies in~$\ker f$. Without loss of generality, we can assume that $p\in\ker f$. It follows from Lemma~\ref{lem_basic} that the element~$x_2$ lies in the multiset~$x_1*p$, hence $x_1\sim x_2$.
\end{proof}

The proved Proposition allows us to define \textit{exact sequences} of involutive commutative two-valued groups and work with them in exactly the same way as in the case of ordinary abelian groups.

\subsection{$C_2$-extensions of two-valued groups}

\begin{defin}
(\textit{Involutive commutative}) \textit{$C_2$-extension} of an involutive commutative two-valued group~$X$ is the exact sequence of involutive commutative two-valued groups
\begin{equation}\label{eq_ext}
1\rightarrow 2C_2\rightarrow \hX \xrightarrow{\pi} X\rightarrow 1,
\end{equation}
where $2C_2$ is the two-element two-valued group obtained by doubling the operation in the group~$C_2$.
\end{defin}

We could also study extensions in which the resulting group~$\hX$ is not involutive or commutative, but they will be not of our interest. In what follows, a $C_2$-extension is always understood as an involutive commutative $C_2$-extension.

The image of a two-valued group~$2C_2$ in~$\hX$ is a two-element subgroup, one of which elements is identity~$e$; its second element, which has order~$2$, we will denote by~$u$. It is easy to see that $X$ is obtained from~$\hX$ by taking quotient by the action of the group $C_2=\{e,u\}$ given by the operation described in section ~\ref{section_ord2}.

We will denote the identities of both groups~$X$ and~$\hX$ by~$e$, this should not not lead to confusion.

\begin{lem}
\textnormal{(a)} If $x\in X$ is an element of order not equal to~$2$, then its preimage~$\pi^{-1}(x)$ consists of two elements, the order of each of which is either~$\ord x$, or~$2\ord x.$ 

\textnormal{(b)} If $x\in X$ is an element of order~$2$, then its preimage~$\pi^{-1}(x)$ consists of either two elements of order~$2$, or one element of order~$4.$
\end{lem}

\begin{proof}
First, note that $\pi^{-1}(e)=\{e,u\}$, so the claim of the Lemma is true for $x=e$.

If $\hx$ is an element in~$\pi^{-1}(x)$, then $\pi^{-1}(x)=\{\hx,u\hx\}$, so we only need to figure out what the orders of the elements~$\hx$ and $u\hx$ can be and when these elements can coincide with each other.

It is obvious that $\pi\bigl(\hx^k\bigr)=x^k$ for all $k\in\Z$. So if $\hx^k=e$, then $x^k=e$. Therefore, $\ord\hx$ is divisible by $\ord x$. Now take $k=\ord x$. Then $x^k=e$, which means that the element~$\hx^k$ belongs to the set $\pi^{-1}(e)=\{e,u\}$, that is, either $\hx^k=e$ or $\hx^k=u$. In the first case $\ord \hx=k$. In the second case $\ord \hx>k$, but at the same time $\hx^{2k}=u^2=e$, so $\ord\hx=2k$.

By Proposition ~\ref{propos_bullet}(c), equality $u\hx=\hx$ holds if and only if $\hx^2=u$. If so, then $x^2=e$, then $\ord x=2$, and also $\ord\hx=4$. In this case, the preimage~$\pi^{-1}(x)$ consists of one element~$\hx$. In all other cases $u\hx\ne\hx$, that is, the preimage~$\pi^{-1}(x)$ consists of two different elements~$\hx$ and~$u\hx$. It remains to note that if $\ord\hx=\ord x= 2$, then according to Proposition ~\ref{propos_bullet} $\ord(u\hx)=2$; thus~$\pi^{-1}(x)$ consists of two elements of order~$2$.
\end{proof}

Recall that, according to Proposition~\ref{propos_bullet}(a), the subset $V\subseteq X$ consisting of the identity~$e$ and all elements of order~$2$ is a Boolean group. Denote by $R\subset V$ the subset consisting of all elements~$x$ of order~$2$ such that the preimage~$\pi^{-1}(x)$ consists of one element of order~$4$. We will call ~$R$ \textit{branching set} of the extensions~\eqref{eq_ext}. Note that the identity~$e$ never belongs to the set~$R$. Moreover, the following statement is true.

\begin{lem}\label{lem_ram1}
The set $V\setminus R$ is a subgroup of the group~$V$.
\end{lem}

\begin{proof}
Let $x,y\in V\setminus R$ and~$\hx$ and~$\hy$ be some elements from preimages~$\pi^{-1}(x)$ and~$\pi^{-1}(y)$ respectively. Then $\ord\hx=\ord\hy=2$. By Proposition ~\ref{propos_bullet}(a) we obtain that $\ord(\hx\hy)=2$. Since $\hx\hy\in \pi^{-1}(xy)$, it follows that $x y\in V\setminus R$. Thus, $V\setminus R$ is a subgroup.
\end{proof}

\section{Direct product decompositions}

Recall that in section ~\ref{section_formulation} we defined the direct product operation of an involutive commutative two-valued group~$X$ and a Boolean group~$W$; the operation of multiplication in the resulting involutive commutative group~$X\times W$ is given by the formula
$$
(x_1,w_1)*(x_2,w_2)=\bigl[(z_1,w_1w_2),(z_2,w_1w_2)\bigr],
$$
where $x_1*x_2=[z_1,z_2]$.

\begin{defin}
Let us say that an involutive commutative two-valued group~$X$ \textit{contains single-valued direct factor} if $X$ is isomorphic to a two-valued group of the form $X'\times W$, where $X'$ is an involutive commutative two-valued group and $W$ is a non-trivial Boolean group. Otherwise, we say that $X$ does not contain single-valued direct factor.
\end{defin}

We want to reduce the classification problem for finitely generated involutive commutative two-valued groups to the same problem for two-valued groups that do not contain single-valued direct factors. To do this, we will need the following two Propositions, the first of which
gives a convenient criterion when a two-valued group does not contain a single-valued direct factor, and the second one claims that the largest single-valued direct factor in a two-valued group is determined in a unique way up to isomorphism.

\begin{propos}\label{propos_neras}
A finitely generated involutive commutative two-valued group~$X$ does not contain a single-valued direct factor if and only if every element $v\in X$ of order~$2$ is a square of some element from~$X$.
\end{propos}

\begin{propos}\label{propos_rasch_iso}
Any finitely generated involutive commutative two-valued group~$X$ can be decomposed into a direct product of a finitely generated involutive commutative two-valued group without single-valued direct factors and a finite Boolean group. If $X=X_i\times W_i$, where $i=1,2$, are two such decompositions, then $X_1\cong X_2$ and~$W_1\cong W_2$.
\end{propos}

Before proving Propositions ~\ref{propos_neras} and ~\ref{propos_rasch_iso}, we establish two useful properties of involutive commutative two-valued groups.

\begin{lem}\label{lem_z1z2_prelim}
Let $X$ be an involutive commutative two-valued group, $x,y\in X$ and $x*y=[z_1,z_2]$. Assume that $x^2\ne y^2$ and $z_1\ne z_2$, then $z_1*z_2=[x^2,y^2]$.
\end{lem}

\begin{proof}
By Lemma~\ref{lem_basic}, the element~$y$ belongs to each of the multisets~$x*z_1$ and~$x*z_2$. Therefore, the identity ~$e$ belongs to the multiset
$$
(x*z_1)*(x*z_2)=(x*x)*(z_1*z_2)=[e,x^2]*(z_1*z_2)=[z_1*z_2,z_1*z_2,x^2*(z_1*z_2)].
$$
However, $e\notin z_1*z_2$ since $z_1\ne z_2$. Hence, $e\in x^2*(z_1*z_2)$, which implies that $x^2\in z_1*z_2$. Similarly we have that $y^2\in z_1*z_2$. Since $x^2\ne y^2$, this implies the Lemma.
\end{proof}

\begin{lem}\label{lem_squares}
Let $X$ be an involutive commutative two-valued group, $Q\subseteq X$ be a subset consisting of all squares of elements from~$X$. Then $Q$ is a subgroup.
\end{lem}
\begin{proof}
We need to prove that  for any $x, y\in X$ the multiset $x^2*y^2$ consists of squares.
Let $x*y=[z_1,z_2]$, then
\begin{align*}
x*x*y*y&=(x*x)*(y*y)=[e,x^2]*[e,y^2]=[e,e,x^2,x^2,y^2,y^2,x^2*y^2],\\
x*x*y*y&=(x*y)*(x*y)=[z_1,z_2]*[z_1,z_2]=[e,e,z_1^2,z_2^2,z_1*z_2,z_1*z_2].
\end{align*}
If $x^2\ne y^2$ and $z_1\ne z_2$, then $z_1*z_2=[x^2,y^2]$ and hence $x^2*y^2=[z_1^2,z_2^2]$ and the claim is proved. It remains to analyze the cases when $x^2=y^2$, or $z_1= z_2.$

If $x^2=y^2$, then
$
x^2*y^2=\left[e,(x^2)^2\right]
$
in agreement with the claim.

If $z_1=z_2$, then $x*x*y*y=[e,e,e,e,z_1^2,z_1^2,z_1^2,z_1^2]$, then the multiset $x^2*y^2$ consists of squares, which completes the proof.
\end{proof}

\begin{proof}[Proof of Proposition~\ref{propos_neras}]
The `only if' part of the statement of the Proposition is obvious: indeed, if $X=X'\times W$, where $W$ is a non-trivial Boolean group, then for any element~$w\in W\setminus\{e\}$ the pair $(e,w)\in X$ is an element of order~$2$ that is not a square.

Conversely, suppose that in the two-valued group~$X$ there is an element~$w$ of order~$2$ that is not a square. We set $W=\{e,w\}$ and prove that the single-valued group~$W\cong C_2$ can be singled out as a direct factor in~$X$. To do this, it suffices to construct a homomorphism of two-valued groups $f\colon X\to W$ such that $f(w)=w$. Indeed, if such a homomorphism is constructed, then together with the natural projection $\pi \colon X\to X/W$ it will induce an isomorphism $X\cong(X/W)\times W$.

If we consider $W$ as a two-valued group, then the multiplication is given by the formulae
$$e*e=[e,e],\qquad e*w=[w,w],\qquad w*w=[e,e].$$
Therefore, the map $f\colon X\to W$ is a homomorphism of two-valued groups if and only if for any elements of $x,y,z\in X$ such that $z\in x*y$, we have the equality
\begin{equation}\label{eq_f_system_of_equations}
f(x)f(y)f(z)=e.
\end{equation}
These equalities for all triples $\{x,y,z\}$ such that $z\in x*y$ constitute a system of homogeneous linear equations over the field~$\F_2$ (in multiplicative notation) with respect to the variables $f(x)$, where $x\in X$. We need to prove that this system of equations has a solution that satisfies the additional condition $f(w)=w$. To do this it suffices to show that the equation $f(w)=e$  does not follow from equations~\eqref{eq_f_system_of_equations}.

Let us assume the opposite. Then equation $f(w)=e$ is a linear combination over~$\F_2$ of equations of the form ~\eqref{eq_f_system_of_equations}. Since we are using multiplicative notation, this means that the equation $f(w)=e$ is the product of the equations
$$
f(x_i)f(y_i)f(z_i)=e,\qquad i =1,\ldots,m,
$$
for some triples $\{x_i,y_i,z_i\}$ such that $z_i\in x_i*y_i$, taking into account that $f(t)^2=e$ for all~$t$. Therefore, the $3m$-element multiset $[x_1,y_1,z_1,\ldots,x_m,y_m,z_m]$ has the form
$$
\bigl[w,t_1,t_1,t_2,t_2,\ldots,t_n,t_n\bigr]
$$
for some $t_1,\ldots,t_n\in X$, where, of course, $n=(3m-1)/2$.

For each $i$ it follows from $z_i\in x_i*y_i$ that the multiset $x_i*y_i*z_i$ contains the identity~$e$: 
\begin{align*}
e\in(x_1*y_1*z_1)*\cdots*(x_m*y_m*z_m)&=w*(t_1*t_1)*(t_2*t_2)*\cdots*(t_n*t_n)\\
&=w*\bigl[e,t_1^2\bigr]*\bigl[e,t_2^2\bigr]*\cdots*\bigl[e,t_n^2\bigr].
\end{align*}
Since the two-valued group~$X$ is involutive, it follows that
$$
w\in \bigl[e,t_1^2\bigr]*\bigl[e,t_2^2\bigr]*\cdots*\bigl[e,t_n^2\bigr].
$$
Hence, $w$ lies in a subgroup of the two-valued group~$X$ generated by squares~$t_i^2$. Thus, by Lemma ~\ref{lem_squares}, the element $w$ is itself a square, which is impossible. The contradiction means that the system of linear equations ~\eqref{eq_f_system_of_equations} has a solution that satisfies $f(w)=w$. This solution is a two-valued group homomorphism which induces an isomorphism $X\cong (X/W)\times W$.\end{proof}

\begin{proof}[Proof of Proposition~\ref{propos_rasch_iso}]
If the two-valued group $X$ contains some single-valued direct factor, then choose it: $X=X'\times W'$, where $\dim W'>0$. If $X'$ again contains some single-valued direct factor, then we choose it again: $X'=X''\times \widetilde{W}$, where $\dim \widetilde{W}>0$, and hence $X=X''\times W''$, where $W''=W'\times \widetilde{W}$, $\dim W''\ge 2$. Continuing in the same way as long as possible, after the $k$th step we will have the decomposition $X=X^{(k)}\times W^{(k)}$, where $\dim W^{(k)}\ge k$. We need to prove that this process will stop at some point, that is, after a certain step, the two-valued group~$X^{(k)}$ will not contain a single-valued direct factor. To do this, note that every group~$W^{(k)}$ with a doubled operation is a quotient group of the two-valued group~$X$. Therefore, if the two-valued group~$X$ is generated by $N$ elements, then the group~$W^{(k)}$ is also generated by $N$ elements, and hence by $k\le N$. Thus, the process under consideration will stop no later than at the $N$th step, and we will obtain the desired decomposition.

Now let $X=X_1\times W_1=X_2\times W_2$ be two decompositions such that $X_1$ and~$X_2$ do not contain single-valued direct factors.
Assume first that $W_1=W_2$. Then each of the projections $X\to X_1$ and~$X\to X_2$ is taking the quotient with respect to the action of the same group $W_1=W_2$, which immediately implies that $X_1$ and~$X_2$ are isomorphic.

Now suppose that $W_1\ne W_2$.
Let $Q\subseteq X$ be the subset consisting of all squares of elements from~$X$. Then $Q\subseteq X_1$ and~$Q\subseteq X_2$. By Lemma~\ref{lem_squares}, the subset~$Q$ is a subgroup. Let $V\subseteq X$ be a subset consisting of all elements of orders~$\le 2$. Then $V$ is a Boolean group and $W_1$ and~$W_2$ are its subgroups. From the fact that $Q$ is a subgroup of the two-valued group~$X$, it follows that $U=Q\cap V$ is a subgroup of the Boolean group~$V$. Since $Q\subseteq X_1$ and ~$Q\subseteq X_2$, then $U\cap W_1=U\cap W_2=\{e\}$. Assume that the direct product $U\times W_1$ does not coincide with the entire group~$V$. Then the element~$v\in V\setminus(U\times W_1)$ goes under the projection $X=X_1\times W_1\to X_1$ into an element of order~$2$ that is not the square of any element from~$X_1$. By Proposition ~\ref{propos_neras}, this contradicts the fact that $X_1$ does not contain single-valued direct factors. Hence $V=U\times W_1$. Similarly, $V=U\times W_2$. Thus, $W_1$ and~$W_2$ are two subspaces of the vector space~$V$ over~$\F_2$, complementary to the subspace~$U$. Therefore, there is a homomorphism $f\colon W_1\to U$ such that the formula $w\mapsto wf(w)$ defines an isomorphism $W_1\cong W_2$, then the map $F\colon X_1\times W_1\to X_1\times W_1$ defined by the formula $F(x,w)=(x,wf(w))$ is an automorphism of the two-valued group~$X$ leaving~$X_1$ in place and taking~$W_1$ to~$W_2$. Thus there is a direct product decomposition $X=X_1\times W_2$, which implies that $X_1\cong X_2$.
\end{proof}

\section{Properties of non-special two-valued groups}

Recall that an involutive commutative two-valued group~$X$ is called \textit{non-special} if the product $x*y$ consists of two distinct elements for any elements~$x,y\in X$ whose orders are greater than~$2$. In this section, we will prove several useful properties of such two-valued groups.

In all statements of this section $X$ is a non-special involutive commutative two-valued group and $V\subseteq X$ is the subset consisting of the identity~$e$ and all elements of order~$2$. Recall that according to Proposition ~\ref{propos_bullet}, the set $V$ is a Boolean group that acts on~$X$. In all statements in this section, lowercase Latin letters $x,y,z,\ldots$ denote elements of the two-valued group~$X$ unless otherwise stated.

\begin{lem}\label{lem_equal_squares}
The equality $x^2=y^2$ holds if and only if the elements~$x$ and~$y$ lie in the same $V$-orbit.
\end{lem}
\begin{proof}
If $x$ and $y$ are in the same $V$-orbit, so that $y=vx$ for some $v\in V$, then the equality $x^2=y^2$ immediately follows from the formula ~\eqref{eq_action}.

Let us prove the reverse implication. If $x^2=y^2=e$, then $x,y$ lie in~$V$, hence, in the same $V$-orbit. Assume that $x^2=y^2\ne e$. Then the orders of the elements $x$ and $y$ are greater than~$2$. Since the pair~$(x,y)$ is not special, it follows that $x*y=[z_1,z_2]$, where $z_1\ne z_2$. We have
\begin{align*}
x*x*y*y&=(x*x)*(y*y)=[e,x^2]*[e,x^2]=[e,e,e,x^2,x^2,x^2,x^2,x^4],\\
x*x*y*y&=(x*y)*(x*y)=[z_1,z_2]*[z_1,z_2]=[e,e,z_1^2,z_2^2,z_1*z_2,z_1*z_2].
\end{align*}
From the fact that $z_1\ne z_2$, it follows that $e\notin z_1*z_2$. So either $z_1^2=e$, or $z_2^2=e$. Without loss of generality we may assume that $z_1^2=e$, so that $z_1\in V$. Since $z_1\in x*y$ by Lemma ~\ref{lem_basic} we have $y\in z_1*x$, so $y=z_1x$.
\end{proof}

\begin{lem}\label{lem_z1z2}
If $x*y=[z_1,z_2]$, then $z_1*z_2=[x^2,y^2]$.
\end{lem}

\begin{proof}
Suppose first that the elements~$x$ and~$y$ belong to different $V$-orbits, and neither of them lies in~$V$. Then by Lemma ~\ref{lem_equal_squares} we have $x^2\ne y^2$. Moreover, from the fact that the considered two-valued group~$X$ is non-special it follows that $z_1\ne z_2$. Therefore the required equality follows from Lemma ~\ref{lem_z1z2_prelim}.

It remains to consider the case when either the elements~$x$ and~$y$ lie in the same $V$-orbit, or at least one of them lies in~$V$. If~$x$ and~$y$ are in the same $V$-orbit, so that $y=vx$ for some $v\in V$, then $x*y=[v,vx^2]$ and $v*(vx^2) =[x^2,x^2]=[x^2,y^2]$. If $x\in V$, then $z_1=z_2=xy$ and $z_1*z_2=[e,y^2]=[x^2,y^2]$; the case $y\in V$ is similar.
\end{proof}

\begin{cor}\label{cor_z1z2}
If $x*y=[z_1,z_2]$, then $z_2\in x^2*z_1$.
\end{cor}

Applying Lemma ~\ref{lem_z1z2} twice, we get the following statement.

\begin{cor}\label{cor_z1z2sq}
If $x*y=[z_1,z_2]$, then $x^2*y^2=[z_1^2,z_2^2]$.
\end{cor}

\begin{lem}\label{lem_ord4}
Let $x*y=[z_1,z_2]$.
\begin{itemize}
\item[\textnormal{(a)}] If the orders of the elements ~$x$ and~$y$ divide~$4$, then the orders of ~$z_1$ and~$z_2$ also divide~$4$. Thus, the set of all elements of a non-special involutive commutative two-valued group~$X$, whose orders divide~$4$, is a subgroup.
\item[\textnormal{(b)}] Elements~$z_1$ and~$z_2$ lie in the same $V$-orbit if and only if the order of at least one of the elements~$x$ and~$y$ divides~$4$. If $\ord x$ divides 4, then $z_2=x^2z_1$.
\item[\textnormal{(c)}] If all three elements~$x$, $z_1$ and~$z_2$ lie in the same $V$-orbit, then the order of element~$y$ divides~$4$.
\end{itemize}
\end{lem}

\begin{proof}
By Corollary ~\ref{cor_z1z2sq} we have $x^2*y^2=[z_1^2,z_2^2]$. Let now the orders of both elements~$x$ and~$y$ divide~$4$. Then $x^2$ and~$y^2$ belong to the set~$V$. Therefore, the element $z_1^2=z_2^2=x^2y^2$ also belongs to~$V$, so that $z_1^4=z_2^4=e$, which proves the statement~(a).

Further, using the fact that the considered two-valued group~$X$ is not special, and Lemma~\ref{lem_equal_squares}, we obtain the following chain of equivalences, proving the first part of statement~(b):
$$
(\ord x\mid4)\vee(\ord y\mid4)\Leftrightarrow (x^2\in V)\vee (y^2\in V)\Leftrightarrow z_1^2=z_2^2 \Leftrightarrow Vz_1=Vz_2\,.
$$
The second part of statement~(b) immediately follows from Corollary ~\ref{cor_z1z2}.

Let now all three elements~$x$, $z_1$ and~$z_2$ lie in the same $V$-orbit, so that $z_1=v_1x$ and $z_2=v_2x$, where $v_1, v_2\in V$, then $[x^2,y^2]=z_1*z_2=[v_1v_2,v_1v_2x^2]$. Hence either $y^2=v_1v_2$ and then $\ord y$ divides 4, or $x^2=v_1v_2$ and $y^2=v_1v_2x^2=e$ and then $\ord y$ divides 2.
\end{proof}

\begin{lem}\label{lem_xy=zt}
If $x*y=z*t$, then either there exists $v\in V$ such that $z=vx$ and $t=vy$, or there exists $v\in V$ such that $z=vy$ and $t=vx$.
\end{lem}

\begin{proof}
Lemma ~\ref{lem_z1z2} immediately implies that if $x*y=z*t$, then $[x^2,y^2]=[z^2,t^2]$. Suppose $x^2=z^2$ and $y^2=t^2$; the second case is completely similar. Then by Lemma ~\ref{lem_equal_squares} there are elements~$v_1,v_2\in V$ such that $z=v_1x$ and $t=v_2y$. If $v_1=v_2$, then the statement of the Lemma is satisfied. 

Suppose $v_1\ne v_2$ and set $w=v_1v_2$, then $w\ne e$. The equality $x*y=z*t$ can be rewritten as $w(x*y)=x*y$. Let $x*y=[s_1,s_2]$, then there are the following two possible cases.

1) $ws_1=s_1$ and $ws_2=s_2$. Then it follows from statement~(c) of Proposition ~\ref{propos_bullet} that $s_1^2=s_2^2=w$, which means that by Lemma ~\ref{lem_equal_squares} we have $s_2=us_1$ for some $u\in V$. Now by Lemma ~\ref{lem_z1z2} we get that
$$
[x^2,y^2]=s_1*s_2=[u,us_1^2]=[u,uw].
$$
Without loss of generality $x^2=u$ and $y^2=uw$. Then $ux=x$ and $uv_1v_2y=y$, hence $z=v_1x=uv_1x$ and $t=v_2y=uv_1y$. Thus, the statement of the Lemma holds for $v=uv_1$.

2) $ws_1=s_2$ and $ws_2=s_1$. Then by Lemma ~\ref{lem_z1z2} we obtain that
$$
[x^2,y^2]=s_1*s_2=[w,ws_1^2].
$$
Without loss of generality $x^2=w=v_1v_2$. Then $v_1v_2x=x$, hence $z=v_1x=v_2x$. Thus, the statement of the Lemma holds for $v=v_2$.
\end{proof}

\begin{lem}\label{lem_3decomp}
Suppose that elements $x_1$, $x_2$, and~$x_3$ lie in pairwise distinct $V$-orbits, and none of them lies in~$V$. Then the multiset $x_1*x_2*x_3$ consists of four different elements. Also, if $x_1*x_2=[p_3,q_3]$, $x_2*x_3=[p_1,q_1]$, and $x_3*x_1 = [p_2,q_2]$, then
$$
x_1*x_2*x_3=(x_1*p_1)\cup (x_1*q_1) =(x_2*p_2)\cup (x_2*q_2)=(x_3*p_3)\cup (x_3*q_3)
$$
are all three different partitions of a four-element set $x_1*x_2*x_3$ into two two-element subsets. In other words, $x_1*p_1$, $x_1*q_1$, $x_2*p_2$, $x_2*q_2$, $x_3*p_3$, and~$x_3*q_3$ are exactly all six distinct two-element subsets of the four-element set $x_1*x_2*x_3$.
\end{lem}

\begin{proof}
Let us first prove that $6$ multisets~$x_1*p_1$, $x_1*q_1$, $x_2*p_2$, $x_2*q_2$, $x_3*p_3$ and~$x_3*q_3$ are pairwise distinct. Let's assume the opposite. There are two essentially different cases: $x_1*p_1=x_1*q_1$ or $x_1*p_1=x_2*p_2$. Each of the other cases is completely analogous to one of these two. Let's consider these cases in turn.

1) $x_1*p_1=x_1*q_1$. Then by Lemma ~\ref{lem_xy=zt} there is an element $v\in V$ such that either $p_1=q_1=vx_1$ or $vx_1=x_1$ and $vp_1=q_1$. However, the equality $p_1=q_1$ is impossible because the two-valued group~$X$ is not special and $x_2,x_3\notin V$. Therefore $vx_1=x_1$ and $vp_1=q_1$ and $v\ne e$. It follows from statement~(c) of Proposition~\ref{propos_bullet} that $\ord x_1=4$ and $v=x_1^2$, and then it follows from statement~(b) of Lemma~\ref{lem_ord4} that one of the elements~$x_2$ and~$x_3$ (without loss of generality,~$x_2$) has an order that divides~$4$ and $x_2^2p_1=q_1$. Hence $x_1^2x_2^2p_1=p_1$ and~$x_1^2x_2^2q_1=q_1$. However, since $x_1$ and~$x_2$ are in different $V$-orbits, we see from Lemma~\ref{lem_equal_squares} that $x_1^2\ne x_2^2$, so that $x_1^2x_2^2\ne e$. Now again it follows from the statement~(c) of Proposition ~\ref{propos_bullet} that $\ord p_1=\ord q_1=4$ and $p_1^2=q_1^2=x_1^2x_2^2$. Now it follows from Corollary ~\ref{cor_z1z2sq} that $x_2^2*x_3^2=[x_1^2x_2^2,x_1^2x_2^2]$, so that $x_2^2x_3^2=x_1^2x_2^2$ and hence $x_3^2=x_1^2$. We get a contradiction with the fact that $x_1$ and~$x_3$ are in different $V$-orbits.

2) $x_1*p_1=x_2*p_2$. Since $x_1$ and~$x_2$ are in different $V$-orbits, by Lemma~\ref{lem_xy=zt} there is an element $v\in V$ such that $p_2=vx_1$ and~$p_1=vx_2$. Then by Lemma ~\ref{lem_basic} we have $x_3\in x_1*p_2=[v,vx_1^2]$, so $x_3=vx_1^2$ since $x_3\notin V$. Similarly, $x_3=vx_2^2$. So $x_1^2=x_2^2$. We get a contradiction with the fact that $x_1$ and~$x_2$ are in different $V$-orbits.

So the six multisets~$x_1*p_1$, $x_1*q_1$, $x_2*p_2$, $x_2*q_2$, $x_3*p_3$ and~$x_3*q_3$ are pairwise distinct. Each of these $6$ multisets consists of two (possibly identical) elements and is contained in the four-element multiset $x_1*x_2*x_3$. However, if the four-element multiset~$x_1*x_2*x_3$ had multiple elements, then the number of different two-element multisets contained in it would be strictly less than~$6$. Therefore, the multiset~$x_1*x_2*x_3$ consists of four different elements, which completes the proof of the Lemma.
\end{proof}

\begin{lem}\label{lem_yyxxx}
Suppose that the elements $x_1$, $x_2$, and~$x_3$ lie in pairwise distinct $V$-orbits, and none of them lies in~$V$.
If a multiset~$y_1*y_2$ is contained in the multiset $x_1*x_2*x_3$, then there are a permutation~$\lambda,\mu$ of numbers~$1,2$, a permutation $i,j,k$ of numbers~$1,2,3$, and an element~$v\in V$ such that $y_{\lambda}=vx_i$ and $y_{\mu}\in v(x_j*x_k)$.
\end{lem}

\begin{proof}
Let $x_i*x_j=[p_k,q_k]$ for each permutation $i,j,k$ of numbers~$1,2,3$. By Lemma~\ref{lem_3decomp}, we know that~$x_1*p_1$, $x_1*q_1$, $x_2*p_2$, $x_2*q_2$, $x_3*p_3$, and~$x_3*q_3$ are exactly all $6$ two-element subsets of the four-element set~$x_1*x_2*x_3$, so one of them coincides with the multiset~$y_1*y_2$. Without loss of generality, we may assume that $y_1*y_2=x_1*p_1$. Then by Lemma~\ref{lem_xy=zt} there is an element~$v$ such that $y_{\lambda}=vx_1$ and $y_{\mu}=vp_1\in v(x_2*x_3)$ for some permutation~$\lambda,\mu$ of numbers~$1,2$.
\end{proof}

\section{Two-valued groups with an element of order other than 1, 2, or~4}\label{section_not124}

In this section, we will prove the following result.

\begin{theorem}
\label{theorem_nonspecial_classify}
Let $X$ be a non-special involutive commutative two-valued group containing at least one element whose order is neither~$1$, nor~$2$, nor~$4$. Then $X$ is isomorphic to a coset two-valued group of the form $X_A^{\ba}=A/\iota_{\ba}$, where $A$ is an abelian group and $\iota_{\ba}$ is the antipodal involution on~$A$.
\end{theorem}

\begin{remark}\label{remark_nonspecial_classify}
This theorem does not require that the two-valued group~$X$ be finitely generated. Moreover, it is easy to check that a two-valued group $A/\iota_{\ba}$ is finitely generated if and only if the abelian group~$A$ is finitely generated.
\end{remark}

The proof of Theorem ~\ref{theorem_nonspecial_classify} will be given using the explicit construction of the abelian group~$A$. This construction will essentially depend on the choice of the element $t\in X$ such that $\ord t\notin\{1,2,4\}$. However, as it will be shown in section ~\ref{section_final_proof} (Proposition ~\ref{propos_different}), for a finitely generated two-valued group~$X$, the resulting abelian group~$A$ will not depend on the choice of the element~$t$ up to isomorphism. Apparently, the same is true without the assumption of finite generation, but we do not know the proof of this fact.

The following Proposition describes the explicit construction of an abelian group~$A$. The proof, broken down into a series of Lemmas, will take up the rest of this section.

\begin{propos}\label{propos_main_construction}
Let $X$ be a non-special involutive commutative two-valued group, $V\subseteq X$ be the subset consisting of all elements of orders~$\le 2$, and $t\in X$ be an element such that $\ord t\notin\{1,2,4\}$. Consider the set
\begin{equation}
A=\{(x,p)\in X\times X\mid p\in t*x\}.
\end{equation}
Then

\textnormal{(а)} If $(x,p)\in A$, $(y,q)\in A$, $t*x=[p, p']$ and $t*y=[q,q']$, then there is a unique pair $(z,r)\in A$ such that
\begin{itemize}
\item[(i)] $z$ lies in each of the three multisets $x*y$, $p*q'$ and $p'*q$;
\item[(ii)] $r$ lies in each of the two multisets $x*q$ and $p*y$.
\item[(iii)] an element~$r'$ such that $t*z=[r,r']$ lies in each of the two multisets $x*q'$ and $p'*y$.
\end{itemize}

\textnormal{(b)} Define the operation of multiplication $\bullet\colon A\times A\to A$ by the formula
$$
(x,p)\bullet (y,q)= (z,r),
$$
where $(z,r)\in A$ is the only pair that has the properties ~\textnormal{(i)}, \textnormal{(ii)} and ~\textnormal{(iii)}. Then $(A,\bullet)$ is a commutative (single-valued) group with identity $(e,t)$ and the inverse element defined by the formula $(x,p)^{-1}=(x,p')$, where $t*x=[p,p']$.

\textnormal{(c)} The projection onto the first factor $A\to X$, $(x,p)\mapsto x$ induces an isomorphism of two-valued groups $A/\iota_{\ba}\cong X$, where $\iota_{\ba}$ is the antipodal involution: $\iota_{\ba}(a)=a^{-1}$.
\end{propos}

The following Lemma is a direct consequence of Proposition ~\ref{propos_bullet}.

\begin{lem}\label{lem_uvaction}
The group~$V$ acts on the set~$A$ according to the formula $v(x,p)=(vx,vp)$, where $v\in V$, $(x,p)\in A$. If statement~\textnormal{(а)} of Proposition~\ref{propos_main_construction} is true for elements $\alpha=(x,p)$ and $\beta=(y,q)$ of set~$A$, then for any elements $u,v\in V$ statement~\textnormal{(а)} of Proposition~\ref{propos_main_construction} is true for elements $u\alpha=(ux,up)$ and~$v\beta=(vy,vq)$ and
\begin{equation}\label{eq_uv_mult}
(u\alpha)\bullet (v\beta)=uv(\alpha\bullet\beta).
\end{equation}
\end{lem}

\begin{lem}\label{lem_formulae_special}
Statement~\textnormal{(а)} of Proposition~\ref{propos_main_construction} holds in special cases when $(y,q)$ is one of the pairs $(e,t)$, $(t,e)$, $(t,t^2)$, $(t^2,t)$, $(x,p)$, $(x,p')$, $(p,x)$, $(p',x)$ and~$(s,p)$, where $t*p=[x,s]$. We have the following formulae
\begin{align}
(x,p)\bullet (e,t) &= (x,p),\label{eq_et}\\
(x,p)\bullet (t,e) &= (p',x),\label{eq_te}\\
(x,p)\bullet (t,t^2)&=(p,s),\label{eq_tt^2}\\
(x,p)\bullet (t^2,t)&=(s',p')\label{eq_t^2t}\\
(x,p)\bullet (x,p) &= (x^2,f),\label{eq_xp}\\
(x,p)\bullet (x,p') &=(e,t),\label{eq_xp'}\\
(x,p)\bullet (p,x) &=(t,e),\label{eq_px}\\
(x,p)\bullet (p',x) &=(f',x^2),\label{eq_p'x}\\
(x,p)\bullet (s,p) &=(t^2,t),\label{eq_sp}
\end{align}
where $t*p'=[x,s']$, $x*p=[t,f]$ and $x*p'=[t,f']$.
\end{lem}

\begin{proof}
The fact that the pairs~$(z,r)$ on the right-hand sides of formulae~\eqref{eq_et}--\eqref{eq_sp} belong to the set~$A$ and have the properties~(i), (ii) and~(iii) can be verified directly using Lemma~\ref{lem_basic} and Corollary~\ref{cor_z1z2}. It remains to prove the uniqueness of such a pair~$(z,r)$ in each of the cases.

First of all, note that the pair $(z,r)$ is definitely unique if at least one of the elements~$x$, $p$, $p'$, $y$, $q$, and~$q'$ lies in~$V$. Indeed,
\begin{itemize}
\item if $x\in V$, then $z\in x*y=[xy,xy]$ and $r\in x*q=[xq,xq]$, hence $(z,r)=(xy,xq)$;
\item if $p\in V$, then $z\in p*q'=[pq',pq']$ and $r\in p*y=[py,py]$, hence $(z,r)=(pq',py)$;
\item if $p'\in V$, then $z\in p'*q=[p'q,p'q]$ and $r'\in p'*y=[p'y,p'y]$, hence $(z,r')=(p'q,p'y)$, after that the element~$r$ is uniquely determined from the condition of $t*z=[r,r']$;
\item cases $y\in V$, $q\in V$ and $q'\in V$ are completely similar.
\end{itemize}
These arguments cover the cases $(y,q)=(e,t)$, $(t,e)$ and $(t,t^2)$. Consider the remaining cases in turn, assuming that none of the elements~$x$, $p$, $p'$, $y$, $q$ and~$q'$ lies in~$V$.

1. In the case when $(y,q)=(t^2,t)$ we have $r\in x*q=[p,p']$, and hence $r=p$ or $r=p'$. In addition, we have $r\in p*y= p*t^2$. Assume first that $p\notin p*t^2$. Then the only possible value is $r=p'$. Moreover, the element $z$ must belong to the multisets $t*r=[x,s']$ and $x*y=x*t^2$. These two conditions uniquely define~$z$ if $x\notin x*t^2$. If $x\in x*t^2$, then $t^2\in x*x$, hence $t^2=x^2$, so by Lemma ~\ref{lem_equal_squares} we get that $x=ut$ for some $u\in V$. Then $[p,p']=x*t=[u,ut^2]$, which is impossible, because by assumption none of the elements~$p$ and~$p'$ lie in~$V$.

Suppose now that $p\in p*t^2$, so that $t^2\in p*p$. Then $t^2=p^2$, so by Lemma ~\ref{lem_equal_squares} we have $p=ut$ for some $u\in V$.
On the other hand, $y=t^2$, so $t*y=[t,t^3]$ and $q'=t^3$.
The element~$z$ must belong to the multisets $x*y=[u,ut^4]$ and $p*q'=[ut^2,ut^4]$. Since $t^2\ne e$, it follows that $z=ut^4$.
Now the element~$r$ must belong to the multisets $t*z=[ut^3,ut^5]$ and $x*q=p*y=[ut,ut^3]$. If $t^5\ne t$, then the only element $z=ut^3$ satisfies these two conditions. If $t^5=t$, then either $t^4=e$ or $t^6=e$. The first of these equalities is impossible since $\ord t\notin\{1,2,4\}$, the second one is impossible since in that case the element $q'=t^3$ would belong to the set~$V$, which contradicts our assumption.

Before going through the rest of the cases, we make one general remark. We have already proved that statement~(a) is true when $(y,q)$ is one of the pairs of $(e,t)$, $(t,e)$, $(t,t^2)$ and $(t^2,t)$, that is, when one of the elements~$y$ and ~$q$ is equal to~$t$. According to Lemma~\ref{lem_uvaction}, this implies that statement~(a) is true if one of the elements of~$y$ and~$q$ lies in the orbit~$Vt$. Similarly, the statement ~(a) is true if one of the elements~$x$ and~$p$ lies in the orbit~$Vt$. Hence in the proof of statement  (a) we can assume that none of $x$, $p$, $y$, and~$q$ lies in~$Vt.$ According to Lemma ~\ref{lem_equal_squares}, this is equivalent to the assumption that none of the elements of~$x^2$, $p^2$, $y^2$ and ~$q^2$ is equal to~$t^2.$ Now we continue to consider all the cases under this additional assumption.

2. In the case when $(y,q)=(x,p)$, we have $z\in x*x=[e,x^2]$ and $z\in p*p'$. If $z=e$, then $p=p'$, which is impossible due to non-speciality of the two-valued group~$X$, since $x\notin V$ and $t\notin V$. So the only possible value is $z=x^2$. Further we have $r\in x*p=[t,f]$ and $r\in x^2*t$. If $t\in x^2*t$, then $x^2\in t*t=[e,t^2]$, which is impossible, because by the assumptions $x\notin V$ and $x^2\ne t^2$. So $t\notin x^2*t$, and hence the only possible value is $r=f$.

3. In the case when $(y,q)=(x,p')$, the element $z$ must belong to each of the multisets $x*x=[e,x^2]$, $p*q'=p*p=[e,p^2]$, and $p'*q=p'*p'=[e,{p'}^2]$. If the elements $x^2$, $p^2$, and~${p'}^2$ are not all the same, then it follows that $z=e$ and hence $r=t$. If $x^2=p^2={p'}^2$, then by Lemma ~\ref{lem_equal_squares} we get that $p=ux$ and $p'=u'x$, where $u,u'\in V$. Then, by Lemma ~\ref{lem_z1z2}, the equality $[x^2,t^2]=p*p'=[uu',uu'x^2]$ holds. So either $t^2=uu',$ or $x^2=uu'$ and $t^2=uu'x^2=e$. Neither is possible since $\ord t\notin\{1,2,4\}$. Hence, the case $x^2=p^2={p'}^2$ is impossible.

4. In the case when $(y,q)=(p,x)$, we have $r\in x*x=[e,x^2]$ and $r\in p*p=[e,p^2]$. If $x^2\ne p^2$, then it follows that $r=e$ and hence $z=t$. Let $x^2=p^2$. Then by Lemma ~\ref{lem_equal_squares} we get that $p=ux$ for some $u\in V$, so $t\in x*p=[u,ux^2]$. Since $t\notin V$, it follows that $t=ux^2$. Then $t*x=[ux,ux^3]$ and hence $p'=ux^3$. Therefore, the element~$z$ belongs to the multisets $x*y=[u,ux^2]$ and $p'*q=[ux^2,ux^4]$. Since $x^4=t^2\ne e$ the only possible value is $z=ux^2=t$. Thus $r\in z*t=[e,x^4]$. However, it was already noted above that $r\in [e,x^2]$. Moreover, by assumption we have $x^4=t^2\ne x^2$. So the only possible value is $r=e$.

5. In the case when $(y,q)=(p',x)$, we have $r\in x*q=[e,x^2]$ and $r\in p*y=p*p'$. However, since $x\notin V$ and $t\notin V$, from the non-speciality of the two-valued group~$X$ it follows that $p\ne p'$. Therefore $e\notin p*p'$ and the only possible value is $r=x^2$. Further we have $z\in x*y=[t,f']$ and $z\in t*r=t*x^2$. If the inclusion $t\in t*x^2$ were true, then we would get that $x^2\in t*t=[e,t^2]$, which is impossible since $x\notin V$ and $x^2\ne t^2$. So $t\notin t*x^2$ and hence the only possible value is $z=f'$.

6. In the case when $(y,q)=(s,p)$, we have $z\in p'*q=p'*p=[t^2,x^2]$ (by Lemma ~\ref{lem_z1z2}) and $z\in x*s$.

Suppose that $x^2\notin x*s$. Then the only possible value is $z=t^2$, and hence $r=t$ or~$r=t^3$. However, the element~$r$ must lie in the multiset $x*q=x*p=[t,f]$. If $t^3=f$ were true, then according to Lemma ~\ref{lem_z1z2} we would have $[x^2,p^2]=t*f=[t^2,t^4]$, which is impossible, since by assumption $x^2\ne t^2$ and $p^2\ne t^2$. So $t^3\ne f$  and therefore the only possible value is $r=t$.

Now suppose that $x^2\in x*s$. Then $s\in x^2*x=[x,x^3]$. However, $x\ne s$, since $t,p\notin V$ and the two-valued group~$X$ is not special. Therefore $s=x^3$. By Lemma ~\ref{lem_z1z2}, we have $[t^2,p^2]=x*s=[x^2,x^4]$. Since $x^2\ne t^2$, it follows that $t^2=x^4$ and $p^2=x^2$. Hence $p=ux$ and $t=vx^2$ for some $u,v\in V$. Then $[p,p']=t*x=[vx,vx^3]$. If the equality $p'=vx$ were true, then by Lemma ~\ref{lem_z1z2} we would get that $[t^2,x^2]=p*p'=[uv,uvx^2]$, which would imply that either $t^2=uv$ or $x^2=uv$ and $t^2=uvx^2=e$, which is impossible since $\ord t\notin\{1,2,4\}$. Hence, $p=vx$ and $p'=vx^3$. Further $[q,q']=t*s=[vx,vx^5]$, hence $q'=vx^5$. The element $z$ must belong to the multisets $x*y=[x^2,x^4]$ and $p*q'=[x^4,x^6]$. If the equality $x^2=x^6$ were true, then it would follow that $x^8=e$, hence $t^4=(vx^2)^4=e$, which is not true. Therefore $x^2\ne x^6$ and hence the only possible value is $z=x^4$. Further, the element $r$ must belong to the multisets $x*q=[v,vx^2]$ and $p*y=[vx^2,vx^4]$. Equality $x^4=e$ is impossible, because it would imply that $t^2=(vx^2)^2=e$. Therefore, the only possible value is $r=vx^2$.
\end{proof}

\begin{proof}[Proof of statement (\textnormal{а}) of Proposition ~\ref{propos_main_construction}]
Lemmas ~\ref{lem_uvaction} and ~\ref{lem_formulae_special} imply that the statement~(\textnormal{а}) for the pairs $(x,p)$ and~$(y,q)$ holds if at least one of the following conditions is satisfied:
\begin{enumerate}
\item one of the orbits $Vx$, $Vy$, $Vp$ and~$Vq$ coincides either with~$V$ or with~$Vt$,
\item one of the orbits $Vx$ and~$Vp$ coincides with one of the orbits $Vy$ and~$Vq$.
\end{enumerate}

It remains for us to consider the case of `general position', when neither of these two conditions is satisfied. Then by Lemma ~\ref{lem_3decomp} the multiset $t*x*y$ consists of four different elements and $(x*q)\cup (x*q')$ and $(p*y)\cup (p'*y)$ are two different partitions of this four-element set into two two-element subsets. Therefore, the sets $x*q$ and~$p*y$ have a single common element, which we denote by ~$r$. Also, $t*x*y=(t*z)\cup (t*z')$, where $x*y=[z,z']$, which means that $r$ belongs to exactly one of the two two-element sets $t*z$ and~$t*z'$. Renaming the elements ~$z$ and ~$z'$ as needed, we can assume that $r\in t*z$ and $r\notin t*z'$. Then $(z,r)$ is the only pair in~$A$ that can have properties~(i), (ii) and~(iii). It remains to show that it does indeed have them. Property~(ii) and inclusion $z\in x*y$ hold by construction.

Let us prove property~(iii). By Lemma~\ref{lem_3decomp}
$$
t*x*y=(x*q)\cup (x*q')=(p*y)\cup (p'*y)=(t*z)\cup (t*z')
$$
are all three different partitions of the four-element set~$t*x*y$ into two two-element subsets. Therefore, since the element $r$ belongs to the subsets $x*q$, $p*y$ and $t*z$, it follows that the element~$r'$ of the subset~$t*z$, which is different from~$r$, belongs to the subsets~$x*q'$ and~$p'*y$, which is the property~(iii).

Now let us prove that $z\in p*q'$. The element~$z$ belongs to the multiset
$$
x*y\subseteq (p*t)*y=p*(t*y)=(p*q)\cup(p*q').
$$
Therefore, if $z\notin p*q'$, then the element~$z$ must lie in the multiset~$p*q$. Let's assume that this is indeed the case. By Lemma ~\ref{lem_3decomp} for the triple $(t,x,y)$ the multiset $t*x*y$ consists of four different elements and $x*q$, $p*y$ and $t*z$ are its three pairwise distinct two-element subsets intersecting in the common element~$r$. Hence, the union of these three subsets (without taking into account the multiplicities of the elements) coincides with the entire set $t*x*y$. On the other hand, since $x\in t*p$, $y\in t*q$, and $z\in p*q$, each of the multisets $x*q$, $p*y$, and $t*z$ is contained in the four-element multiset $t*p*q$. Hence $t*p*q=t*x*y$. Therefore, the $32$-element multiset $(t*p*q)*(t*x*y)$ must contain the identity~$e$ with multiplicity at least~$4$. On the other hand, we have $x*p = [t,f]$ and $y*q=[t,g]$ for some $f,g\in X$; then
\begin{multline*}
(t*p*q)*(t*x*y)=(t*t)*(x*p)*(y*q)=(t*t)*[t,f]*[t,g]\\{}=[t*t*t*t, t*t*t*f, t*t*t*g, t*t*f*g].
\end{multline*}
At the same time, the multiset
$$
t*t*t*t=[e,t^2]*[e,t^2]=[e,e,e,t^2,t^2,t^2,t^2,t^4]
$$
contains $e$ only with multiplicity~$3$ since $\ord t\ne \{1,2,4\}$. Therefore, the element~$e$ must belong to at least one of the three multisets~$t*t*t*f$, $t*t*t*g$ and~$t*t*f*g$.

If $e\in t*t*t*f$, then $f\in t*t*t=[t,t,t,t^3]$, then either $f=t$ or $f=t^3$. By Lemma~\ref{lem_z1z2}
$$
[x^2,p^2]=t*f=\left\{
\begin{aligned}
&[e,t^2],&&\text{если}\ f=t,\\
&[t^2,t^4],&&\text{если}\ f=t^3.
\end{aligned}
\right.
$$
Hence, $x^2=t^2$ or $p^2=t^2$, so by Lemma ~\ref{lem_equal_squares} it follows that $Vx=Vt$ or~$Vp=Vt$, which contradicts the assumption that condition~(1) is not satisfied. Similarly, we arrive at a contradiction by assuming that $e\in t*t*t*g$. If $e\in t*t*f*g$, then by Lemma ~\ref{lem_z1z2} we have
$$
t*t*f*g=(t*f)*(t*g)=[x^2,p^2]*[y^2,q^2].
$$
Hence, one of the equalities $x^2=y^2$, $x^2=q^2$, $p^2=y^2$, and $p^2=q^2$ holds, which by Lemma ~\ref{lem_equal_squares} implies that condition~(2) holds. This leads to a contradiction again.

Thus, $z\notin p*q$, so $z\in p*q'$ and, similarly, $z\in p'*q$. Therefore, the pair $(z,r)$ indeed has properties~(i), (ii) and~(iii).
\end{proof}

By construction, the operation~$\bullet$ is commutative. It follows from the formula~\eqref{eq_et} that the element $\varepsilon=(e,t)$ is the identity for the multiplication~$\bullet$, that is, $\varepsilon\bullet\alpha=\alpha$ for all~$\alpha\in A$. Further, it follows from the formula ~\eqref{eq_xp'} that $(x,p)\bullet (x,p')=\varepsilon$ if $t*x=[p,p']$. If $\alpha=(x,p)\in A$, then we will denote the element~$(x,p')$ by~$\alpha^{-1}$. Then $(\alpha^{-1})^{-1}=\alpha$ and $\varepsilon^{-1}=\varepsilon$. Note, however, that we have not yet proved the uniqueness of the inverse, that is, that the equality $\alpha\bullet\beta=\varepsilon$ implies that $\beta=\alpha^{-1}$. In order to prove that $(A,\bullet)$ is an abelian group, we need to prove the uniqueness of the inverse and the associativity of multiplication. The proof will be divided into several Lemmas.

The next Lemma immediately follows from the fact that properties~(i), (ii) and~(iii) do not change when $p\leftrightarrow p'$, $q\leftrightarrow q'$ and $r\leftrightarrow r'$ are simultaneously interchanged.

\begin{lem}\label{lem_ab-1}
For any two elements $\alpha,\beta\in A$, the equality $(\alpha\bullet\beta)^{-1}=\alpha^{-1}\bullet \beta^{-1}$ holds.
\end{lem}

\begin{lem}\label{lem_a-1ab}
For any two elements $\alpha,\beta\in A$, the equality $\alpha^{-1}\bullet(\alpha\bullet\beta)=\beta$ holds.
\end{lem}

\begin{proof}
Let $\alpha=(x,p)$, $\beta=(y,q)$ and $\alpha\bullet\beta=(z,r)$; let also $t*x=[p,p']$, $t*y=[q,q']$ and $t*z=[r,r']$. Then the triples~$(x,p,p')$, $(y,q,q')$, and $(z,r,r')$ satisfy properties~(i), (ii) and~(iii). We need to prove that $(x,p')\bullet (z,r)=(y,q)$. This means that we need to prove that properties~(i),(ii),(iii) will still hold if we substitute triples~$(x,p',p)$, $(z,r,r')$ and~$(y,q,q')$ instead of $(x,p,p' )$, $(y,q,q')$ and $(z,r,r')$ respectively. Let's write down the required properties:
\begin{itemize}
\item[(i')] $y$ belongs to each of the multisets~$x*z$, $p*r$ and~$p'*r'$,
\item[(ii')] $q$ belongs to each of the multisets~$x*r$ and~$p'*z$,
\item[(iii')] $q'$ belongs to each of the multisets~$x*r'$ and~$p*z$.
\end{itemize}
The fact that the union of these three properties is equivalent to the union of the three initial properties~(i), (ii) and~(iii) immediately follows from Lemma ~\ref{lem_basic}.
\end{proof}

\begin{cor}\label{cor_bijection}
For any element $\alpha\in A$, the operation of multiplication by $\alpha$ defines a bijection of the set~$A$ onto itself. In particular, if $\alpha\bullet\beta=\varepsilon$, then $\beta=\alpha^{-1}$, that is, the inverse element is unique.
\end{cor}

Introduce the notation $\tau=(t,e)$ and $\tau^k=(t^k,t^{k-1})$ for all $k\in\Z$. Denote by $\T\subseteq A$ the subset consisting of all elements of $\tau^k$, $k\in\Z$. By directly verifying properties~(i), (ii), and~(iii), we easily obtain the following statement.

\begin{lem}\label{lem_T}
The equalities $\tau^k\bullet\tau^m=\tau^{k+m}$ hold for all $k,m\in\Z$. Thus, the set $\T$ with the operation~$\bullet$ is the cyclic group generated by the element~$\tau$.
\end{lem}

\begin{lem}\label{lem_tab}
For any two elements $\alpha,\beta\in A$, the equality $\tau\bullet(\alpha\bullet\beta)=(\tau\bullet\alpha)\bullet\beta$ holds.
\end{lem}

\begin{proof}
Let $\alpha=(x,p)$, $\beta=(y,q)$, $\alpha\bullet\beta=(z,r)$ and $t*x=[p,p']$, $t*y=[q,q']$, $t*z=[r,r']$. Then the triples~$(x,p,p')$, $(y,q,q')$, and $(z,r,r')$ satisfy properties~(i), (ii) and~(iii). According to formula ~\eqref{eq_te}, we have $\tau\bullet\alpha=(p',x)$ and $\tau\bullet(\alpha\bullet\beta)=(r',z)$. Thus, we need to prove that
\begin{equation}\label{eq_tauab}
(p',x)\bullet (y,q)=(r',z). 
\end{equation}
Let $(p',x)\bullet (y,q)=(\zeta,\rho)$.

Let us first prove that $\rho=z$. Indeed, according to the multiplication construction~$\bullet$, the element~$\rho$ must belong to the multisets~$p'*q$ and~$x*y$. However, the element~$z$ also belongs to these two multisets. Assume that $\rho\ne z$. Then $p'*q=x*y=[z,\rho]$. By Lemma~\ref{lem_xy=zt}, this implies that there is an element $u\in V$ such that either $p'=ux$ and $q=uy$ or $p'=uy$ and $q=ux$. Let's consider these cases separately.

1. Let $p'=ux$ and $q=uy$. Then the element~$t$ belongs to each of the multisets $x*p'=[u,ux^2]$ and $y*q=[u,uy^2]$. Since $t\notin V$, we get that $t=ux^2=uy^2$, so $x^2=y^2$. By Lemma ~\ref{lem_equal_squares}, this implies that $y=vx$ for some element $v\in V$. We have $t*x=[ux,ux^3]$, hence $p=ux^3$.
Thus, $\alpha=(x,ux^3)$, $\beta=(vx,uvx)$. Now it follows from the formula ~\eqref{eq_xp'} that $\alpha\bullet\beta=(v,vt)$, that is, $z=v$ and $r=r'=vt$. The equality ~\eqref{eq_tauab}, which we need to prove, takes the form $(ux,x)\bullet (vx,uvx)=(vt,v)$ and immediately follows from the formula ~\eqref{eq_px}.

2. Let $p'=uy$ and $q=ux$. From formulae~\eqref{eq_p'x} and ~\eqref{eq_xp} it follows that
\begin{gather*}
\alpha\bullet\beta = (x,p)\bullet (up',ux)=(uf',ux^2),\\
(\tau\bullet\alpha)\bullet\beta = (p',x)\bullet (up',x) = (u{p'}^2,uf'),
\end{gather*}
where $x*p'=[t,f']$. By Lemma~\ref{lem_z1z2} we have $t*f'=[x^2,{p'}^2]$, so ~\eqref{eq_tauab} becomes
$$
\tau\bullet (uf',ux^2)= (u{p'}^2,uf')
$$
and follows from formula~\eqref{eq_te}.

Thus we have proved that  $\rho=z$ in all cases.
We need to prove that $\zeta=r'$. Let's assume the opposite. According to the multiplication construction~$\bullet$ the element $\zeta$ must belong to the multisets~$p'*y$ and~$x*q'$. However, the element~$r'$ also belongs to these two sets. Hence $p'*y=x*q'=[\zeta,r']$. According to Lemma~\ref{lem_xy=zt}, this implies that there is an element $u\in V$ such that either $p'=ux$ and $q'=uy$ or $p'=uy$ and $q'=ux$. Let's consider these cases separately.

1. Let $p'=ux$ and $q'=uy$. Just like in case 1 above, we get that $t=ux^2=uy^2$ and $y=vx$ for some element $v\in V$. We have $t*x=[ux,ux^3]$, hence $p=ux^3$. Similarly, $q=uy^3=uvx^3$. Thus, $\alpha=(x,ux^3)$, $\beta=(vx,uvx^3)$. Now it follows from formula~\eqref{eq_xp} that $\alpha\bullet\beta=(vx^2,vf)$, where $x*p=[t,f]$. However, $x*p=[ux^2,ux^4]$, so $f=ux^4$. Thus $z=vx^2=uvt$ and $r=uvx^4$. We have $t*z=[uv,uvx^4]$, hence $r'=uv$. The equality ~\eqref{eq_tauab} takes the form $(ux,x)\bullet (vx,uvx^3)=(uv,uvt)$ and immediately follows from the formula ~\eqref{eq_xp'}.

2. Let $p'=uy$ and $q'=ux$. From the formula ~\eqref{eq_px} it follows that
$$
\alpha^{-1}\bullet\beta^{-1}=(x,p')\bullet (up',ux)=(ut,u)=u\tau.
$$
Hence, by Lemmas~\ref{lem_ab-1} and ~\ref{lem_T}
\begin{gather*}
\alpha\bullet\beta = u\tau^{-1},\\
\tau\bullet(\alpha\bullet\beta)=u\varepsilon,\\
(\tau\bullet(\alpha\bullet\beta))\bullet\beta^{-1}=u\beta^{-1}=(p',x)=\tau\bullet\alpha,
\end{gather*}
hence, in view of Corollary ~\ref{cor_bijection}, we have the required equality ~\eqref{eq_tauab}.

Thus, $\zeta=r'$ in all cases, which completes the proof of the Lemma.
\end{proof}

\begin{cor}\label{cor_T}
The operation~$\bullet$ restricted to $\T\times A\subseteq A\times A$ defines the action of the cyclic group~$\T$ on the set~$A$ such that for any elements $\alpha,\beta\in A$ and any integer~$k$ the equality $\tau^k\bullet(\alpha\bullet\beta)=(\tau^k\bullet\alpha)\bullet\beta$ is satisfied.
\end{cor}

Recall that the formula $v(x,p)=(vx,vp)$, $v\in V$, $(x,p)\in A$ defines the action of the commutative group~$V$ on the set~$A$. According to the formula ~\eqref{eq_uv_mult}, this action has the property $v(\alpha\bullet\beta)=(v\alpha)\bullet\beta$ for all $v\in V$, $\alpha,\beta\in A$. It is easy to see that the constructed actions of the groups~$V$ and~$\T$ on the set~$A$ commute. Thus, we get the action of the direct product $G=V\times \T$ on the set~$A$, which has the property
\begin{equation}\label{eq_G_action}
g(\alpha\bullet\beta)=(g\alpha)\bullet\beta
\end{equation}
for all $g\in G$, $\alpha,\beta\in A$. It is easy to see that the following two Lemmas follow from this property and the uniqueness of the inverse with respect to the operation~$\bullet$.
 
\begin{lem}\label{lem_Gorbits}
The operation~$\bullet$ induces a well-defined commutative operation on the set of $G$-orbits $A/G$, with respect to which the orbit~$G\varepsilon$ is the unit, and the element inverse to each orbit $G\alpha$ is unique and coincides with the orbit~$G\alpha^{-1}$.
\end{lem}

\begin{lem}\label{lem_assoc_compare}
Let the elements~$\alpha_1$, $\alpha_2$, $\alpha_3$, $\beta_1$, $\beta_2$, $\beta_3$ of the set~$A$ be such that for each~$i$ the elements~$\alpha_i$ and~$\beta_i$ lie in the same $G$-orbit. Then $(\alpha_1\bullet\alpha_2)\bullet\alpha_3=\alpha_1\bullet(\alpha_2\bullet\alpha_3)$ if and only if $(\beta_1\bullet\beta_2)\bullet\beta_3=\beta_1\bullet(\beta_2\bullet\beta_3)$.
\end{lem}

From the last Lemma and Lemma ~\ref{lem_a-1ab}, the following statement easily follows.

\begin{cor}\label{cor_assoc_special}
If the elements~$\alpha_1$, $\alpha_2$, $\alpha_3$ of the set~$A$ are such that
$\alpha_i$ lies in the $G$-orbit $G\varepsilon$ for some~$i$ or
$\alpha_i$ and $\alpha_j^{-1}$ lie in the same $G$-orbit for some $i\ne j$, then $(\alpha_1\bullet\alpha_2)\bullet\alpha_3=\alpha_1\bullet(\alpha_2\bullet\alpha_3)$.
\end{cor}

\begin{lem}\label{lem_diff_orb}
Let the elements~$\alpha=(x,p)$ and~$\beta=(y,q)$ of the set~$A$ be such that $\beta$ does not lie in the same $G$-orbit with either the element~$\alpha$ or the element~$\alpha^{-1}$, and let $t*x=[p,p']$ and $t*y=[q,q']$. Then none of the elements~$x$, $p$ and~$p'$ lie in the same $V$-orbit with any of the elements~$y$, $q$ and~$q'$.
\end{lem}

\begin{proof}
It follows from the equalities ~\eqref{eq_te} and ~\eqref{eq_tt^2} that the elements~$(x,p)$, $(p',x)$ and~$(p,s)$ (where $t*p=[x,s]$) lie in the same $G$-orbit. Similarly, the elements~$(y,q)$, $(q',y)$ and~$(q,z)$ (where $t*q=[y,z]$) also lie in the same $G$-orbit. Therefore, if one of the elements~$x$, $p$ and~$p'$ lies in the same $V$-orbit with one of the elements~$y$, $q$ and~$q'$, then in $G$-orbits of elements~$\alpha$ and~$\beta$ there are elements $\widetilde{\alpha}=\bigl(\widetilde{x},\widetilde{p}\bigr)$ and $\widetilde{\beta}=\bigl(\widetilde{y},\widetilde{q}\bigr)$, respectively, such that $\widetilde{y}=v\widetilde{x}$ for some $v\in V$. Then either $\widetilde{q}=v\widetilde{p}$ or $\widetilde{q}=v\widetilde{p}\,'$, where $t*\widetilde{x}=[\widetilde{p},\widetilde{p}\,']$. Hence, the element $\widetilde{\beta}$ lies in the same $G$-orbit with one of the elements~$\widetilde{\alpha}$ and~$\widetilde{\alpha}^{-1}$, so, by Lemma~\ref{lem_Gorbits}, it follows that the element $\beta$ lies in the same $G$-orbit with one of the elements~$\alpha$ and~$\alpha^{-1}$.
\end{proof}

The next Lemma follows directly from the construction of the operation~$\bullet$.

\begin{lem}\label{lem_triple_condition}
Let $\alpha_i=(x_i,p_i)$, $i=1,2,3$, be elements of the set~$A$ and $t*x_i=[p_i,p_i']$. Let $(z,r)=(\alpha_1\bullet \alpha_2) \bullet\alpha_3$ and $t*z=[r,r']$. Then
\begin{enumerate}
\item the element $z$ belongs to the multiset $x_1*x_2*x_3$ and to each of the six multisets $x_i*p_j*p_k'$, where $i,j,k$ are all possible permutations of numbers~$1,2,3$,
\item the element~$r$ belongs to each of the three multisets $x_1*x_2*p_3$, $x_1*p_2*x_3$ and~$p_1*x_2*x_3$,
\item the element~$r'$ belongs to each of the three multisets $x_1*x_2*p'_3$, $x_1*p'_2*x_3$ and~$p'_1*x_2*x_3$.
\end{enumerate}
\end{lem}

\begin{lem}\label{lem_assoc_different}
Let the elements $\alpha_1$, $\alpha_2$, and~$\alpha_3$ of the set~$A$ lie in pairwise distinct $G$-orbits. Then $(\alpha_1\bullet\alpha_2)\bullet\alpha_3=\alpha_1\bullet(\alpha_2\bullet\alpha_3)=(\alpha_1\bullet\alpha_3)\bullet\alpha_2$.
\end{lem}

\begin{proof}
If any of the elements~$\alpha_i$ lies in the $G$-orbit $G\varepsilon$ or the elements
$\alpha_i$ and $\alpha_j^{-1}$ lie in the same $G$-orbit for some $i\ne j$, then the required equality holds by Corollary ~\ref{cor_assoc_special}. Therefore, we can assume that none of the elements~$\alpha_i$ lies in the same $G$-orbit with either the element~$\varepsilon$ or any of the elements~$\alpha_j^{\pm 1}$, where $j\ne i$. Then it follows from Lemma~\ref{lem_diff_orb} that, first, none of the elements~$x_i$, $p_i$ and~$p_i'$, where $i=1,2,3$, lies either in~$V$ or in~$Vt$, and secondly, for $i\ne j$ none of the elements~$x_i$, $p_i$ and~$p_i'$ lies in the same $V$-orbit with any of the elements~$x_j$, $p_j$ and~$p_j'$.

Assume that the statement of the Lemma is false so that the products of elements~$\alpha_1$, $\alpha_2$ and~$\alpha_3$ in two of the indicated three orders give different results~$(z_1,r_1)\ne (z_2,r_2)$. Let $t*z_i=[r_i,r_i']$, $i=1,2$. Note that properties~(1)--(3) from Lemma~\ref{lem_triple_condition} are invariant under permutations of elements~$\alpha_1$, $\alpha_2$ and~$\alpha_3$, so these properties hold for each of the two triples $(z_1,r_1,r_1')$ and $(z_2,r_2,r_2')$.

Let $(y_i,q_i)=\alpha_j\bullet\alpha_k$ for each permutation~$i,j,k$ of numbers~$1,2,3$. Then $y_i\in x_j*x_k$ and $q_i\in t*y_i$. Let $x_j*x_k=[y_i,y_i']$ and $t*y_i=[q_i,q_i']$.

Assume first that $z_1\ne z_2$. According to property~(1) from Lemma~\ref{lem_triple_condition}, each of the elements~$z_1$ and~$z_2$ belongs to the multiset~$x_1*x_2*x_3$. However, since the elements~$x_1$, $x_2$, and~$x_3$ lie in pairwise distinct $V$-orbits, it follows from Lemma~\ref{lem_3decomp} that the multiset~$x_1*x_2*x_3$ consists of four distinct elements and $x_1*y_1$ , $x_1*y'_1$, $x_2*y_2$, $x_2*y'_2$, $x_3*y_3$, and $x_3*y'_3$ are exactly all six of its two-element subsets. Therefore, one of these six subsets coincides with the set $[z_1,z_2]$. Renumbering the elements~$\alpha_1$, $\alpha_2$, $\alpha_3$, we can assume that $[z_1,z_2]=x_1*h$, where $h$ is one of the elements~$y_1$ and~$y_1'$.

Now consider the multisets $x_i*p_j*p_k'$, where $i,j,k$ are permutations of numbers~$1,2,3$. According to the property~(1) from Lemma~\ref{lem_triple_condition}, each of these multisets contains the multiset~$[z_1,z_2]=x_1*h$. Therefore, by Lemma~\ref{lem_yyxxx} for each triple $(x_i,p_j,p_k')$ there is an element~$v\in V$ such that one of the elements of the triple is equal to one of the two elements~$vx_1$ and~$vh$, and the product of the other two elements of the triple contains the remaining of the two elements~$vx_1$ and ~$vh$. Moreover, by what has been proved, the elements~$x_2$, $p_2$, $p_2'$, $x_3$, $p_3$, $p_3'$ cannot lie in the orbit~$Vx_1$. In addition, both elements $p_1$ and~$p_1'$ cannot simultaneously lie in the orbit~$Vx_1$, since by Lemma~\ref{lem_ord4}(c) this would imply that $\ord t$ divides~$4$. Consider 3 cases.

1. \textit{Both elements $p_1$ и~$p_1'$ lie in the orbit~$Vh$.} Since $p_1$ and~$p_1'$ cannot simultaneously lie in the orbit~$Vx_1$, then $Vx_1\ne Vh$. Since $\ord t\notin\{1,2,4\}$, by Lemma ~\ref{lem_ord4}(b) we have $\ord x_1\in\{1,2,4\}$. None of the elements~$x_2$, $p_2$, $p_2'$, $x_3$, $p_3$, $p_3'$ lies in either~$Vx_1$ or $Vp_1=Vh$. Thus, from the Lemma~\ref{lem_yyxxx} for the triples $(x_i,p_j',p_k)$, $i=2,3$, it follows that each of the multisets $x_2*p_3$, $x_2*p_3'$, $x_3*p_2$ and~$x_3*p_2'$ intersects the orbit~$Vx_1$. Hence, by Lemma~\ref{lem_basic}, it follows that the multiset~$x_1*x_2$ intersects with each of the orbits~$Vp_3$ and~$Vp_3'$, and the multiset~$x_1*x_3$ with each of the orbits~$Vp_2$ and~$Vp_2'$. However, since the order of the element~$x_1$ divides~$4$, it follows from the Lemma~\ref{lem_ord4}(b) that each of the multisets~$x_1*x_2$ and~$x_1*x_3$ consists of two elements lying in the same $V$-orbit. Therefore, $p_2'=w_2p_2$ and $p_3'=w_3p_3$ for some elements $w_2,w_3\in V$. Since $t*x_i=[p_i,p_i']$, it now follows from Lemma~\ref{lem_ord4}(b) that the orders of the elements~$x_2$ and~$x_3$ also divide~$4$. Note that the element~$t$ belongs to the subgroup of the two-valued group~$X$ generated by the set~$V$ and the elements~$x_1$, $x_2$ and~$x_3$. Indeed, $t\in x_3*p_3$, and the element~$p_3$ lies in the same $V$-orbit with the elements of the multiset $x_1*x_2$. Hence, by Lemma~\ref{lem_ord4}(a) we get that the order of the element~$t$ divides~$4$, which is a contradiction.

2. \textit{One of the elements $p_1$ and ~$p_1'$ lies in the orbit~$Vx_1$ with the second one  belonging to the orbit~$Vh$.} Since $p_1$ and~$p_1'$ cannot simultaneously lie in the orbit~$Vx_1$, then $Vx_1\ne Vh$. Consider the case of $p_1\in Vx_1$ and $p_1'\in Vh$; the second case is completely similar. From the fact that $p_1=ux_1$ for some $u\in V$, it follows that $t\in x_1*p_1=[u,ux_1^2]$, hence $t=ux_1^2$ since $t\notin V$. Also, since $t\notin V$, we have $x_1^2\notin V$, which means that the order of the element~$x_1$ does not divide~$4$. Then, according to statement~(c) of Proposition ~\ref{propos_bullet}, all elements~$vx_1$, where $v$ runs through the group~$V$, are pairwise distinct. In particular, $vx_1\ne x_1$ under $v\ne e$ and $vx_1\ne ux_1$ under $v\ne u$. Moreover, none of the elements~$x_2$, $p_2$, $p_2'$, $x_3$, $p_3$, $p_3'$ lies either in~$Vx_1$ or in $Vp_1=Vh$. Therefore, from Lemma~\ref{lem_yyxxx} for the triples $(x_1,p_2,p_3')$, $(x_1,p_3,p_2')$, $(x_2,p_1,p_3')$ and~$(x_3,p_1,p_2')$ it follows that the element $h$ lies in each of the multisets~$p_2*p_3'$ and~$p_3*p_2'$, and the element~$uh$ lies in each of the multisets~$x_2*p_3'$ and~$x_3*p_2'$. Using Lemma~\ref{lem_basic}, we get from here that the multiset $h*p_2'$ contains the elements~$ux_3$ and~$p_3$, and the multiset $h*p_3'$ contains the elements~$ux_2$ and~$p_2$. If the equality $p_2=ux_2$ were true, then we would get that $t\in x_2*p_2=[u,ux_2^2]$, so $t=ux_2^2$ since $t\notin V$. But then it would follow from the equality $x_2^2=ut=x_1^2$ and Lemma~\ref{lem_equal_squares} that the elements~$x_1$ and~$x_2$ lie in the same $V$-orbit, which is not true. Hence $ux_2\ne p_2$ and thus $h*p_3'=[ux_2,p_2]$. Likewise, $ux_3\ne p_3$ and $h*p_2'=[ux_3,p_3]$. Then by Lemma~\ref{lem_z1z2} we get that the element~$h^2$ belongs to the multisets~$ux_2*p_2$ and~$ux_3*p_3$. However, the element~$ut$ also belongs to the same two multisets. Moreover, it follows from Lemma~\ref{lem_xy=zt} that $ux_2*p_2\ne ux_3*p_3$, since none of the elements~$x_2$ and~$p_2$ lie in the same $V$-orbit with any of the elements~$x_3$ and~$p_3$. Therefore, $h^2=ut=x_1^2$, which, in view of Lemma ~\ref{lem_equal_squares}, contradicts $Vh\ne Vx_1$.

3. \textit{At least one of the elements $p_1$ и~$p_1'$ belongs neither to~$Vx_1$ nor to~$Vh$}. We will assume that $p_1'\notin Vx_1\cup Vh$; the second case is similar. In the triples~$(x_2,p_3,p_1')$ and~$(x_3,p_2,p_1')$ none of the elements lies in the orbit~$Vx_1$; moreover, the element~$p_1'$ does not lie in the orbit~$Vh$. By Lemma~\ref{lem_yyxxx}, for these triples at least one of the elements~$x_2$ and~$p_3$ lies in~$Vh$ and at least one of the elements~$x_3$ and~$p_2$ lies in~$Vh$. However, none of the elements~$x_2$ and~$p_2$ lie in the same $V$-orbit with any of the elements~$x_3$ and~$p_3$. Therefore, either $x_2,p_2\in Vh$ and $x_3,p_3\notin Vh$, or $x_2,p_2\notin Vh$ and $x_3,p_3\in Vh$. These cases are similar to each other, so we will consider the first one. Let $x_2=uh$ and $p_2=vh$, where $u,v\in V$. It follows from the equality $p_2=uvx_2$ that $t\in x_2*p_2=[uv,uvx_2^2]$, so $t=uvx_2^2$ since $t\notin V$. Now recall that $h$ is either~$y_1$ or~$y_1'$ and hence $h\in x_2*x_3$. Therefore, $x_3\in x_2*h=[u,ux_2^2]=[u,vt]$, which contradicts the fact that the element~$x_3$ lies neither in~$V$ nor in~$Vt$.

The obtained contradiction proves that $z_1=z_2$. Denote this element simply by~$z$. Consider the triple of elements~$\beta_1=\tau\bullet\alpha_1$, $\beta_2=\alpha_2$, $\beta_3=\alpha_3$.
Lemma~\ref{lem_tab} implies that $\beta_i\bullet(\beta_j\bullet\beta_k)=\tau\bullet(\alpha_i\bullet(\alpha_j\bullet\alpha_k))$ for all permutations~$i,j,k$ of numbers~$1,2,3$. Therefore from the fact that the products of the elements $\alpha_1$, $\alpha_2$, and~$\alpha_3$ in two different orders are equal to~$(z,r_1)$ and~$(z,r_2)$ and the formula~\eqref{eq_te}, it follows that the products of the elements $\beta_1$ , $\beta_2$ and~$\beta_3$ in the same two orders are equal to~$\tau\bullet (z,r_1)=(r_1',z)$ and~$\tau\bullet (z,r_2)=(r_2',z)$ respectively, where $t*z=[r_1,r_1']=[r_2,r_2']$. The elements $\beta_1$, $\beta_2$, and~$\beta_3$ also lie in pairwise different $G$-orbits. Repeating for them the arguments above for the elements~$\alpha_1$, $\alpha_2$ and~$\alpha_3$ proving that $z_1=z_2$, we get that $r_1'=r_2'$, hence $r_1=r_2$
\end{proof}

\begin{lem}\label{lem_aaa}
For any element~$\alpha$ of the set~$A$ we have the following equality $(\alpha\bullet\alpha)\bullet(\alpha\bullet\alpha)=\alpha\bullet(\alpha\bullet(\alpha\bullet\alpha))$.
\end{lem}

\begin{proof}
Consider the three elements $\alpha$, $\beta=\alpha\bullet(\alpha\bullet\alpha)$ and~$\gamma=(\alpha\bullet\alpha)^{-1}$. Suppose first that $\alpha$ and~$\beta$ lie in different $G$-orbits and that $\beta$ and~$\gamma$ lie in different $G$-orbits. Then
\begin{equation}\label{eq_abc}
(\alpha\bullet\beta)\bullet\gamma=\alpha\bullet(\beta\bullet\gamma).
\end{equation}
Indeed, if $\alpha$, $\beta$, and~$\gamma$ belong to three pairwise distinct $G$-orbits, then ~\eqref{eq_abc} follows from Lemma~\ref{lem_assoc_different}, and if $\alpha$ and~$\gamma$ lie in the same $G$-orbit, then by Lemma~\ref{lem_assoc_compare} the equality ~\eqref{eq_abc} is equivalent to the valid equality $(\alpha\bullet\beta)\bullet\alpha=\alpha\bullet(\beta\bullet\alpha)$. Using equality~\eqref{eq_abc} and Lemma~\ref{lem_a-1ab}, we obtain
\begin{align*}
\bigl((\alpha\bullet\alpha)\bullet(\alpha\bullet\alpha)\bigr)\bullet (\alpha\bullet\alpha)^{-1}&=\alpha\bullet\alpha,\\
\bigl(\alpha\bullet(\alpha\bullet(\alpha\bullet\alpha))\bigr)\bullet (\alpha\bullet\alpha)^{-1}&=(\alpha\bullet\beta)\bullet\gamma=\alpha\bullet(\beta\bullet\gamma)\\
{}&=
\alpha\bullet\left(\bigl(\alpha\bullet(\alpha\bullet\alpha)\bigr)\bullet (\alpha\bullet\alpha)^{-1}\right)=\alpha\bullet\alpha.
\end{align*}
Applying Corollary~\ref{cor_bijection}, we obtain the statement of the Lemma.

It remains to consider the case when some of the elements~$\alpha$ and~$\gamma$ lies in the same $G$-orbit with the element~$\beta$.

Assume that the elements~$\alpha$ and $\beta=\alpha\bullet(\alpha\bullet\alpha)$ lie in the same $G$-orbit. Then by Lemma~\ref{lem_Gorbits} the element~$\alpha\bullet\alpha$ lies in the same $G$-orbit with the element~$\varepsilon$, hence the elements~$\alpha$ and~$\alpha^{-1}$ lie in the same $G$-orbit. In this case, the equality being proved follows from Corollary ~\ref{cor_assoc_special} applied to the triple $\alpha$, $\alpha$, $\alpha\bullet\alpha$.

Suppose now that the elements~$\beta=\alpha\bullet(\alpha\bullet\alpha)$ and~$\gamma=(\alpha\bullet\alpha)^{-1}$ lie in the same $G$-orbit. Then $\beta=v\tau^k\bullet \gamma$ for some $v\in V$, $k\in \Z$. Then by Lemma~\ref{lem_ab-1} $$\alpha\bullet\alpha =\gamma^{-1}=v\tau^k\bullet\beta^{-1}=v\tau^k\bullet \left(\alpha^{-1}\bullet(\alpha\bullet\alpha)^{-1}\right).$$
Applying Lemmas~\ref{lem_ab-1} and ~\ref{lem_a-1ab} and Corollary~\ref{cor_T}, we obtain
\begin{align*}
(\alpha\bullet\alpha)\bullet(\alpha\bullet\alpha)&=(\alpha\bullet\alpha)\bullet\left(v\tau^k\bullet \left(\alpha^{-1}\bullet(\alpha\bullet\alpha)^{-1}\right)\right)\\
{}&=v\tau^k\bullet\left((\alpha\bullet\alpha)\bullet \left(\alpha^{-1}\bullet(\alpha\bullet\alpha)^{-1}\right)\right)=v\tau^k\bullet\alpha^{-1},\\
\alpha\bullet(\alpha\bullet(\alpha\bullet\alpha))&=\alpha\bullet\left(v\tau^k\bullet(\alpha\bullet\alpha)^{-1}\right)=v\tau^k\bullet\left(\alpha\bullet\bigl(\alpha^{-1}\bullet\alpha^{-1}\bigr)\right)=v\tau^k\bullet\alpha^{-1},
\end{align*}
which completes the proof of the Lemma.
\end{proof}

\begin{lem}\label{lem_aab}
For any elements~$\alpha$ and~$\beta$ of the set~$A$, the equality $(\alpha\bullet\alpha)\bullet\beta=\alpha\bullet(\alpha\bullet\beta)$ holds.
\end{lem}

\begin{proof}
Suppose first that none of the elements~$\alpha$ and~$\beta^{-1}$ lie in the same $G$-orbit with the element~$\alpha\bullet\beta$. Using Lemmas ~\ref{lem_assoc_different}, \ref{lem_assoc_compare} and ~\ref{lem_a-1ab}, we obtain
$$
(\alpha\bullet(\alpha\bullet\beta))\bullet \beta^{-1}=\alpha\bullet((\alpha\bullet\beta)\bullet \beta^{-1})=\alpha\bullet\alpha.
$$
Multiplying both sides of this equality by~$\beta$ and once again using Lemma~\ref{lem_a-1ab}, we obtain the required equality~$\alpha\bullet(\alpha\bullet\beta)=(\alpha\bullet\alpha)\bullet\beta$.

It remains to consider the case when some of the elements~$\alpha$ and~$\beta^{-1}$ lies in the same $G$-orbit with the element~$\alpha\bullet\beta$.

Assume that the elements~$\alpha$ and~$\alpha\bullet\beta$ lie in the same $G$-orbit. Then it follows from formula~\eqref{eq_G_action} and Lemma~\ref{lem_a-1ab} that the elements $\varepsilon=\alpha^{-1}\bullet\alpha$ and $\beta=\alpha^{-1}\bullet(\alpha\bullet\beta)$ also lie in the same $G$-orbit. In this case, the statement to be proved immediately follows from Corollary ~\ref{cor_assoc_special}.

Suppose now that the elements~$\alpha\bullet\beta$ and~$\beta^{-1}$ lie in the same $G$-orbit. Then it follows from Lemmas~\ref{lem_Gorbits} and ~\ref{lem_a-1ab} that the element~$\alpha$ lies in the same $G$-orbit with the element~$\beta^{-1}\bullet\beta^{-1}$.
Therefore, it follows from Lemma ~\ref{lem_assoc_compare} that the required associativity property  for the triple $(\alpha,\alpha,\beta)$ is equivalent to the associativity property for the triple $\bigl(\beta^{-1}\bullet\beta^{-1},\beta^{-1}\bullet\beta^{-1},\beta\bigr)$. However, using Lemmas ~\ref{lem_aaa} and ~\ref{lem_a-1ab}, we get that
\begin{multline*}
\left((\beta^{-1}\bullet\beta^{-1})\bullet(\beta^{-1}\bullet\beta^{-1})\right)\bullet\beta=
\left(\bigl((\beta^{-1}\bullet\beta^{-1})\bullet\beta^{-1}\bigr)\bullet\beta^{-1}\right)\bullet\beta\\{}=
(\beta^{-1}\bullet\beta^{-1})\bullet\beta^{-1}=(\beta^{-1}\bullet\beta^{-1})\bullet\left((\beta^{-1}\bullet\beta^{-1})\bullet\beta\right).
\end{multline*}
The Lemma is proven.
\end{proof}

Now we can finally fully prove the associativity of multiplication~$\bullet$.

\begin{lem}\label{lem_assoc_always}
The equality $(\alpha_1\bullet\alpha_2)\bullet\alpha_3=\alpha_1\bullet(\alpha_2\bullet\alpha_3)$ holds for any elements $\alpha_1$, $\alpha_2$ and~$\alpha_3$ of the set~$A$.
\end{lem}

\begin{proof}
The required equality follows from Lemma~\ref{lem_assoc_different} if the elements $\alpha_1$, $\alpha_2$ and~$\alpha_3$ lie in pairwise distinct $G$-orbits, and from Lemmas~\ref{lem_aab} and ~\ref{lem_assoc_compare} if two of these three elements lie in the same $G$-orbit.
\end{proof}

Statement~(b) of Proposition ~\ref{propos_main_construction} follows from Corollary~\ref{cor_bijection} and Lemma~\ref{lem_assoc_always}; statement~(c) immediately follows from statements~(a) and~(b).

\section{Non-special two-valued groups consisting of elements of orders~$1$, $2$ and~$4$}
\label{section_124}

\begin{theorem}\label{thm_124}
Let $X$ be a finitely generated non-special involutive commutative two-valued group that does not contain a single-valued direct factor and consists entirely of elements of orders~$1$, $2$ and~$4$. Then $X$ is isomorphic to one of the groups in the following two series:
\begin{enumerate}
\item $X_{n\times 4}^{\ba}=C_4^n/\iota_{\ba}$, where $\iota_{\ba}(a)=a^{-1}$;
\item $X_n^{\bu}=(C_2^n\times C_2^n)/\iota_{\bu}$, where $\iota_{\bu}(a,b)=(a,ab)$.
\end{enumerate}
Each of the two-valued groups~$X_{n\times 4}^{\ba}$ and~$X_n^{\bu}$ consists of $2^{2n-1}+2^{n-1}$ elements. When $n\le 2$ these two two-valued groups are isomorphic, while for $n\ge 3$ they are not isomorphic.
\end{theorem}

Let $X$ be a two-valued group satisfying the conditions of Theorem ~\ref{thm_124}, and $V\subseteq X$ be the subgroup consisting of all elements of orders~$1$ and~$2$. For each element $x\in X\setminus V$, its order is ~$4$, so the order of its square~$x^2$ is ~$2$, i.e. $x^2\in V\setminus\{e\}$. For each element $v\in V$, denote by $X_v\subseteq X$ the subset consisting of all elements~$x$ such that $x^2=v$. Then $X_e=V$ and $X_v$ consists of elements of order~$4$ for $v\ne e$. According to Proposition ~\ref{propos_neras}, the fact that $X$ does not contain a single-valued direct factor implies that all subsets~$X_v$ are non-empty. Moreover, Lemma~\ref{lem_equal_squares}, Corollary~\ref{cor_z1z2sq}, and Lemma~\ref{lem_ord4}(b) immediately imply the following statement.

\begin{lem}\label{lem_124mult}
Each of the sets~$X_v$ is a $V$-orbit. For $v\ne e$ the stabilizer of the elements of this orbit is the two-element group~$\{e,v\}$. If $x\in X_u$, $y\in X_{v}$ and~$x*y=[z_1,z_2]$, then $z_1,z_2\in X_{uv}$ and $z_2=uz_1=vz_1$.
\end{lem}

\begin{cor}
If the two-valued group $X$ is finitely generated, then the Boolean group~$V$ is finite.
\end{cor}

\begin{proof}
Let $X$ be generated by the elements of $x_1,\ldots,x_N$ and let $x_i\in X_{v_i}$ under $i=1,\ldots,N$. Then any multiset of the form $x_{i_1}*\cdots*x_{i_k}$ consists of elements lying in the orbits $X_v$, where $v=v_{i_1}^{\pm1}\cdots v_{i_k}^{\pm 1}$. Hence, the union of orbits~$X_v$, where $v$ runs through the subgroup generated by the elements $v_1,\ldots,v_N$, is the entire two-valued group~$X$. Therefore, the elements $v_1,\ldots,v_N$ generate the whole group~$V$, so $\dim V \le N$.
\end{proof}

\begin{cor}
If $\dim V=n<\infty$, then the two-valued group $X$ consists of exactly $2^{2n-1}+2^{n-1}$ elements.
\end{cor}

In each of the subsets $X_v$, $v\in V$, we choose some representative~$x_v$, and as~$x_e$ we take the identity~$e$. Lemma ~\ref{lem_124mult} implies that
\begin{equation}\label{eq_phi_defin}
x_u*x_v=\bigl[\varphi(u,v)x_{uv},u\varphi(u,v)x_{uv}\bigr]=\varphi(u,v)[x_{uv},ux_{uv}]
\end{equation}
for some element~$\varphi(u,v)\in V$, satisfying
\begin{gather}
\varphi(u,v)\equiv \varphi(v,u) \pmod{\langle u,v\rangle},\label{eq_phi_cond1}\\
\varphi(u,u)\equiv e\pmod{\langle u\rangle},\label{eq_phi_cond2}\\
\varphi(e,u)\equiv\varphi(u,e)\equiv e\pmod{\langle u\rangle},\label{eq_phi_cond3}
\end{gather}
where by $\langle v_1,\ldots,v_k\rangle$ we denote the subgroup generated by the elements of $v_1,\ldots,v_k$.

Note that the condition~\eqref{eq_phi_defin} defines the element~$\varphi(u,v)$ not uniquely, but only up to multiplication by an element from the subgroup~$\langle u,v\rangle$ generated by~$u$ and~$v$. Indeed, multiplication by each of the elements~$u$ and~$v$ simply rearranges the elements of the multiset~$x_u*x_v$.

It is easy to see that multiplication in the two-valued group $X$ can be completely recovered from the constructed map $\varphi\colon V\times V\to V$. Let us now write down the condition that a map~$\varphi$ satisfying~\eqref{eq_phi_cond1}--\eqref{eq_phi_cond3} defines indeed an associative multiplication on~$X$. For arbitrary $u,v,w\in V$ we have
\begin{align*}
(x_u*x_v)*x_w&=\varphi(u,v)[x_{uv},ux_{uv}]*x_w\\
{}&=\varphi(u,v)\varphi(uv,w)[x_{uvw},ux_{uvw},vx_{uvw},wx_{uvw}]
\end{align*}
and
\begin{align*}
x_u*(x_v*x_w)&=x_u*\varphi(v,w)[x_{vw},vx_{vw}]\\
{}&=\varphi(v,w)\varphi(u,vw)[x_{uvw},ux_{uvw},vx_{uvw},wx_{uvw}].
\end{align*}
Since the element~$x_{uvw}$ is stabilized by the elements~$e$ and~$uvw$ and is not stabilized by any other elements of the group~$V$, the condition that these two multisets coincide is written as an inclusion
\begin{equation}\label{eq_cocycle}
\varphi(u,v)\varphi(uv,w)\varphi(u,vw)\varphi(v,w)\in\langle u,v,w\rangle
\end{equation}
If instead of a set of representatives~$\{x_v\}$ of sets~$X_v$ we choose another set of representatives~$\{x_v'\}$, where again $x_e'=e$, then $x_v'=\chi(v)x_v$ for some~$\chi(v)\in V$ such that $\chi(e)=e$. Then the map $\varphi$ will be replaced by the map given by the formula
\begin{equation*}
\varphi'(u,v)=\varphi(u,v)\chi(u)\chi(v)\chi(uv).
\end{equation*}
In view of the fact that, as noted above, the value~$\varphi(u,v)$ is defined up to multiplication by an element from the group~$\langle u,v\rangle$, it is more correct to say that~$\varphi$ will be replaced by some map~$\varphi'$ such that
\begin{equation}\label{eq_homological}
\varphi'(u,v)\equiv\varphi(u,v)\chi(u)\chi(v)\chi(uv)\pmod{\langle u,v\rangle}.
\end{equation}

We will call maps $\varphi\colon V\times V\to V$ satisfying the conditions ~\eqref{eq_phi_cond3} and ~\eqref{eq_cocycle} \textit{quasi-cocycles} on the Boolean group~$V$. The quasi-cocycle~$\varphi$ will be called \textit{symmetric} if it satisfies the condition ~\eqref{eq_phi_cond1}, and \textit{involutive} if it satisfies the condition ~\eqref{eq_phi_cond2}. We call two quasi-cocycles~$\varphi$ and~$\varphi'$ \textit{cohomologous} if there exists a map $\chi\colon V\to V$ such that~$\chi(e)=e$ and for all $u,v\in V$ the congruence~\eqref{eq_homological} holds. In particular, a quasi-cocycle is called \textit{cohomologically trivial} if it is cohomologous to the quasi-cocycle that is identically equal to one.

Since for the quasi-cocycle~$\varphi$ in formula~\eqref{eq_phi_defin} the values~$\varphi(u,v)$ are defined only up to multiplication by elements from~$\langle u,v\rangle$, it will be convenient to say that two quasi-cocycles~$\varphi$ and~$\varphi'$ on~$V$ are \textit{equivalent} if for all $u,v\in V$ there is a congruence
$
\varphi(u,v)\equiv \varphi'(u,v)\pmod{\langle u,v\rangle}.
$

The following Proposition immediately follows from the above.

\begin{propos}\label{propos_2gqc}
A finitely generated non-special involutive commutative two-valued group~$X$, which does not contain a single-valued direct factor and consists of elements of order~$1$, $2$ or ~$4$, is completely defined by a Boolean group~$V$ and an involutive symmetric quasi-cocycle $\varphi\colon V\times V\to V$. The two-valued groups corresponding to the pairs $(V,\varphi)$ and~$(V',\varphi')$ are isomorphic if and only if there is an isomorphism of the groups~$V$ and~$V'$ taking the quasi-cocycle~$\varphi$ into a quasi-cocycle cohomologous to the quasi-cocycle~$\varphi'$.
\end{propos}

The terminology described above has the following origin. Let $G$ be a group and~$A$ be an abelian group in additive notation. Recall the definition of cohomology groups $H^*(G;A)$, see ~\cite[Chapter~III, \S 1]{Bro87}. 

For each $k\ge 0$ consider the group $\mathcal{C}^k(G;A)$ consisting of the maps
$$
\varphi\colon \underbrace{G\times\cdots\times G}_{k}\to A
$$
such that $\varphi(g_1,\ldots,g_k)=0$ if at least one of the elements~$g_i$ is equal to~$e$. The differential $\delta\colon \mathcal{C}^k(G;A)\to \mathcal{C}^{k+1}(G;A)$ is defined by the formula
\begin{multline*}
(\delta\varphi)(g_1,g_2,\ldots,g_{k+1})=\varphi(g_2,\ldots,g_{k+1})+
\sum_{i=1}^k(-1)^i\varphi(g_1,\ldots,g_ig_{i+1},\ldots,g_{k+1})\\{}+(-1)^{k+1}\varphi(g_1,\ldots,g_k).
\end{multline*}
The maps~$\varphi$ that lie in the kernel of this mapping are called $k$-\textit{cocycles} on the group~$G$ with values in~$A$. The cohomology group $H^k(G;A)$ is by definition the quotient group of the group of $k$-cocycles on~$G$ with values in~$A$ by the subgroup~$\delta\mathcal{C}^{k-1}(G;A)$.

If the group~$A$ is Boolean and we use the multiplicative notation for it, then the equation of the $2$-cocycle can be written as
\begin{equation}\label{eq_cocycle_strict} 
\varphi(u,v)\varphi(uv,w)\varphi(u,vw)\varphi(v,w)=e
\end{equation}
for all $u,v,w\in G$. Moreover, two cocycles~$\varphi$ and~$\varphi'$ are cohomologous if and only if there exists a map $\chi \colon G\to A$ such that $\chi(e)=e$ and
\begin{equation*}
\varphi'(u,v)=\varphi(u,v)\chi(u)\chi(v)\chi(uv)
\end{equation*}
for all $u,v\in G$.
Thus, the above definitions for quasi-cocycles are weaker versions of these conditions.

Proposition ~\ref{propos_2gqc} reduces our classification problem of finitely generated non-special involutive commutative two-valued groups that do not contain single-valued direct factors and consist of elements of orders~$1$, $2$ and~$4$ to the problem of classification of involutive symmetric quasi-cocycles on Boolean groups up to cohomology. On the set~$\CH(V)$ of cohomology classes of involutive symmetric quasi-cocycles on~$V$ the multiplication operation defines the structure of an abelian group. Our task is to calculate this group.
Before solving this problem, we note that for genuine cocycles satisfying the equation ~\eqref{eq_cocycle_strict} exactly, a similar problem is trivial.

\begin{lem}\label{lem_cocycle}
Let $V$ and~$W$ be Boolean groups and $\varphi\colon V\times V\to W$ be an involutive symmetric cocycle, that is, a map satisfying the cocycle equation~\eqref{eq_cocycle_strict} and the equations
\begin{gather}
\varphi(u,v)= \varphi(v,u),\label{eq_cond1_strict}\\
\varphi(u,u)= \varphi(e,u)=\varphi(u,e)=e.\label{eq_cond2_strict}
\end{gather}
Then $\varphi$ is cohomologically trivial, that is, $\varphi=\delta\chi$ for some $\chi\colon V\to W$ such that $\chi(e)=e$.
\end{lem}
\begin{proof}
Obviously, it suffices to prove the Lemma in the case of $W=\F_2$. (In this case, of course, we will write ~$\F_2$ in additive notation.) Let $\dim V=n$ and $b_1,\ldots,b_n$ be a basis of the Boolean group~$V$ as a vector space over~$\F_2$. Then
$$
H^*(V;\F_2)=H^*\left((\mathbb{RP}^{\infty})^n;\F_2\right)=\F_2[t_1,\ldots,t_n],
$$
where $t_1,\ldots,t_n$ is the basis of the space $H^1(V;\F_2)=\Hom(V,\F_2)$ dual to the basis $b_1,\ldots,b_n$. Therefore, the cohomology class of the cocycle~$\varphi$ has the form $\sum\lambda_{i}t_i^2+\sum_{i< j}\mu_{ij}t_it_j$ for some $\lambda_i,\mu_{ij}\in\F_2$. It immediately follows from the definition of multiplication in cohomology that $\lambda_i=\varphi(b_i,b_i)$ and $\mu_{ij}=\varphi(b_i,b_j)+\varphi(b_j,b_i)$ Thus, from equations ~\eqref{eq_cond1_strict} and ~\eqref{eq_cond2_strict} it follows that the cohomology class of the cocycle~$\varphi$ is zero.
\end{proof}

Let us now give a non-trivial example of an involutive symmetric quasi-cocycle. Let $\B=\{b_i\}_{i\in\Lambda}$ be a basis of a Boolean group~$V$ (as a vector space over~$\F_2$). Then each element of the group~$V$ can be uniquely written in the form $b_I=\prod_{i\in I}b_i$, where $I\subseteq \Lambda$ is a finite subset; in particular, $b_{\varnothing}=e$. Let's put
\begin{equation}\label{eq_phiB}
\varphi_{\B}(b_I,b_J)=b_{I\cap J}.
\end{equation}
The conditions ~\eqref{eq_phi_cond1}--\eqref{eq_cocycle} can be verified directly, so that $\varphi_{\B}$ is indeed involutive symmetric quasi-cocycle on~$V$. Our main result about the group~$\CH(V)$ is as follows.

\begin{propos}\label{propos_quasi}
\textnormal{(1)} If $\dim V\le 2$, then every involutive symmetric quasi-cocycle on~$V$ is trivial, and thus the group~$\CH(V)$ is trivial.

\textnormal{(2)} If $3\le\dim V<\infty$, then the quasi-cocycle~$\varphi_{\B}$ defined by the formula ~\eqref{eq_phiB} is cohomologically nontrivial. In this case every involutive symmetric quasi-cocycle on~$V$ is cohomologous to either the trivial quasi-cocycle or the quasi-cocycle~$\varphi_{\B}$. (In particular, quasi-cocycles~$\varphi_{\B}$ corresponding to different bases~$\B$ are cohomologous to each other.) Thus, $\CH(V)$ is the cyclic group of order~$2$.
\end{propos}

We will postpone the proof of this Proposition until the next section, and now we will show how Theorem ~\ref{thm_124} follows from it. To do this, we calculate the quasi-cocycles~$\varphi$ corresponding to the two-valued groups~$X^{\ba}_{n\times 4}$ and~$X^{\bu}_n$.

\smallskip

1. Consider the two-valued group $X=X^{\bu}_{n}=(C_2^{n}\times C_2^{n})/\iota_{\bu}$, where $\iota_{\bu}(a,b)=(a,ab)$. The group of elements of order~$2$ in the two-valued group~$X$ is exactly the subgroup $V=\{e\}\times C_2^n\subset C_2^n\times C_2^n$ consisting of the elements fixed by the involution~$\iota_{\bu}$. Identifying $v$ with $(e,v)$, we identify~$C_2^n$ with~$V$. For any $v\in V=C_2^n$, the subset $X_v\subset X$ consists of the images under the projection $\pi\colon C_2^n\times C_2^n\to X$ of all elements of the form $(v,w)$, where $w\in C_2^n$. Let us choose a representative~$x_v=\pi(v,e)=\pi(v,v)$ in this subset. Then
$$
x_u*x_v=[x_{uv},ux_{uv}],
$$
therefore, the two-valued group~$X^{\bu}_n$ corresponds to the trivial quasi-cocycle on the group~$C_2^n$.

\smallskip

2. Consider the two-valued group $X=X^{\ba}_{n\times 4}=C_4^n/\iota_{\ba}$, where $\iota_{\ba}(a)=a^{-1}$. Let $\{a_i\}_{i\in\Lambda}$ be the system of standard generators of the direct product~$C_4^n$. The group of order-$2$ elements in~$X$ is exactly the subgroup $V= C_2^n\subset C_4^n$ consisting of the elements fixed by the involution~$\iota_{\ba}$. The basis $\B=\{b_i\}_{i\in\Lambda}$ of the Boolean group~$V$ consists of the elements~$b_i=a_i^2$. An arbitrary element of the group~$V$ has the form $b_I=\prod_{i\in I}b_i$, where $I\subseteq \Lambda$ is a finite subset. As a representative~$x_{b_I}$ of the subset $X_{b_I}\subset X$ we take the image of the element $a_I=\prod_{i\in I}a_i$ under the projection $\pi\colon C_4^n\to X$. We have $a_Ia_J=b_{I\cap J}a_{I\Delta J}$, $a_I^{-1}a_J=b_{I\setminus J}a_{I\Delta J}$, where $I\Delta J$ is the symmetric difference between the sets~$I$ and~$J$. Hence
$$
x_{b_I}*x_{b_J}=[b_{I\cap J}x_{b_{I\Delta J}},b_{I\setminus J}x_{b_{I\Delta J}}]=b_{I\cap J}[x_{b_{I\Delta J}},b_Ix_{b_{I\Delta J}}].
$$
Thus, the two-valued group~$X_{n\times 4}^{\ba}$ corresponds to the quasi-cocycle~$\varphi_{\B}$ on the group~$C_2^n$.

\smallskip

Theorem ~\ref{thm_124} immediately follows from the above calculation and Propositions ~\ref{propos_2gqc} and ~\ref{propos_quasi}. Recall also that an explicit isomorphism between two-valued groups~$X^{\ba}_{4,4}$ and~$X^{\bu}_2$ was described in the Example~\ref{ex_iso}.

\section{Quasi-cocycles on Boolean groups}\label{section_quasi}

In this section we will prove the Proposition ~\ref{propos_quasi}. When $\dim V\le 2$ from the fact that $\langle u,v\rangle =V$ for any pair of distinct nontrivial elements~$u,v\in V$, it immediately follows that every involutive symmetric quasi-cocycle on~$V$ is equivalent and, therefore, cohomologous to the trivial one. Thus, statement~(1) of Proposition~\ref{propos_quasi} is proved.

We divide the proof of statement~(2) into several Lemmas.

\begin{lem}\label{lem_dim3}
Let $V$ be the Boolean group of $8$ elements, that is, a three-dimensional vector space over~$\mathbb{F}_2$. Then $\CH(V)$ is a cyclic group of order two. If $\B$ is a basis of the Boolean group~$V$, then the cohomology class of the quasi-cocycle~$\varphi_{\B}$ is a generator of the group~$\CH(V)$.
\end{lem}

\begin{proof}
Let $\varphi$ be an involutive symmetric quasi-cocycle on~$V$. If $u$ and~$v$ are distinct nontrivial elements of~$V$, then there are only two cosets of the group~$V$ by the subgroup~$\langle u,v\rangle$. Set $\lambda(u,v)=0$ if $\varphi(u,v)\in \langle u,v\rangle$, and $\lambda(u,v)=1$ if $\varphi(u,v)\notin \langle u,v\rangle$. Then $\lambda(u,v)=\lambda(v,u)$. The cocycle condition~\eqref{eq_cocycle} for the triple~$(u,v,uv)$ will be written as $\lambda(u,v)=\lambda(v,uv)$, which implies that the value~$\lambda(u,v)$ depends only on the two-dimensional subspace $\langle u,v\rangle\subset V$ and not on the specific choice of its basis~$u,v$. Moreover, two-dimensional subspaces of the vector space~$V$ are in a natural one-to-one correspondence with non-trivial elements of the dual vector space $V^*=\Hom(V,\F_2)$. Therefore, we get a well-defined map $\lambda\colon V^*\setminus\{e\}\to \F_2$. It is easy to see that the equivalent quasi-cocycles~$\varphi$ lead to identical maps~$\lambda$ and the quasi-cocycle~$\varphi$ can be uniquely recovered from the map~$\lambda$ up to equivalence.

Let us now prove that to any mapping $\lambda\colon V^*\setminus\{e\}\to \F_2$ there corresponds an equivalence class of quasi-cocycles. To do this, we first calculate the mapping $\lambda_{\B}$ corresponding to the quasi-cocycle~$\varphi_{\B}$, where $\B=\{b_1,b_2,b_3\}$ is some basis of the Boolean group~$V$. Direct calculation shows that $\lambda_{\B}(\xi_{\B})=1$, where $\xi_{\B}\in V^*$ is an element such that $\xi_{\B}(b_1)=\xi_{\B}(b_2)=\xi_{\B}(b_3)=1$, and $\lambda_{\B}(\xi)=0$ for all $\xi\ne\xi_{\B}$. The quasi-cocycles on~$V$ form a group with respect to pointwise multiplication, and the product of quasi-cocycles corresponds to the pointwise addition of the corresponding maps~$\lambda$. Thus, it is easy to see that by multiplying quasi-cocycles~$\varphi_{\B}$ for different~$\B$, we can get as $\lambda$ any map~$\lambda\colon V^*\setminus\{e\}\to \F_2$. Since the set~$V^*\setminus\{e\}$ consists of $7$ elements, there are exactly $128$ equivalence classes of quasi-cocycles on~$V$.

Let us now find out how the cohomological equivalence of two cocycles is rewritten in terms of maps~$\lambda$. For map $\lambda\colon V^*\setminus\{e\}\to \F_2$ we set
$$
\sigma(\lambda)=\sum_{\xi\in V^*\setminus\{e\}}\lambda(\xi).
$$
Let us prove that the quasi-cocycles corresponding to the maps~$\lambda$ and~$\lambda'$ are cohomologous if and only if $\sigma(\lambda)=\sigma(\lambda')$.

For two elements $a,b\in V\setminus\{e\}$, consider the map $\chi_{a,b}\colon V\to V$ such that $\chi_{a,b}(a)=b$ and $\chi_{a,b}(u)=e$ under $u\ne a$. Consider two quasi-cocycles cohomologous to each other by means of the map~$\chi_{a,b}$:
$$
\varphi'(u,v)=\varphi(u,v)\chi_{a,b}(u)\chi_{a,b}(v)\chi_{a,b}(uv),
$$
and compare the maps corresponding to them $\lambda,\lambda'\colon  V^*\setminus\{e\}\to \F_2$. Direct calculation shows that
$$
\lambda'(\xi)=\left\{
\begin{aligned}
&\lambda(\xi)+1,&&\text{if $\xi(a)=0$ and $\xi(b)=1$},\\
&\lambda(\xi),&&\text{otherwise}. 
\end{aligned}
\right.
$$
Thus, if $a=b$ then the quasi-cocycles~$\varphi$ and~$\varphi'$ are equivalent and $\lambda=\lambda'$, and if $a\ne b$ then $\lambda'(\xi)\ne \lambda(\xi)$ for exactly two elements $\xi\in V^*\setminus\{e\}$. Therefore, firstly, $\sigma(\lambda')=\sigma(\lambda)$, and secondly, by choosing a pair~$(a,b)$, we can change the map~$\lambda$ on any two predefined elements of the set~$V^*\setminus\{e\}$. Since any map $\chi\colon V\to V$ can be represented as a pointwise product of mappings of the form~$\chi_{a,b}$ for different pairs~$(a,b)$, we obtain that the equivalence classes of the quasi-cocycles corresponding to the maps $\lambda,\lambda'\colon V^*\setminus\{e\}\to\F_2$ lie in the same cohomology class if and only if~$\sigma(\lambda)=\sigma(\lambda')$. Therefore, $\CH(V)$ is the cyclic group of order~$2$. Since $\sigma(\lambda_{\B})=1$ for any basis~$\B$ of the group~$V$, all quasi-cocycles~$\varphi_{\B}$ are cohomologous and their cohomology class is generating the group ~$\CH(V)$.
\end{proof}

Let $V$ be a Boolean group and $\varphi\colon V\times V\to V$ be a quasi-cocycle. If $U\subset V$ is a subgroup and $\Pi\colon V\to U$  is a projection, then
the map
$$
U\times U\subset V\times V\xrightarrow{\varphi}V\xrightarrow{\Pi} U
$$
we will call \textit{restriction} of the quasi-cocycle~$\varphi$ to the subgroup~$U$ along the projection~$\Pi$. It is easy to see that this map is a quasi-cocycle on~$U$. Moreover, if the original quasi-cocycle was involutive and symmetric, then the same will be true for the resulting quasi-cocycle on~$U$. The operation of taking restriction along the projection~$\Pi$ induces a homomorphism $\Pi_*\colon\CH(V)\to\CH(U)$.

\begin{lem}\label{lem_restrict_3_sur}
Let $V$ be a finite Boolean group such that $\dim V> 3$, and $\Pi\colon V\to U$ be a projection onto some subgroup $U\subset V$ with $\dim U=3$. Then the homomorphism $\Pi_*\colon \CH(V)\to\CH(U)$ is surjective.
\end{lem}

\begin{proof}
Let $\B$ be a basis of the Boolean group~$V$, part of which is a basis~$\mathcal{A}$ of the subgroup~$U$. Then the restriction of the quasi-cocycle~$\varphi_{\B}$ along~$\Pi$ coincides with the quasi-cocycle~$\varphi_{\mathcal{A}}$. This immediately implies the surjectivity of the homomorphism~$\Pi_*$, since according to Lemma~\ref{lem_dim3} the group~$\CH(U)$ is generated by the cohomology class of the quasi-cocycle~$\varphi_{\mathcal{A}}$.
\end{proof}

\begin{lem}\label{lem_restrict_3_inj}
Let $V$ be a finite Boolean group and $\Pi\colon V\to U$ be a projection onto some subgroup of $U\subset V$ with $\dim U\ge 3$. Then the homomorphism $\Pi_*\colon \CH(V)\to\CH(U)$ is injective.
\end{lem}

\begin{proof}
Note that it suffices to prove the Lemma in the case when the subgroup~$U$ has codimension~$1$ in~$V$ (if we consider~$V$ as a vector space over~$\F_2$); then the general case will automatically follow by induction. Denote by~$c$ the generator of the one-dimensional kernel of the projection~$\Pi$. Then $V=U\times\langle c\rangle$.

Let $\varphi$ be an involutive symmetric quasi-cocycle on~$V$ and $\psi$ be its restriction to~$U$ along~$\Pi$. We need to prove that if $\psi$ is cohomologically trivial, then~$\varphi$ is also cohomologically trivial.

It follows from the cohomological triviality of the quasi-cocycle~$\psi$ that there exists a map $\xi\colon U\to U$ such that $\xi(e)=e$ and $\varphi(u_1,u_2)\xi(u_1)\xi(u_2)\xi(u_1u_2)\in \langle u_1,u_2,c\rangle$ for arbitrary $u_1,u_2\in U$. We set $\chi(u)=\xi(u)$ and $\chi(uc)=\xi(u)\varphi(u,c)$ for all $u\in U$ and consider the involutive symmetric quasi-cocycle~$\varphi'$ on~$V$ given by the formula
$$\varphi'(v,w)=\varphi(v,w)\chi(v)\chi(w)\chi(vw).$$
Then the quasi-cocycle~$\varphi'$ is cohomologous to~$\varphi$ and
\begin{equation}\label{eq_lambda}
\varphi'(u_1,u_2)\equiv c^{\lambda(u_1,u_2)}\pmod{\langle u_1,u_2\rangle},\qquad u_1,u_2\in U,
\end{equation}
where $\lambda\colon U\times U\to\F_2$ is a map satisfying $\lambda(u_1,u_2)=\lambda(u_2,u_1)$ and $\lambda(u,u)=\lambda(u,e)=0$. Moreover, $\varphi'(u,c)\in\langle u,c\rangle$ for all $u\in U$.

It follows from condition~\eqref{eq_cocycle} for triples $(c,c,u)$, where $u\in U$, that $\varphi'(uc,c)\in\langle u,c\rangle$.

Now, from the condition~\eqref{eq_cocycle} for triples $(c,u_1,u_2)$, where $u_1,u_2\in U$, it follows that $\varphi'(u_1c,u_2)\in\langle u_1,u_2,c\rangle$. Consequently,
\begin{equation}\label{eq_mu}
\varphi'(u_1c,u_2)\equiv c^{\mu(u_1,u_2)}\pmod{\langle u_1c,u_2\rangle},
\end{equation}
where $\mu\colon U\times U\to\F_2$ is a map that satisfies $\mu(e,u)=\mu(u,e)=0$.

Finally, from the condition ~\eqref{eq_cocycle} for triples $(u_1,c,u_2c)$, where $u_1,u_2\in U$, it follows that $\varphi'(u_1c,u_2c)\in\langle u_1,u_2,c\rangle$. Consequently,
\begin{equation}\label{eq_nu}
\varphi'(u_1c,u_2c)\equiv c^{\nu(u_1,u_2)}\pmod{\langle u_1c,u_2c\rangle},
\end{equation}
where $\nu\colon U\times U\to\F_2$ is a map satisfying $\nu(u_1,u_2)=\nu(u_2,u_1)$ and $\nu(u,u)=\nu(u,e)=0$.

Now the condition ~\eqref{eq_cocycle} for the triples $(u_1,u_2,u_3)$, $(u_1c,u_2,u_3)$ and~$(u_1c,u_2c,u_3)$, where $u_1,u_2, u_3\in U$, can be written as the system of equations
\begin{align}
\lambda(u_2,u_3)+\lambda(u_1u_2,u_3)+\lambda(u_1,u_2u_3)+\lambda(u_1,u_2)&=0,&&\label{eq_lmn1}\\
\lambda(u_2,u_3)+\mu(u_1u_2,u_3)+\mu(u_1,u_2u_3)+\mu(u_1,u_2)&=0&&\text{for }u_1\notin\langle u_2,u_3\rangle,\label{eq_lmn2}\\
\mu(u_2,u_3)+\lambda(u_1u_2,u_3)+\nu(u_1,u_2u_3)+\nu(u_1,u_2)&=0&&\text{for }u_1\notin\langle u_3\rangle, u_2\notin\langle u_3\rangle.\label{eq_lmn3}
\end{align}
The restrictions $u_1\notin\langle u_2,u_3\rangle$ in ~\eqref{eq_lmn2} and $u_1\notin\langle u_3\rangle$, $u_2\notin\langle u_3\rangle$ in ~\eqref{eq_lmn3} are essential. They are related to the fact that exactly under these restrictions the element~$c$ does not belong to the subgroup generated by the triple of elements under consideration, and therefore we can write down the condition that the degree of occurrence of the element~$c$ in the left hand side of the corresponding congruence~\eqref{eq_cocycle} is equal to zero.

Equation~\eqref{eq_lmn1} is exactly the usual cocycle condition for~$\lambda$.
Lemma ~\ref{lem_cocycle} implies that there is a map $\eta\colon U\to\F_2$ such that $\eta(e)=0$ and
\begin{equation}\label{eq_lambdaeta}
\lambda(u_1,u_2)=\eta(u_1)+\eta(u_2)+\eta(u_1u_2).
\end{equation}

If the elements $u,u'\in U$ are distinct and none of them is equal to~$e$, then equation~\eqref{eq_lmn3} for the triple $(u_1,u_2,u_3)=(u,u,u')$ gives
\begin{equation}\label{eq_munu}
\mu(u,u')=\nu(u,uu').
\end{equation}
Note that if one of the elements~$u$ and~$u'$ is equal to~$e$, then the equality ~\eqref{eq_munu} also holds, because $\mu(u,e)=\mu(e,u)=e$ and $\nu(u,u)=\nu(u,e)=\nu(e,u)=e$. However, the condition $u\ne u'$ is essential.

Now suppose that the elements of $u_1,u_2,u_3\in U$ are such that $u_1\notin\langle u_2,u_3\rangle$. Then each of the pairs~$(u_1u_2,u_3)$, $(u_1,u_2u_3)$ and~$(u_1,u_2)$ consists of distinct elements. Therefore, substituting ~\eqref{eq_lambdaeta} and ~\eqref{eq_munu} into ~\eqref{eq_lmn2}, we get the equation
\begin{equation*}
\nu(u_1u_2,u_1u_2u_3)+\nu(u_1,u_1u_2u_3)+\nu(u_1,u_1u_2)=\eta(u_2) +\eta(u_3)+\eta(u_2u_3).
\end{equation*}
Making change $a_1=u_1$, $a_2=u_1u_2$, $a_3=u_1u_2u_3$ in this equation and introducing a new function
\begin{equation}\label{eq_rhonu}
\rho(u,u')=\nu(u,u')+\eta(uu'),
\end{equation}
we get that
\begin{equation}\label{eq_rho}
\rho(a_1,a_2)+\rho(a_2,a_3)+\rho(a_3,a_1)=0
\end{equation}
provided that~$a_1$ is not equal to any of the elements~$e$, $a_1a_2$, $a_2a_3$ and~$a_3a_1$. The last condition means exactly that none of the elements~$a_1$, $a_2$, $a_3$ and~$a_1a_2a_3$ is equal to~$e$. Note that the symmetry of the function~$\nu$ immediately implies the symmetry of the function~$\rho$. Moreover, $\rho(u,u)=\nu(u,u)+\eta(e)=0$ for all $u\in U$.

Until now, we have never used the key condition~$\dim U\ge 3$, without which the Lemma is false. Let us prove that under this condition the equality ~\eqref{eq_rho} holds for all triples~$(a_1,a_2,a_3)$ of elements of the set~$U\setminus\{e\}$ without the assumption that $a_1a_2a_3\ne e$. Indeed, let $a_1$, $a_2$ and~$a_3$ be the elements in~$U\setminus\{e\}$ such that $a_1a_2a_3=e$. Then $\dim\langle a_1,a_2,a_3\rangle=2$. Therefore, it follows from the condition~$\dim U\ge 3$ that there is an element $b\in U$, which does not lie in the subgroup~$\langle a_1,a_2,a_3\rangle$. Then, as it was proved, the equation ~\eqref{eq_rho} holds for each of the triples~$(a_1,a_2,b)$, $(a_2,a_3,b)$ and~$(a_3,a_1,b)$. Adding these equations, we obtain the equation~\eqref{eq_rho} for the initial triple~$(a_1,a_2,a_3)$.

From the fact that the equation ~\eqref{eq_rho} holds for all triples~$(a_1,a_2,a_3)$ of elements of the set~$U\setminus\{e\}$, it immediately follows that there is a function $\zeta\colon U\setminus\{e\}\to\F_2$ such that
\begin{equation}\label{eq_rhozeta}
\rho(u,u')=\zeta(u)+\zeta(u'), \qquad u,u'\ne e.
\end{equation}
Indeed, by choosing an arbitrary element~$u_0\in U\setminus\{e\}$, we can put~$\zeta(u)=\rho(u_0,u)$ for all $u\in U\setminus\{e\}$. Then equality~\eqref{eq_rhozeta} will follow from equation~\eqref{eq_rho}. From ~\eqref{eq_munu}, \eqref{eq_rhonu} and ~\eqref{eq_rhozeta} we get
\begin{align}
\mu(u,u')&=\zeta(u)+\eta(u')+\zeta(uu'),&  u&\ne u',\label{eq_muzeta}\\
\nu(u,u')&=\zeta(u)+\zeta(u')+\eta(uu'),&  u,u'&\ne e.\label{eq_nuzeta}
\end{align}
We extend the function~$\zeta$ to the entire ~$U$ by setting $\zeta(e)=0$. Note, however, that formulae~\eqref{eq_rhozeta} and ~\eqref{eq_nuzeta} will not necessarily be true if one of the elements of~$u$ and~$u'$ is equal to~$e$, and the same holds for formula~\eqref{eq_muzeta}  if $u=u'$ .

Consider the map $\theta\colon V\to V$ and the involutive symmetric quasi-cocycle $$\varphi''\colon V\times V\to V$$ defined by the formulae
\begin{gather*}
\begin{aligned}
\theta(u)&=c^{\eta(u)},&u&\in U,\\
\theta(uc)&=c^{\zeta(u)},&u&\in U,
\end{aligned}\\
\varphi''(v,w)=\varphi'(v,w)\theta(v)\theta(w)\theta(vw),\qquad v,w\in V.
\end{gather*}
It follows from formulae~\eqref{eq_lambda}, \eqref{eq_mu}, \eqref{eq_nu}, \eqref{eq_lambdaeta}, \eqref{eq_muzeta} and ~\eqref{eq_nuzeta} that
\begin{equation}\label{eq_phi''}
\varphi''(v,w)\in \langle v,w\rangle
\end{equation}
for all $v,w\in V$, that is, the quasi-cocycle~$\varphi''$ is equivalent to the trivial one. Note that in the exceptional cases of $(v,w)=(uc,u)$ and $(v,w)=(uc,c)$, where the formulae ~\eqref{eq_muzeta} and ~\eqref{eq_nuzeta} may not be true, the inclusion~\eqref{eq_phi''} still holds, since in these cases $c\in \langle v,w\rangle$. Since $\varphi''$ is cohomologous to~$\varphi$ by construction, we obtain the statement of the Lemma.
\end{proof}

Statement~(2) of Proposition ~\ref{propos_quasi} follows from Lemmas ~\ref{lem_dim3}, ~\ref{lem_restrict_3_sur} and ~\ref{lem_restrict_3_inj}.

\section{Special two-valued groups}
\label{section_special}

In this section, we will prove the following theorem.

\begin{theorem}\label{theorem_special_classify}
Any special involutive commutative two-valued group that does not contain single-valued direct factors is isomorphic to a two-valued group~$Y_V$, where $V$ is a Boolean group with $\dim V\ge 2$.
\end{theorem}

\begin{lem}\label{lem_Y2contain}
If $(x,y)$ is a special pair in an involutive commutative two-valued group~$X$ and $x*y=[z,z]$, then $x^2=y^2=z^2$ and $5$ elements $e$, $x$, $y$, $z$ and~ $s=x^2$ are pairwise distinct and form a subgroup of the two-valued group~$X$ with multiplication table
\begin{gather*}
x*x=y*y=z*z=[e,s],\quad s*s=[e,e],\\
\begin{aligned}
x*y&=[z,z],& y*z&=[x,x],& z*x&=[y,y],\\
s*x&=[x,x],& s*y&=[y,y],& s*z&=[z,z],
\end{aligned}
\end{gather*}
that is, a subgroup isomorphic to a two-valued group~$Y_2$. In particular, $x$, $y$ and~$z$ are the elements of order~$4$.
\end{lem}
\begin{proof}
Let $x*y=[z,z]$. By Lemma ~\ref{lem_basic} we have $y*z=[x,x']$ for some~$x'$. Then
\begin{align*}
x*y*z&=(x*y)*z=[z,z]*z=[e,e,z^2,z^2],\\
x*y*z&=x*(y*z)=x*[x,x']=[e,x^2,x*x'].
\end{align*}
However, we know that $x^2\ne e$ because $\ord x>2$. So $x^2=z^2$. In addition, we see that the identity~$e$ is contained in the multiset~$x*x'$, which implies that $x'=x$. Thus, $y*z=[x,x]$. Similarly we have $y^2=x^2$ and $x*z=[y,y]$. Let us set $s=x^2=y^2=z^2$, so $x*x=y*y=z*z=[e,s]$.

We have
\begin{align*}
x*y*y&=(x*y)*y=[z,z]*y=[x,x,x,x],\\
x*y*y&=x*(y*y)=x*[e,s]=[x,x,x*s],
\end{align*}
so $x*s=[x,x]$ and hence $x^3=x$. Since $\ord x>2$, it follows that $\ord x=4$. Hence $\ord s=2$, that is, $s*s=[e,e]$. Similarly to the equality $x*s=[x,x]$ it is proved that $y*s=[y,y]$ and $z*s=[z,z]$. Thus $5$ elements of $e,s,x,y,z$ form a subgroup of the two-valued group~$X$ isomorphic to~$Y_2$.
\end{proof}

It follows from this Lemma that if $(x,y)$ is a special pair and $x*y=[z,z]$, then $(x,z)$ and $(y,z)$ are also special pairs and $x*z=[y,y]$ and $y*z=[x,x]$. Such triples of elements $(x,y,z)$ we will call \textit{special triples}. Elements of a special triple always have orders~$4$ and $x^2=y^2=z^2$ holds for them.

\begin{propos}\label{propos_orders124}
Let $X$ be a special involutive commutative two-valued group. Then

\textnormal{(a)} all elements of the two-valued group~$X$ have orders $1$, $2$ or~$4$;

\textnormal{(b)} all elements of order~$4$ of the two-valued group~$X$ have the same square, i.e. there exists an element $s\in X$ such that $x^2=s$ for all elements of $x\in X$ of order~$4$.
\end{propos}

The proof of this Proposition will be divided into several Lemmas.

\begin{lem}\label{lem_4not4}
Assume that in an involutive commutative two-valued group~$X$, the order of element~$x$ is equal to~$4$, while the order of element~$q$ is neither~$1$, nor~$2$, nor~$4$. Set $s=x^2$. Then there is a sequence~$\{x_n\}_{n\in\Z}$ of elements of the two-valued group~$X$ such that $x_0=x$ and 
\begin{align}
q^k*x_n&=[x_{n-k},x_{n+k}],& k,n&\in\Z,\label{eq_qk*xn}\\
s x_n&=x_{-n},& n&\in\Z,\label{eq_dxn}\\
x*x_1&=[q,s q],&&\label{eq_xx1}\\ x_1*x_1&=[e,s q^2].\label{eq_x1x1}&&
\end{align}
Moreover, such a sequence is unique up to reversal of the sign of the index~$n$, that is, up to the simultaneous permutation $x_n\leftrightarrow x_{-n}$ for all~$n$.
At the same time, $x_1\ne x_{-1}$ and the order of each of the elements~$x_1$ and~$x_{-1}$ is neither~$1$, nor~$2$, nor~$4$.
\end{lem}

\begin{proof}
Let us set $x_0=x$. Since $\ord q\ne 4$, Lemma ~\ref{lem_Y2contain} implies that the pair $(q,x)$ is not special. Since the orders of ~$q$ and ~$x$ are greater than ~$2$, it follows that $q*x=[x_1,x_{-1}]$ for some elements $x_1,x_{-1}\in X$ such that $x_1\ne x_{-1}$.

The element $s=x^2$ has order~$2$ and we have the following equalities
$$
[x*x_1,x*x_{-1}]=x*(x*q)=(x*x)*q=[e,s]*q=[q,q,sq,sq].
$$
In addition we have
$$
s(x*x_1)=(s x)*x_1=x*x_1.
$$
Hence
$$
x*x_1=x*x_{-1}=[q,s q].
$$
From the fact that $q\in x*x_1$ it follows that the element $q^2$ belongs to the multiset
$$
(x*x_1)*(x*x_1)=(x*x)*(x_1*x_1)=[e,s]*[e,x_1^2]=[e,e,s,s,x_1^2,x_1^2,s x_1^2,s x_1^2].
$$
If the order of element~$x_1$ were equal to~$1$, $2$ or~$4$, then the order of element~$x_1^2$ would be equal to~$1$ or $2$ and the multiset~$(x*x_1)*(x*x_1)$ would consist of elements of orders~$1$ and~$2$. But this cannot be true since the order of element~$q^2$ is neither ~$1$ nor~$2$. So $\ord x_1\notin\{1,2,4\}$. Similarly, $\ord x_{-1}\notin\{1,2,4\}$.

Now it follows from Proposition~\ref{propos_bullet}(c) that $s x_1\ne x_1$ and $s x_{-1}\ne x_{-1}$. However, at the same time
$$
[s x_1, s x_{-1}]=s [x_1,x_{-1}]=s (x*q)=(s x)*q=x*q=[x_1,x_{-1}].
$$
Hence $s x_1=x_{-1}$ and $s x_{-1}=x_1$.

Equalities~\eqref{eq_qk*xn} with $k=1$ have the form
\begin{equation}\label{eq_q_recur}
q*x_n=[x_{n-1},x_{n+1}].
\end{equation}
These equalities allow to define recursively uniquely the elements of the sequence~$\{x_n\}$, starting from the already defined elements~$x_0$, $x_1$ and~$x_{-1}$. Indeed, the elements~$x_n$ with positive indices are defined consequently in ascending order of ~$n$. By construction, for $n\ge 1$, the element $x_{n}$ lies in the multiset $q*x_{n-1}$, therefore, by Lemma~\ref{lem_basic}, the multiset~$q*x_n$ does indeed contain the element~$x_{n-1}$ and we can and must take the second element of this multiset as~$x_{n+1}$. Similarly, the elements~$x_n$ with negative indices are uniquely defined consequently in descending order of~$n$.

Let us prove that $q^k*x_n=[x_{n+k},x_{n-k}]$ by induction on~$k$. For $k=0$ and $k=1$ this is true by construction. For $k\ge 2$ we will prove this, assuming that the statement is true for~$k-2$ and~$k-1$. We have
\begin{align*}
q*q^{k-1}*x_n&=q*(q^{k-1}*x_n)=q*[x_{n+k-1},x_{n-k+1}]=[x_{n+k},x_{n+k-2}, x_{n-k+2},x_{n-k}],\\
q*q^{k-1}*x_n&=(q*q^{k-1})*x_n=[q^{k-2},q^k]*x_n=[x_{n+k-2},x_{n-k+2},q^k*x_n],
\end{align*}
so $q^k*x_n=[x_{n+k}, x_{n-k}]$.

It suffices to prove equality ~\eqref{eq_dxn} for $n\ge 0$, and this is also done by induction on $n$. Indeed, the induction base for $n=0$ and $n=1$ is true. The inductive step is obtained by the action of the element ~$s$ on the equality ~\eqref{eq_q_recur}: if $s x_{n-1}=x_{-n+1}$ and $s x_n=x_{-n}$, then
$$
[x_{-n+1},x_{-n-1}]=q*x_{-n}=q*(s x_n)=s(q*x_n)=s [x_{n+1},x_{n-1}]=[s x_{n+1},x_{-n+1}],
$$
so $s x_{n+1}=x_{-n-1}$.

Let us now prove the equalities ~\eqref{eq_xx1} and ~\eqref{eq_x1x1}.

Since $q*x=[x_1,s x_1]$, $x_1$ lies in the multisets $q*x$ and $(s q)*x$. By Lemma~\ref{lem_basic}, this implies that the elements~$q$ and~$sq$ lie in the multiset~$x*x_1$. However, $\ord q\ne 4$, so according to Proposition~\ref{propos_bullet}(c) $q\ne sq$. Hence $x*x_1=[q,sq]$.

Finally, we have
$$
(s q^2)*x_1=q^2*(s x_1)=q^2*x_{-1}=[x_1,x_{-3}],
$$
so, by Lemma~\ref{lem_basic}, it follows that the element~$s q^2$ lies in the multiset~$x_1*x_1$. Since $\ord q\ne 4$, then $\ord q^2\ne 2$ and hence $s q^2\ne e$. Therefore, $x_1*x_1=[e,s q^2]$.\end{proof}

\begin{lem}\label{lem_special_sequences}
Let $(x,y,z)$ be a special triple in~$X$, $s=x^2=y^2=z^2$ and $q\in X$ be an element whose order is neither~$1$, nor~$2$, nor~$4$. Let $\{x_n\}_{n\in\Z}$, $\{y_n\}_{n\in\Z}$ and~$\{z_n\}_{n\in\Z}$ be sequences satisfying the conditions of Lemma ~\ref{lem_4not4} for the pairs $(x,q)$, $(y,q)$ and~$(z,q)$ respectively. Then for all $n\in\Z$ the following equalities hold:
\begin{align*}
x_n*y&=[z_n,s z_n],&
y_n*z&=[x_n,s x_n],&
z_n*x&=[y_n,s y_n],\\
x*y_n&=[z_n,s z_n],&
y*z_n&=[x_n,s x_n],&
z*x_n&=[y_n,s y_n].
\end{align*}
In addition, changing if necessary $z_{n}\leftrightarrow z_{-n}$ for all $n\in\Z$ simultaneously, we can ensure that the following equalities hold:
\begin{align*}
x_1*y_1&=[z,z_2],& 
x_2*y_1=x_1*y_2&=[z_1,z_3],\\
y_1*z_1&=[x,x_2],&
y_2*z_1=y_1*z_2&=[x_1,x_3],\\
z_1*x_1&=[y,y_2],&
z_2*x_1=z_1*x_2&=[y_1,y_3].
\end{align*}
\end{lem}

\begin{proof}
We have
\begin{align*}
q^n*x*y&=(q^n*x)*y=[x_n,sx_n]*y=[x_n*y,s(x_n*y)],\\
q^n*x*y&=q^n*(x*y)=q^n*[z,z]=[z_n,s z_n,z_n,s z_n],
\end{align*}
so $x_n*y=[z_n,s z_n]$. The formulae for $y_n*z$, $z_n*x$, $x*y_n$, $y*z_n$ and~$z*x_n$ are proved similarly.

We have
\begin{align*}
(x*y)*(q*q)&=[z,z]*[e,q^2]=[z,z,z,z,z_2,z_2,sz_2,sz_2],\\
(q*x)*(q*y)&=[x_1,s x_1]*[y_1,s y_1]=[x_1*y_1,x_1*y_1,s (x_1*y_1),s (x_1*y_1)].
\end{align*}
Since $s z=z$ it follows that either $x_1*y_1=[z,z_2]$ or $x_1*y_1=[z,s z_2]=[z,z_{-2}]$. In the second case we need to make a permutation of $z_n\leftrightarrow z_{-n}$ simultaneously for all~$n$. Thus, we achieve the equality $x_1*y_1=[z,z_2]$. Then
\begin{align*}
q*x_1*y_1&=(q*x_1)*y_1=[x,x_2]*y_1=[z_1,s z_1,x_2*y_1],\\
q*x_1*y_1&=q*(x_1*y_1)=q*[z,z_2]=[z_1,s z_1,z_1,z_3],
\end{align*}
so $x_2*y_1=[z_1,z_3]$. Similarly, $x_1*y_2=[z_1,z_3]$.

Similarly, we get that either $y_1*z_1=[x,x_2]$, or $y_{1}*z_1=[x,x_{-2}]$. However, now we can no longer do the permutation $x_{n}\leftrightarrow x_{-n}$ since it would spoil the already proved equality~$x_1*y_1=[z,z_2]$. Therefore, we need to prove that the equality $y_1*z_1=[x,x_2]$ holds in another way. Note that in the case when $x_2=x_{-2}$ this equality has already been proved. Therefore, we only need to consider the case $x_2\ne x_{-2}$. In this case $x_2\ne x$. Indeed, if the equality $x=x_2$ were true, then the equality $x_{-2}=s x_2=s x=x=x_2$ would hold. Now note that it follows from the already proven equalities $x*y_1=[z_1,s z_1]$ and $x_2*y_1=[z_1,z_3]$ and Lemma ~\ref{lem_basic} that the elements~$x$ and~$x_2$ lie in the multiset~$y_1*z_1$. Since $x_2\ne x$, this implies the required equality $y_1*z_1=[x,x_2]$. Similarly we have $z_1*x_1=[y,y_2]$.

The formulae for $y_2*z_1$, $y_1*z_2$, $z_2*x_1$, $z_1*x_2$ are now proved similarly to the formulae for~$x_2*y_1$ and~$x_1*y_2$.
\end{proof}

\begin{lem}\label{lem_special_square}
Let $(x,y)$ be a special pair in an involutive commutative group~$X$. Then there is no element~$q$ in $X$ such that $q^2=x$.
\end{lem}

\begin{proof}
Assume that such an element~$q$ exists. Let $x*y=[z,z]$ and $s=x^2=y^2=z^2$ and $\{x_n\}_{n\in\Z}$, $\{y_n\}_{n\in\Z}$ and~$\{z_n\}_{n\in\Z}$ be sequences satisfying the conditions of Lemma ~\ref{lem_special_sequences}.

We have
$$
[x_2,s x_2]=q^2*x=x*x=[e,s],
$$
so either $x_2=e$ or $x_2=s$. Therefore, either $x_2*y_1=[y_1,y_1]$ or $x_2*y_1=[s y_1,s y_1]$. However, by Lemma ~\ref{lem_special_sequences}, we have $x_2*y_1=[z_1,z_3]$, which means that either $y_1=z_1$, or $s y_1=z_1$, that is, $y_1=z_{-1}$. In each of these cases, we get that the multiset $q*y_1$ contains the elements~$y$ and~$z$. Since $y\ne z$, we have $q*y_1=[y,z]$. Finally,
\begin{multline*}
[y_1,y_1,x*y_1]=[e,x]*y_1=(q*q)*y_1=q*(q*y_1)=q*[y,z]={} \\ [y_1,y_{-1},z_1,z_{-1}]=[y_1,y_{-1},y_1,y_{-1}],
\end{multline*}
so $x*y_1=[y_{-1},y_{-1}]$. However, by the formula from Lemma ~\ref{lem_special_sequences}, we have $x*y_1=[y_{1},y_{-1}]$. We obtain a contradiction since $y_1\ne y_{-1}$, which completes the proof of the Lemma.
\end{proof}

\begin{lem}\label{lem_order3}
A special involutive commutative two-valued group~$X$ cannot contain elements of order~$3$.
\end{lem}

\begin{proof}
Assume the opposite: let $q\in X$ be an element of order~$3$, then $q^2=q$, that is, $q*q=[e,q]$. Let $(x,y,z)$ be some special triple in~$X$, $s=x^2=y^2=z^2$ and $\{x_n\}_{n\in\Z}$, $\{y_n\}_{n\in\Z}$ and~$\{z_n\}_{n\in\Z}$ be the sequences satisfying the conditions of Lemma ~\ref{lem_special_sequences}.
We have
$$
[x_2,x_{-2}]=q^2*x=q*x=[x_1,x_{-1}].
$$
So either $x_2=x_1$ or $x_2=x_{-1}=s x_1$. If the equality $x_2=x_1$ were true, then we would get that $q*x_1=[x_1,x]$ and, therefore,
$$
[x_1,x_1,x_1,x]=[e,q]*x_1=(q*q)*x_1=q*(q*x_1)=q*[x,x_1]=[x_1,x_{-1},x_1,x].
$$
This is not true since $x_{1}\ne x_{-1}$. Hence $x_2=s x_1$. Likewise, $y_2=s y_1$ and $z_2=s z_1$.

Now, using the equalities $x_1*x_1=[e,s q]$, $x_1*z=[y_1,s y_1]$, $x_1*y_1=[z,z_2]$ and~$x_1*z_1=[y,y_2]$, we get
\begin{align*}
(x_1*x_1)*y_1&=[e,s q]*y_1=[y_1,y_1,s y_2,s y]=[y_1,y_1,y_1,y],\\
x_1*(x_1*y_1)&=x_1*[z,z_2]=x_1*[z,s z_1]=[y_1,s y_1,s y,s y_2]=[y_1,y_{-1},y,y_1],
\end{align*}
which gives a contradiction since $y_1\ne y_{-1}$. This completes the proof of the Lemma.
\end{proof}

\begin{proof}[Proof of Proposition~\ref{propos_orders124}\textnormal{(a)}]
Assume that the two-valued group~$X$ contains an element~$q$ whose order is neither~$1$, nor~$2$, nor~$4$. According to Lemma~\ref{lem_order3}, the order of $q$ is not equal to~$3$; thus $\ord q>4$. Let $(x,y,z)$ be some special triple in~$X$, $s=x^2=y^2=z^2$ and $\{x_n\}_{n\in\Z}$, $\{y_n\}_{n\in\Z}$ and~$\{z_n\}_{n\in\Z}$ be the sequences satisfying the conditions of Lemma ~\ref{lem_special_sequences}. Using the properties of these sequences, we obtain
\begin{align*}
q*y_1*x_2&=q*(y_1*x_2)=q*[z_1,z_3]=[z,z_2,z_2,z_4],\\
q*y_1*x_2&=(q*y_1)*x_2=[y,y_2]*x_2=[z_2,s z_2,x_2*y_2],
\end{align*}
which implies that element~$s z_2$ coincides with one of the three elements~$z$, $z_2$ and~$z_4$.

If the equality $s z_2=z_2$ were true, then we would get that $q^2*z=[z_2,z_2]$. Since $\ord q^2>2$ and $\ord z=4$, it would follow that $(q^2,z)$ is a special pair, which is impossible by Lemma ~\ref{lem_special_square}. Therefore, $s z_2\ne z_2$

If the equality $s z_2=z$ were true, then we would again get that
$$
s z_2=z=s z=z_2,
$$
which, as we have already shown, is impossible. So $s z_2\ne z$.

Thus, $s z_2=z_4$. Then
$$
[z_3,z_5]=q*z_4=q*(s z_2)=s (q*z_2)=s [z_1,z_3]=[s z_1,s z_3].
$$
If $z_3=s z_3$, then $q^3*z=[z_3,z_3]$. Hence, either $\ord q^3\le 2$ or $(q^3,z)$ is a special pair, and then by Lemma~\ref{lem_Y2contain} the order of element~$q^3$ is equal to~$4$. Thus, in any case, the order of element~$q^3$ divides~ $4$, and hence the order of element~$q$ divides~$12$. On the other hand, the order of the element $q$ is larger than $4$, so $\ord q=6$, or $\ord q=12$. Then either $q^2$ or $q^4$ has order~$3$, which is impossible by Lemma~\ref{lem_order3}.
Hence $z_3\ne s z_3$, so $z_3= s z_1$. Then
$$
[z_2,z_4]=q*z_3=q*(s z_1)=s (q*z_1)=s [z,z_2]=[z,s z_2]=[z,z_4],
$$
so $z=z_2$. But then again $s z_2=z_2$, which, as we have already proved, is impossible. The resulting contradiction completes the proof of part~(a) of the Proposition.
\end{proof}

\begin{lem}\label{lem_44}
Let $x$ and~$p$ be two elements of order~$4$ in an involutive commutative two-valued group~$X$ such that $x^2\ne p^2$. Then there is an element $a\in X$ of order~$4$ such that $a^2=x^2 p^2$ and
\begin{align*}
x*p&=[a,x^2 a],&
x*a&=[p,x^2 p],&
p*a&=[x,p^2 x].
\end{align*}
\end{lem}

\begin{proof}
Since $x^2\ne p^2$ it follows from Lemma ~\ref{lem_Y2contain} that the pair $(x,p)$ is not special. So $x*p=[a,b]$ for some $a\ne b$. Lemma ~\ref{lem_basic} now implies that $p\in x*a$. At the same time, on the one hand,
$$
x^2(x*a)=(x^2x)*a=x*a,
$$
and on the other hand $x^2 p\ne p$ (by Proposition~\ref{propos_bullet}(c) since $p^2\ne x^2$). Hence $x*a=[p,x^2 p]$. Similarly $p*a=[x,p^2 x]$.

Now we have
\begin{align*}
(x*x)*(a*a)&=[e,x^2]*[e,a^2]=[e,e,x^2,x^2,a^2,a^2,x^2 a^2,x^2 a^2],\\
(x*a)*(x*a)&=[p,x^2 p]*[p,x^2 p]=[e,e,x^2,x^2,p^2,p^2,x^2p^2,x^2 p^2],
\end{align*}
so either $a^2=p^2$, or $a^2=x^2p^2$. Similarly, from consideration of the product $p*p*a*a$ it follows that either $a^2=x^2$, or $a^2=x^2 p^2$. Since $x^2\ne p^2$ it follows that $a^2=x^2 p^2$.

Since $x^2$ and $p^2$ are different elements of order~$2$, we get that $\ord a^2=2$, hence $\ord a=4$. Moreover, we see that $a^2\ne x^2$ since $p^2\ne e$. Hence, by Proposition ~\ref{propos_bullet}(c) $x^2 a\ne a$. However,
$$
[x^2 a, x^2 b]=x^2(x*p)=(x^2 x)*p=x*p=[a,b],
$$
hence $b=x^2 a$.
\end{proof}

\begin{proof}[Proof of Proposition~\ref{propos_orders124}\textnormal{(b)}]
Let $(x,y,z)$ be some special triple in~$X$ and $s=x^2=y^2=z^2$. Assume that there is an element $p\in X$ of order~$4$ such that $p^2=t\ne s$. Applying Lemma ~\ref{lem_44} to pairs $(x,p)$, $(y,p)$ and~$(z,p)$, we get that in the group~$X$ there are elements~$a$, $b$ and~$c$ of order~$4$ such that
\begin{gather*}
a^2=b^2=c^2=st,\\
\begin{aligned}
x*p&=[a,s a],&
x*a&=[p,s p],&
p*a&=[x,t x],\\
y*p&=[b,s b],&
y*b&=[p,s p],&
p*b&=[y,t y],\\
z*p&=[c,s c],&
z*c&=[p,s p],&
p*c&=[z,t z].
\end{aligned}
\end{gather*}
We have
\begin{align*}
x*y*p&=(x*y)*p=[z,z]*p=[c,c,s c,s c],\\
x*y*p&=x*(y*p)=x*[b,s b]=[x*b,s (x*b)].
\end{align*}
Moreover, by Lemma~\ref{lem_Y2contain}, the pair~$(x,b)$ is not special, since $x^2\ne b^2$. Hence, the multiset $x*b$ consists of two different elements. Hence $x*b=[c,s c]$. Now it follows from Lemma~\ref{lem_basic} that the element $x$ belongs to the multiset~$b*c$. Hence, the element $t x=st x=b^2 x$ also belongs to the multiset
$$
b^2(b*c)=(b^2  b)*c=b*c.
$$
At the same time, $t x\ne x$ by Proposition ~\ref{propos_bullet}(c) since $x^2\ne t$. Thus, $b*c=[x,t x]$. Finally, we get that
\begin{multline*}
[c,c,c,c]=[e,c^2]*c=[e,b^2]*c=(b*b)*c=b*(b*c)=b*[x,t x]={}\\
[c,s c,t c,st c]=[c,s c,s c,c],
\end{multline*}
so $s c=c$, which is impossible according to Proposition~\ref{propos_bullet}(c) since $c^2\ne s$. Thus, $p^2=s$ for all elements of $p\in X$ of order~$4$.
\end{proof}

\begin{proof}[Proof of Theorem~\ref{theorem_special_classify}]
According to Proposition~\ref{propos_orders124}, all elements of the two-valued group~$X$ have orders~$1$, $2$ or~$4$ and, moreover,  there is an element~$s$ of order~$2$ in~$X$ such that $s=x^2$ for all elements~$x\in X$ of order~$4$. Since $X$ does not contain single-valued direct factors, it follows from Proposition~\ref{propos_neras} that $s$ is the only element of order~$2$ in~$X$. Thus, $x^2=s$ for all elements $x\in X$, except~$e$ and~$s$. The elements~$e$ and~$s$ form a two-element subgroup of the two-valued group~$X$. Consider the two-valued quotient group $V=X/\{e,s\}$. We have $sx=x$ for all elements~$x$, except~$e$ and~$s$, so $V$ is simply the set~$X$, in which two elements~$e$ and~$s$ are identified in one element~$e$ and no other identifications made. In the two-valued group $V$ the equality $x^2=s$ turns into the equality~$x^2=e$. Thus, all nontrivial elements of the two-valued group~$V$ have order~$2$. According to Proposition~\ref{propos_bullet} we get that $V$ is a Boolean group with doubled multiplication, so $X\cong Y_V$.
\end{proof}

\section{End of classification in finitely generated case}\label{section_final_proof}

In this section, we collect together the results of sections ~\ref{section_1gen}--\ref{section_special} and prove Theorem ~\ref{theorem_main_fg}.

It follows from Proposition~\ref{propos_rasch_iso} and Theorems~\ref{theorem_nonspecial_classify}, \ref{thm_124} and ~\ref{theorem_special_classify} that every finitely generated involutive commutative two-valued group~$X$ is isomorphic to one of the following two-valued groups~$X^{\ba}_A\times C_2^m$, where $A$ is finitely generated abelian group, $X^{\bu}_n\times C_2^m$ or ~$Y_n\times C_2^m.$ Taking into account the isomorphism $X_A^{\ba}\times C_2^m\cong X_{A\times C_2^m}$, we see that $X$ is isomorphic to one of the two-valued groups listed in Theorem ~\ref{theorem_main_fg}.

Now we need to show that all these two-valued groups are pairwise non-isomorphic with the following exceptions:
\begin{gather*}
X^{\ba}_{m\times 2,4}\cong X^{\bu}_1\times C_2^m\cong Y_1\times C_2^m,\label{eq_main_iso1'}\\
X^{\ba}_{m\times 2,4,4}\cong X^{\bu}_2\times C_2^m.\label{eq_main_iso2'}
\end{gather*}
The first of these exceptional isomorphisms is obvious; the second follows from the isomorphism $X^{\ba}_{4,4}\cong X^{\bu}_2$, see example ~\ref{ex_iso} and Theorem ~\ref{thm_124}.

Two-valued groups are divided into two types: special and non-special, and two-valued groups of different types cannot be isomorphic to each other. According to Proposition~\ref{propos_special_yn}, two-valued groups $Y_n\times C_2^m$ with $n\ge 2$ are special, and all other two-valued groups listed in theorem ~\ref{theorem_main_fg} are non-special.

The special two-valued group $Y_n\times C_2^m$ consists of $2^m(2^n+1)$ elements, so these groups are pairwise non-isomorphic for different pairs $(n,m)$.

The two-valued group $X_n^{\bu}\times C_2^m$ consists of $2^m(2^{2n-1}+2^{n-1})$ elements, so these groups for different pairs $(n,m)$ are also pairwise non-isomorphic.

To complete the proof of Theorem ~\ref{theorem_main_fg}, it remains for us to prove the following two Propositions.

\begin{propos}\label{propos_ua_different}
If $n\ge 3$ and $m\ge 0$, then the two-valued group $X_n^{\bu}\times C_2^m$ is not isomorphic to any of the two-valued groups~$X^{\ba}_{A}$ of the principal series.
\end{propos}

\begin{propos}\label{propos_different}
Let $A$ and~$B$ be finitely generated abelian groups such that the two-valued groups~$X_A^{\ba}$ and~$X_B^{\ba}$ are isomorphic. Then the groups~$A$ and~$B$ are also isomorphic.
\end{propos}

We start with the following simple Lemma, which will be used in the proofs of both Propositions.

\begin{lem}\label{lem_orders}
Let $A$ be an abelian group and $\pi\colon A\to X_A^{\ba}$ be the quotient map by the antipodal involution~$\iota_{\ba}$. Then $\ord \pi(a)=\ord a$ for all $a\in A$.
\end{lem}

\begin{proof}
By the definition of a coset two-valued group, we have $$\pi(a^k)*\pi(a)=\bigl[\pi(a^{k-1}),\pi(a^{k+1})\bigr]$$ for all $k$. From this equality it follows easily by induction that
$\pi(a)^k=\pi(a^k)$ for all $a\in A$ and $k>0$. Since $\ord \pi(a)$ is by definition the least~$k$ such that $\pi(a)^k=e$ and the inverse image of the identity of the two-valued group $X_A^{\ba}$ under the map ~$\pi$ consists only of the identity of the group~$A$, then $\ord \pi(a)=\ord a$.
\end{proof}

\begin{proof}[Proof of Proposition~\ref{propos_ua_different}]
The two-valued group~$X_n^{\bu}\times C_2^m$ consists of elements of orders~$1$, $2$ and~$4$. Therefore, the isomorphism $X_n^{\bu}\times C_2^m\cong X^{\ba}_A$ can hold only if all elements of the two-valued group~$X^{\ba}_A$ have orders~$1$, $2$ and~$4$. According to Lemma~\ref{lem_orders}, this is equivalent to the fact that all elements of the abelian group~$A$ have orders~$1$, $2$ and~$4$, that is, $A\cong C_4^p\times C_2^q$ for some $p,q\ge 0$. Then the two-valued group~$X^{\ba}_A$ consists of $2^q(2^{2p-1}+2^{p-1})$ elements, whence $p=n$ and $q=m$. Thus, the only two-valued group of the principal series that could be isomorphic to the two-valued group~$X_n^{\bu}\times C_2^m$ is the two-valued group $X^{\ba}_{C_4^n\times C_2^m}\cong X^{\ba}_{n\times 4}\times C_2^m$. It is easy to see that in the two-valued groups $X^{\ba}_{n\times 4}$ and~$X^{\bu}_n$ all elements of order~$2$ are squares. Therefore by Proposition ~\ref{propos_neras} these two two-valued groups do not contain single-valued direct factors. Hence by Proposition ~\ref{propos_rasch_iso} the isomorphism $X_n^{\bu}\times C_2^m\cong X^{\ba}_{n\times 4}\times C_2^m$ would imply the isomorphism $X_n^{\bu}\cong X^{\ba}_{n\times 4}$. It remains to note that by Theorem~\ref{thm_124}, the two-valued groups~$X_n^{\bu}$ and~$X^{\ba}_{n\times 4}$ are not isomorphic when $n\ge 3$.
\end{proof}

Denote by~$N_k(A)$ (respectively, $N_k(X)$) the number of elements of order~$k$ in the abelian group~$A$ (respectively, in the involutive commutative two-valued group~$X$).
The following well-known property of finite abelian groups is easily deduced from their classification theorem.

\begin{lem}
If finite abelian groups~$A$ and~$B$ are such that $N_k(A)=N_k(B)$ for all~$k$, then $A$ and~$B$ are isomorphic.
\end{lem}

\begin{cor}\label{cor_ab_diff}
If finitely generated abelian groups~$A$ and~$B$ are such that $N_k(A)=N_k(B)$ for all~$k$ and the ranks of the free parts of the Abelian groups~$A$ and~$B$ are equal to each other, then $A$ and ~$B$ are isomorphic.
\end{cor}

\begin{lem}\label{lem_orders_compare}
If $A$ is a finitely generated abelian group, then $N_2(X_A^{\ba})=N_2(A)$ and~$N_k(X_A^{\ba})=N_k(A)/2$ for all $k>2$.
\end{lem}

\begin{proof}
Under the projection $\pi\colon A\to X^{\ba}_A$ each pair $\{a,a^{-1}\}$ of mutually inverse elements of the group~$A$ is glued into one element of the two-valued group~$X^{\ba}_A$. In addition $a=a^{-1}$ if $\ord a = 2$, and $a\ne a^{-1}$ if $\ord a >2$. Therefore the claim immediately follows from Lemma ~\ref{lem_orders}.
\end{proof}

\begin{lem}\label{lem_ranks_compare}
Let $A$ be a finitely generated abelian group. Then the rank of the free part of the group~$A$ is equal to the largest cardinality of the system of elements $x_1,\ldots,x_r\in X_A^{\ba}$ such that none of the multisets of $x_1^{n_1}*\cdots*x_r^{n_r}$, where $n_1,\ldots,n_r\in\Z$ and at least one of the numbers~$n_i$ is non-zero, contains the identity ~$e$.
\end{lem}

\begin{proof}
If $a_1,\ldots,a_r$ are elements of the group~$A$ and $x_i=\pi(a_i)$, $i=1,\ldots,r$, then the multiset $x_1^{n_1}*\cdots*x_r^{n_r}$ consists of $2^{r-1}$ elements $\pi\left(a_1^{n_1} a_2^{\pm n_2}\cdots a_r^{\pm n_r}\right)$ corresponding to all possible choices of signs~$\pm$.
Therefore, the condition that none of these elements is equal to~$e$ if at least one of the numbers~$n_i$ is non-zero is equivalent to the condition of the absence of a non-trivial relation $ a_1^{m_1}\cdots a_r^{m_r}=e$.
\end{proof}

\begin{proof}[Proof of Proposition~\ref{propos_different}]
According to Lemmas ~\ref{lem_orders_compare} and ~\ref{lem_ranks_compare}, the numbers $N_k(A)$ and the rank of the free part of a finitely generated abelian group~$A$ are uniquely determined by the structure of the two-valued group~$X_A^{\ba}$. Therefore Proposition~\ref{propos_different} follows from Corollary~\ref{cor_ab_diff}.
\end{proof}

Thus, Theorem ~\ref{theorem_main_fg} is proved.

\section{Non-finitely generated two-valued groups}\label{section_non_fg}

Let us now discuss to what extent our classification theorem can be extended to the case of non-finitely generated involutive commutative two-valued groups. There are two natural formulations of the problem: we can consider two-valued groups without topology or topological two-valued groups. The first problem will be considered in this section, the second one in sections~\ref{section_top} and~\ref{section_lc}. Naturally, in neither case can we expect a complete classification theorem such as Theorem ~\ref{theorem_main_fg} in finitely generated case. Indeed, in the non-finitely generated case, there is no classification theorem even for single-valued commutative groups. Therefore, the best result we can hope for is the solution of the classification problem for involutive commutative two-valued groups `modulo' the classification of single-valued commutative groups .

The main result of this section is the following theorem.

\begin{theorem}\label{theorem_non_fg}
Every involutive commutative two-valued group is isomorphic to one of the two-valued groups
\begin{itemize}
\item[\textnormal(1)] $X^{\ba}_A$, where $A$ is an abelian group;
\item[\textnormal(2)] $X^{\bu}_V\times W$, where $V$ and~$W$ are Boolean groups and $\dim V\ge 3$;
\item[\textnormal(3)] $Y_V\times W$, where $V$ and~$W$ are Boolean groups and $\dim V\ge 2$.
\end{itemize}
Two-valued groups from different series are never isomorphic to each other. Two two-valued groups, both in the series~$2$ or both in the series~$3$, are isomorphic to each other if and only if the corresponding Boolean groups~$V$ are isomorphic and simultaneously the corresponding Boolean groups~$W$ are isomorphic.
\end{theorem}

In order to give theorem ~\ref{theorem_non_fg} a complete form, one would like to answer the following question.

\begin{question}
Can it happen that abelian groups~$A$ and~$B$ are not isomorphic, but two-valued groups~$X_A^{\ba}$ and~$X_B^{\ba}$ are isomorphic?
\end{question}

We do not know the answer to this question for arbitrary abelian groups. Recall that in the case of finitely generated abelian groups, the negative answer was obtained in Proposition~\ref{propos_different}. This was proved on the basis of counting the number of elements of different orders; this proof, of course, cannot be generalized to the non-finitely generated case.

The proof of Theorem ~\ref{theorem_non_fg} follows the same scheme as the proof of Theorem ~\ref{theorem_main_fg}. Let $X$ be an arbitrary involutive commutative two-valued group. Just as in the proof of Theorem ~\ref{theorem_main_fg}, we consider three cases:

\begin{enumerate}
\item $X$ contains an element whose order is neither~$1$, nor~$2$, nor~$4$ (and then $X$ is necessarily non-special);
\item $X$ is non-special and consists entirely of elements of orders~$1$, $2$ and~$4$;
\item $X$ is special.
\end{enumerate}

In the first case, by Theorem ~\ref{theorem_nonspecial_classify}, a two-valued group~$X$ is isomorphic to a coset two-valued group of the form~$X_A^{\ba}$, where $A$ is an abelian group. Recall that in the proof of Theorem ~\ref{theorem_nonspecial_classify} we never used the finite generation property, see Remark~\ref{remark_nonspecial_classify}.

In the second and third cases, we first of all note that Propositions~\ref{propos_neras} and~\ref{propos_rasch_iso} can easily be carried over to the non-finitely generated case. The only difference is that in the proof of Proposition ~\ref{propos_rasch_iso}, the usual induction on dimension must be replaced in the standard way by a transfinite induction. Thus, from an arbitrary involutive commutative two-valued group, the maximal single-valued direct factor, which is a Boolean group, is split off in a unique way up to isomorphism. This allows us to reduce the classification theorem to the case of two-valued groups~$X$ that do not contain single-valued direct factors, that is (by Proposition ~\ref{propos_neras}) such two-valued groups~$X$, in which every element of order~$2$ is the square of an element of order ~$4$. If such a two-valued group~$X$ is special, then $X$ is isomorphic to a group of the form~$Y_V$ according to Theorem~\ref{theorem_special_classify}, in the proof of which the finite generation property was also not used anywhere.

It remains for us to consider the case when $X$ is a non-special two-valued group consisting entirely of elements of orders~$1$, $2$ and~$4$, in which each element of order~$2$ is the square of an element of order~ $4$, and prove that in this case $X$ is isomorphic either to a group of the form~$X_{C_4^{\omega}}^{\ba}$, where $\omega$ is some (possibly infinite) cardinal, or to a group of the form~$X_V^{\bu}$, where $V\cong C_2^{\omega}$ is a Boolean group (possibly infinite). Here $A^{\omega}$ denotes the direct sum (not the direct product) of $\omega$ copies of the abelian group~$A$. This case is the most  substantial and requires some additional considerations, in addition to those contained in sections ~\ref{section_124} and ~\ref{section_quasi}. More precisely, in the case of a not necessarily finitely generated two-valued group~$X$ (and, therefore, not necessarily a finite Boolean group~$V$ of its elements of order~$2$), all arguments of sections ~\ref{section_124} and ~\ref{section_quasi} except for the proof of Lemma ~\ref{lem_restrict_3_inj} carry over word for word. Thus, we only need to prove the following statement.

\begin{lem}\label{lem_restrict_3_inj_inf}
Let $V$ be an infinite Boolean group and $\Pi\colon V\to U$ be a projection onto some finite subgroup $U\subset V$ with $\dim U\ge 3$. Then the homomorphism $\Pi_*\colon \CH(V)\to\CH(U)$ is injective.
\end{lem}

We will need the following two technical Lemmas.

\begin{lem}\label{lem_criterium_delta}
Let $\varphi$ be an involutive symmetric quasi-cocycle on a Boolean group~$W$ for which the conditions
\begin{gather}
\varphi(u,v)= \varphi(v,u),\label{eq_cond1_strict_new}\\
\varphi(u,u)= \varphi(e,u)=\varphi(u,e)=e\label{eq_cond2_strict_new}
\end{gather}
hold exactly. Consider the map $\delta\varphi\colon W\times W\times W\to W$ given by the formula
\begin{equation}\label{eq_deltavarphi}
(\delta\varphi)(u,v,w)=\varphi(v,w)\varphi(uv,w)\varphi(u,vw)\varphi(u,v).
\end{equation}
The quasi-cocycle~$\varphi$ is cohomologically trivial if and only if there exists a map $\lambda\colon W\times W\to\F_2$ such that $\lambda(u,u)=\lambda(e,u)=\lambda(u,e)=0$ for all $u\in W$ and
\begin{equation}\label{eq_deltaphiuvw}
(\delta\varphi)(u,v,w)=
u^{\lambda(uv,w)+\lambda(u,vw)+\lambda(u,v)}
v^{\lambda(v,w)+\lambda(uv,w)+\lambda(vw,u)+\lambda(v,u)}
w^{\lambda(w,v)+\lambda(w,uv)+\lambda(wv,u)}
\end{equation}
for all $u,v,w\in W$.
\end{lem}

\begin{proof}
It is easy to see that any quasi-cocycle cohomologous to the quasi-cocycle~$\varphi$ and satisfying the conditions~\eqref{eq_cond1_strict_new}, \eqref{eq_cond2_strict_new} has the form
$$
\varphi'(u,v)=\varphi(u,v)\chi(u)\chi(v)\chi(uv)u^{\lambda(u,v)}v^{\lambda(v,u)}
$$
for some maps $\chi\colon W\to W$ and $\lambda\colon W\times W\to\F_2$ such that $\chi(e)=e$ and $\lambda(u,u)=\lambda(e,u)=\lambda(u,e)=0$ for all $u\in W$. Then
\begin{multline*}
(\delta\varphi')(u,v,w)={}\\(\delta\varphi)(u,v,w)
 u^{\lambda(uv,w)+\lambda(u,vw)+\lambda(u,v)}
v^{\lambda(v,w)+\lambda(uv,w)+\lambda(vw,u)+\lambda(v,u)}
w^{\lambda(w,v)+\lambda(w,uv)+\lambda(wv,u)}.
\end{multline*}
If quasi-cocycle~$\varphi$ is cohomologically trivial, then as~$\varphi'$ we can take the trivial quasi-cocycle, and we obtain exactly the equality ~\eqref{eq_deltaphiuvw}.

Conversely, if there is a map~$\lambda$ satisfying ~\eqref{eq_deltaphiuvw}, then the quasi-cocycle~$\varphi$ is equivalent to a true involutive symmetric cocycle
$$
\varphi'(u,v)=\varphi(u,v)u^{\lambda(u,v)}v^{\lambda(v,u)},
$$
for which condition $(\delta\varphi')(u,v,w)=1$ and conditions ~\eqref{eq_cond1_strict_new}, \eqref{eq_cond2_strict_new} are satisfied exactly. By Lemma~\ref{lem_cocycle}, the cocycle~$\varphi'$ is cohomologically trivial, and hence the quasi-cocycle~$\varphi$ is also cohomologically trivial.
\end{proof}

\begin{lem}\label{lem_restrict_finite}
An involutive symmetric quasi-cocycle on an infinite Boolean group~$V$ is cohomologically trivial if its restrictions to all finite subgroups along all possible projections are cohomologically trivial.
\end{lem}

\begin{proof}
Let $\varphi$ be an involutive symmetric quasi-cocycle on~$V$ whose restrictions along all possible projections to finite subgroups are cohomologically trivial. Replacing the quasi-cocycle~$\varphi$ with an equivalent one, we can assume that the conditions~\eqref{eq_cond1_strict_new}, \eqref{eq_cond2_strict_new} are satisfied exactly. Consider the map $\delta\varphi\colon V\times V\times V\to V$ given by the formula ~\eqref{eq_deltavarphi}.

Denote by~$\CF$ the set of finite subgroups of the group~$V$.
According to Lemma~\ref{lem_criterium_delta}, for each $U\in\CF$ there is a map $\lambda\colon U\times U\to\F_2$ such that $\lambda(u,u)=\lambda(u,e)=\lambda(e,u)=0$ for all $u\in U$ and the equality ~\eqref{eq_deltaphiuvw} holds for all $u,v,w\in U$. Denote by~$\Lambda_U$ the set of all maps~$\lambda$ that have these properties. Then $\Lambda_U$ is a non-empty finite set. Obviously, if $U'\subset U$, then $\lambda|_{U'}\in\Lambda_{U'}$ for every $\lambda\in\Lambda_U$. (For simplicity, hereafter we write $\lambda|_{U'}$ instead of~$\lambda|_{U'\times U'}$.)

Let us prove that it is possible to choose representatives~$\lambda_U\in\Lambda_U$, $U\in\CF$, so that for any two subgroups $U_1,U_2\in\CF$, the restrictions of the mappings~$\lambda_{U_1}$ and~$\lambda_{U_2}$ to~$U_1\cap U_2$ coincide. This choice is made in the following way. By Zermelo's theorem, the set~$\CF$ has a well-ordering~$\prec$. Moreover, we can assume that the smallest with respect to this ordering is the trivial subgroup~$\{e\}$. Let's select the representatives of~$\lambda_U$ in the order given by the ordering~$\prec$ using transfinite recursion. Each representative~$\lambda_U$ will be chosen such that the following property holds:
\begin{itemize}
\item[$(*)$] for any finite set of subgroups $U_1,\ldots,U_k\in\CF$ such that $U_1\prec \cdots\prec U_k=U$ and any subgroup $W\in \CF$ containing all the subgroups $U_1,\ldots,U_k$, there exists a map $\mu\in \Lambda_W$ such that $\mu|_{U_i}=\lambda_{U_i}$ with $i=1,\ldots,k$.
\end{itemize}

The recursion starts with a trivial map~$\lambda_{\{e\}}$ on a trivial subgroup~$\{e\}$.

Let us prove that the representative~$\lambda_U$ can indeed be chosen in accordance with the condition~$(*)$ if the previous representatives~$\lambda_{U'}$, $U'\prec U$ were chosen so that for each of them the condition~$(*)$ is satisfied. Assume the contrary: none of the maps in the set~$\Lambda_U$ if chosen as $\lambda_U$ will satisfy the condition~$(*)$. Let $\lambda_1,\ldots,\lambda_N$ be all elements of the finite set~$\Lambda_U$. Then for each of the maps~$\lambda_j$ there is a set of subgroups $U_1^{(j)},\ldots,U_{k_j}^{(j)}\in\CF$ such that $U_1^{(j)}\prec\cdots\prec U_{k_j}^{(j)}=U$, and a subgroup of $W_j\in \CF$ containing $U_1^{(j)},\ldots,U_{k_j}^{(j)}$, such that the set of maps $\lambda_{U_1^{(j)}},\ldots,\lambda_{U_{k_j-1}^{(j)}},\lambda_j$ cannot be extended to a map belonging to the set~$\Lambda_{W_j}$. Combine the sets of subgroups $U_1^{(j)},\ldots,U_{k_j}^{(j)}$, $j=1,\ldots,N$ by dropping the repeating ones and ordering the resulting finite set of subgroups using~$\prec$. As a result, we get a finite sequence of subgroups $U_1\prec \cdots\prec U_k=U$. Let $W$ be a finite subgroup of the group~$V$ generated by all the subgroups $W_1,\ldots,W_N$. Then for all~$j$ the set of maps $\lambda_{U_1},\ldots,\lambda_{U_{k-1}},\lambda_j$ cannot be extended to a map belonging to the set~$\Lambda_{W}$. However, it follows from the condition~$(*)$ for the subgroup $U_{k-1}\prec U$ that there is a map $\mu\in\Lambda_W$ such that $\mu|_{U_i}=\lambda_{U_i}$, $i=1,\ldots,k-1$. Since $\mu|_U$ is one of the mappings~$\lambda_1,\ldots,\lambda_N$, we obtain a contradiction, which completes the proof of the possibility of constructing maps~$\lambda_U$.

From the condition~$(*)$, in particular, it follows that for any two subgroups $U_1, U_2\in \CF$, the restrictions of the maps~$\lambda_{U_1}$ and~$\lambda_{U_2}$ to the intersection $U_1\cap U_2$ coincide. Therefore, the maps~$\lambda_U$ all together are combined into a map $\lambda\colon V\times V \to\F_2$ such that $\lambda(u,u)=\lambda(u,e)=\lambda(e,u)=0$ for all $u\in V$ and the equality ~\eqref{eq_deltaphiuvw} holds for all $u,v,w\in V$. Therefore, by Lemma ~\ref{lem_criterium_delta}, the quasi-cocycle~$\varphi$ is cohomologically trivial.
\end{proof}

Lemma~\ref{lem_restrict_3_inj_inf} follows from Lemmas~\ref{lem_restrict_3_sur}, \ref{lem_restrict_3_inj} and~\ref{lem_restrict_finite}.

Thus, we have proved that any involutive commutative two-valued group belongs to one of the three series listed in Theorem ~\ref{theorem_non_fg}.

The two-valued groups of the third series are special and, therefore, are not isomorphic to the two-valued groups of the first two series. The two-valued groups of the first and second series (in the case of the presence of only elements of orders~$1$, $2$ and~$4$)  are different since the corresponding quasi-cocycles are not cohomologous. To complete the proof of Theorem~\ref{theorem_non_fg}, we need to establish that a pair of Boolean groups~$(V,W)$ can be uniquely restored up to isomorphism from a two-valued group~$X$ of the second or third series. This is indeed true: the Boolean group~$W$ is the maximal single-valued direct factor in the decomposition $X\cong X'\times W$, which  by analogy with Proposition~\ref{propos_rasch_iso} is unique up to an isomorphism, and $V$ is the Boolean group of all elements of order $2$ in~$X'$ for the second series and the Boolean group of all elements of order $4$ in~$X'$ for the third series.

\section{Topological two-valued groups}\label{section_top}

If $Y$ is a topological space, then we endow the symmetric square $\Sym^2(Y)$ with the quotient topology of the direct product topology on~$Y\times Y$.

\begin{defin}
\textit{Topological two-valued group} is a topological space ~$X$ with two-valued group structure such that the operation of two-valued multiplication $X\times X\to \Sym^2(X)$ and the operation of taking the inverse $X\to X$ are continuous.
\end{defin}

We will consider only Hausdorff topological two-valued groups; the non-Hausdorff case is apparently much more complicated. It is easy to see that if $A$ is a single-valued commutative Hausdorff topological group and $\iota$ is an involutive continuous automorphism of a group~$A$, then the coset two-valued group $A/\iota$ has the natural structure of a Hausdorff topological two-valued group. Thus, topological two-valued groups~$X_A^{\ba}$ for an arbitrary Hausdorff topological abelian group~$A$ and $X_V^{\bu}$ for a Hausdorff topological Boolean group~$V$ are well defined. (Here and below we refer to \textit{topological Boolean group} as a topological group all of whose nontrivial elements have order~$2$.).

Again, we have three essentially different cases:

\begin{enumerate}
\item $X$ contains an element whose order is neither~$1$, nor~$2$, nor~$4$;
\item $X$ is non-special and consists entirely of elements of orders~$1$, $2$ and~$4$;
\item $X$ is special.
\end{enumerate}

Let's consider them separately. In cases (1) and (3) we will obtain analogues of the classification Theorems~\ref{theorem_nonspecial_classify} and~\ref{theorem_special_classify}. Case~(2) is more complicated. In this case the classification problem remains open. We will construct a non-trivial example showing that a direct analogue of the classification Theorem~\ref{thm_124} is false.

\subsection*{Topological two-valued groups~$X$ containing an element~$t$ such that $\ord t\notin\{1,2,4\}$.} In this case, there is a direct analogue of Theorem ~\ref{theorem_nonspecial_classify}, but its proof requires additional considerations.
\begin{theorem}
\label{theorem_nonspecial_top}
Let $X$ be an involutive commutative Hausdorff topological two-valued group containing at least one element whose order is neither~$1$, nor~$2$, nor~$4$. Then $X$ is isomorphic to a coset topological two-valued group of the form $X_A^{\ba}=A/\iota_{\ba}$, where $A$ is a commutative Hausdorff topological group and $\iota_{\ba}$ is the antipodal involution on~$A$.
\end{theorem}

As in section ~\ref{section_not124}, we define the desired topological group as the set
\begin{equation}
A=\{(x,p)\in X\times X\mid p\in t*x\}
\end{equation}
with the multiplication and reciprocal operations defined in Proposition~\ref{propos_main_construction}. Endow the set~$A$ with the topology induced by the direct product topology on~$X\times X$. The fact that the resulting structure of a commutative group is well-defined  was proved in section ~\ref{section_not124}, so we only need to prove that the introduced multiplication and inverse operations in the group~$A$ are continuous. (It is obvious that the isomorphism $A/\iota_{\ba}\to X$ given by the projection onto the first factor is continuous in both directions.) To prove the continuity of operations, we need the notion of the Smith-Dold ramified covering, see ~\cite{Smi83}, ~\cite{Dol86}.

\begin{defin}\label{defin_DS}
Let $\pi\colon \widetilde{Y}\to Y$ be a surjective continuous mapping of topological spaces with finite inverse images of points. A map~$\pi$ is called \textit{$d$-fold ramified covering in the sense of Smith-Dold} if there is a continuous map $\tau\colon Y\to \Sym^d\bigl( \widetilde{Y}\bigr)$ such that
\begin{enumerate}
\item $\tilde y\in \tau\circ\pi(\tilde y)$ for any point $\tilde y\in \widetilde{Y}$,
\item $\Sym^d(\pi)\circ\tau (y)=\bigl[\,\underbrace{y,\dots,y}_d\,\bigr]$ for any point $y\in Y$.
\end{enumerate}
\end{defin}

\begin{remark}
In the above definition we follow ~\cite{BR-08}. The equivalent original definition of L.\,Smith is as follows. A surjective continuous mapping~$\pi\colon \widetilde{Y}\to Y$ is called \textit{ramified covering} if there is a map
$\mu\colon \widetilde{Y}\to \Z_{>0}$, called \textit{multiplicity map}, and the following two conditions are satisfied.
\begin{enumerate}
\item For any point $y\in Y$
$$
\sum_{\tilde{y}\in\pi^{-1}(y)}\mu(\tilde{y})=d.
$$
\item A map $\tau\colon Y\to \Sym^d\bigl(\widetilde{Y}\bigr)$ such that $y$ goes into a multiset~$\pi^{-1}(y)$ in which each point~$\tilde y$ is counted with a multiplicity~$\mu(\tilde y)$ is continuous.
\end{enumerate}
The multiplicity map~$\mu$ is recovered from the map~$\tau$ from the definition~\ref{defin_DS} as the multiplicity of occurrence of the point~$\tilde y$ in the multiset~$\tau\circ\pi(\tilde y)$.
\end{remark}

Informally speaking, a $d$-fold ramified covering is a continuous map having a continuous $d$-valued inverse.

Let $\pi\colon A\to X$ be the projection onto the first factor. Since the multiplication in the two-valued group~$X$ is continuous, the map $x\mapsto (x,t*x)$ is a continuous two-valued inverse of the map~$\pi$. Therefore the projection~$\pi$ is a two-fold ramified Smith-Dold covering. As noted in A.~Dold~\cite[Sec.~1]{Dol86}, in the Hausdorff case the partially defined operation of removing one of the elements of the multiset $$\Sym^d(Y)\times Y\dashrightarrow \Sym^{d-1}(Y),\qquad (\sigma,y)\mapsto \sigma\setminus y$$
is continuous on its domain. This immediately implies the continuity of the operation of taking the inverse in the group~$A$, which was defined by the formula~$(x,p)^{-1}=(x,p')$, where $[p,p']=x*t$.

\begin{remark}
Let us pay attention to the specifics of two-fold Smith-Dold ramified coverings: on the total space of such a covering there is always a continuous involution that swaps sheets. In our case this involution is an antipodal involution on the group~$A$. In the case of the $d$-fold Smith-Dold ramified covering $\widetilde{Y}\to Y$ with $d>2$ the situation is more complicated: on the space~$\widetilde{Y}$ itself, generally speaking, there is no action of the order $d$ group. Nevertheless, Buchstaber and Rees ~\cite{BR-08}, using the results of Dold~\cite{Dol86}, found a construction of the associated $d!$-fold ramified covering $E$ over~$Y$ with the action of the permutation group~$S_d$ such that $\widetilde{Y}=E\times_{S_d}\{1,\dots,d\}$.
\end{remark}

To prove the continuity of the product operation, we need the following technical Lemma.

\begin{lem}\label{lem_DS}
Let $\pi\colon\widetilde{Y}\to Y$ be the Smith-Dold ramified covering of Hausdorff topological spaces and $s\colon Y\to \widetilde{Y}$ be its set-theoretic section, that is, a map such that $\pi\circ s=\mathrm{id}_Y$. Assume that the graph $\Gamma_s\subset Y\times \widetilde{Y}$ of the section~$s$ is closed. Then $s$ is continuous.
\end{lem}

\begin{proof} Identifying $\Gamma_s$ with its preimage under the embedding $\pi\times\mathrm{id}\colon \widetilde{Y}\to Y\times\widetilde{Y}$, we may assume that $\Gamma_s\subset\widetilde{Y}$. The subset $\Gamma_s$ is closed in $\widetilde{Y}$ since it is closed in~$Y\times\widetilde{Y}$.
Let us prove the continuity of the section~$s$ at an arbitrary point $y\in Y$. Consider an arbitrary neighbourhood $U$ of the point~$s(y)$. Let $y_1=s(y),y_2,\ldots,y_k$ be all different inverse images of the point~$y$ under the map~$\pi$ without regard to multiplicities. The points $y_2,\ldots,y_k$ do not belong to the graph~$\Gamma_s$. Therefore, since the graph is closed, it follows that they have open neighbourhoods $U_2,\ldots,U_k$ respectively, disjoint from~$\Gamma_s$. Now, from the continuity of the multivalued inverse to the projection~$\pi$ it follows that there is an open neighbourhood~$V$ of the point~$y$ in~$Y$ such that the set~$\pi^{-1}(V)$ is contained in the union of neighbourhoods $U\cup U_2\cup\cdots\cup U_k$. Therefore the set $s(V)$ is contained in the intersection of this union of neighbourhoods with the graph~$\Gamma_s$ and, therefore, in the original neighbourhood~$U$. Hence the map~$s$ is continuous at the point~ $y$.
\end{proof}

Let us deduce from Lemma~\ref{lem_DS} the continuity of the operation of multiplication~$\bullet$ constructed in Proposition~\ref{propos_main_construction}. Consider the closed subset $B\subset A\times A\times A$ consisting of all triples $\bigl((x,p),(y,q),(z,r)\bigr)$ such that $z\in x*y$ and the projection $\sigma\colon B\to A\times A$ onto the product of the first two factors. It is easy to see that $\sigma$ is a four-fold Smith-Dold ramified covering. Indeed, a continuous four-valued inverse mapping assigns to each pair $\bigl((x,p),(y,q)\bigr)$ four (with multiplicities) triples $\bigl((x,p),(y,q),(z,r)\bigr)$ such that $z\in x*y$ and $r\in t*z$. According to the construction in Proposition~\ref{propos_main_construction}, as the product $(x,p)\bullet (y,q)$ we need to take $(z,r)$ such that $\bigl((x,p),(y,q),(z,r)\bigr)\in\sigma^{-1}\bigl((x,p),(y,q)\bigr)$, which additionally satisfies the conditions $z\in x*y$, $z\in p*q'$, $z\in p'*q$, $r\in x*q$ , $r\in p*y$, $r'\in x*q'$ and $r'\in p'*y$, where $t*x=[p,p']$, $t*y=[q,q']$ and~$t*z=[r,r']$, and these conditions always determine the desired point~$(z,r)$ uniquely. In this way,
$$
s\colon\bigl((x,p),(y,q)\bigr)\mapsto \bigl((x,p),(y,q),(x,p)\bullet (y,q)\bigr)
$$
is a set-theoretic section of the ramified covering~$\sigma$. As noted earlier, it follows from the Hausdorff property that the points~$p'$, $q'$, and $r'$ depend continuously on the pairs $(x,p)$, $(y,q)$ and~$(z,r)$, respectively. Therefore each of the listed conditions cuts out a closed subset in $A\times A\times A$. The graph of the section $s$ is the intersection of all these closed subsets and the closed set~$B$. Hence, it is closed and, therefore, by Lemma~\ref{lem_DS}, the section~$s$ is continuous, which completes the proof of Theorem~\ref{theorem_nonspecial_top}.

\subsection*{Special topological two-valued groups.} Let $V$ be a Boolean group with discrete topology and $\dim V\ge 2$, $U$ be an arbitrary Hausdorff topological Boolean group and~$s\in U$ be an arbitrary element other than ~$e$. Consider the quotient space $W=U/\langle s\rangle$ with the quotient topology (which is always Hausdorff). Let $\pi\colon U\to W$ be the projection.

Now consider the topological space
\begin{equation}\label{eq_YVUs}
Y_{V,U,s}= (\{e\}\times U)\sqcup \bigl( (V\setminus\{e\})\times W\bigr),
\end{equation}
where $\sqcup$ denotes a disjoint union, and define a commutative two-valued multiplication on it by the formulae
\begin{align*}
(e,u_1)*(e,u_2)&=\bigl[(e,u_1u_2),(e,u_1u_2)\bigr],&&u_1,u_2\in U,\\ 
(e,u)*(v,w)&=\bigl[(v,\pi(u)w),(v,\pi(u)w)\bigr],&&u\in U,\,v\in V\setminus\{e\},\,w\in W,\\
(v,w_1)*(v,w_2)&=\bigl(e,\pi^{-1}(w_1w_2)\bigr),&&v\in V\setminus\{e\},\,w_1,w_2\in W,\\
(v_1,w_1)*(v_2,w_2)&=\bigl[(v_1v_2,w_1w_2),(v_1v_2,w_1w_2)\bigr],&&v_1,v_2\in V\setminus\{e\},\,v_1\ne v_2,\, w_1,w_2\in W.
\end{align*}
It can be directly verified that $Y_{V,U,s}$ is a special Hausdorff topological two-valued group. Moreover, if we forget about topology, this two-valued group becomes isomorphic to the two-valued group $Y_V\times W$.

\begin{theorem}\label{theorem_special_top}
Any special involutive commutative Hausdorff topological two-valued group~$X$ is isomorphic to one of the topological two-valued groups of the form~$Y_{V,U,s}$. Two topological two-valued groups $Y_{V_i,U_i,s_i}$, $i=1,2$ are isomorphic to each other if and only if $V_1\cong V_2$ as abstract groups and there exists an isomorphism of topological groups $U_1\cong U_2$ taking~$s_1$ to~$s_2$.
\end{theorem}

\begin{proof}
Theorem ~\ref{theorem_non_fg} immediately implies that, if we forget about the topology, then $X$ is isomorphic to some two-valued group of the form~$Y_{V,U,s}$. Therefore, the first statement of Theorem ~\ref{theorem_special_top} follows from the following Proposition.

\begin{propos}
Any Hausdorff topology on the set $Y_{V,U,s}$, with respect to which the introduced operation of two-valued multiplication is continuous, has the form of a disconnected union of some topology on~$\{e\}\times U=U$, with respect to which $U$ is a Hausdorff topological group, and the product topology on $(V\setminus\{e\})\times W$, where the set $V\setminus\{e\}$ is endowed with the discrete topology, and the set~$W$ is equipped with the quotient topology of the chosen topology on~$U$.
\end{propos}

\begin{proof}
Consider an arbitrary Hausdorff topology~$\T$ on a two-valued group~$Y=Y_{V,U,s}$ with respect to which the operation of two-valued multiplication is continuous. We will identify the subset $\{e\}\times U$ with~$U$ and denote the element~$(e,u)$ of the two-valued group~$Y$ simply by~$u$. Note that $s\in x*x$ if and only if the element $x$ lies in the subset $(V\setminus\{e\})\times W$ of the two-valued group~$Y$. Since the multiplication operation in~$Y$ is continuous, the condition $s\in x*x$ singles out a closed subset in~$Y$. Hence, the subset $\{e\}\times U$ is open in~$Y$. Denote by~$\T_U$ the restriction of the topology~$\T$ to~$U$; then $(U,\T_U)$ is a topological group. Denote by $\T_W$ the quotient topology of the topology~$\T_U$ under the projection $\pi\colon U\to W$.

It remains for us to prove that each of the sets $\{v\}\times W$ is also open in~$Y$ and the topology~$\T|_{\{v\}\times W}$ coincides with the quotient topology~$\T_W$ under the standard identification of $\{v\}\times W$ with~$W$. Let $y = (v,e)$; then the formula $x\mapsto x*y$ defines a continuous map $\mu_y\colon Y\to\Sym^2(Y)$. First, note that the set $\{v\}\times W$ coincides with the inverse image of the open set $\Sym^2(U)$ under the mapping~$\mu_y$, so $\{v\}\times W$ is open. Secondly, if the mapping~$\mu_y$ is restricted to the set~$\{e\}\times U$, then its two branches coincide and it defines a continuous mapping $\nu_y\colon \{e\}\times U\to\{v\}\times W$ coinciding with the mapping $(e,u)\mapsto \bigl(v,\pi(u)\bigr)$. Third, the mapping~$\mu_y|_{\{v\}\times W}$ is a continuous two-valued inverse of the mapping~$\nu_y$. Therefore, $\nu_y$ is a two-fold Smith-Dold ramified covering, which immediately implies that the topology on the set~$\{v\}\times W$ coincides with the quotient topology~$\T_W$.
\end{proof}

The second statement of Theorem~\ref{theorem_special_top} follows from the fact that the topological Boolean group~$U$, its element~$s$, and the discrete Boolean group~$V$ are uniquely restored up to isomorphism from the topological two-valued group~$Y=Y_{V,U,s}$. Indeed, $U\subset Y$ is a subgroup consisting of the identity and all elements of order~$2$, with the topology of a subset, $s$ is the element which is the square of all elements of order~$4$ and $V\cong Y/U$.
\end{proof}

\subsection*{Non-special topological two-valued groups consisting of elements of orders~1, 2, and~4.} Here we give an example showing that fundamentally new effects arise in this case and that the direct analogue of Theorem ~\ref{thm_124} is false.

First of all, we note that for topological two-valued groups there are no analogues of direct factorization results, such as Propositions~\ref{propos_neras} and ~\ref{propos_rasch_iso}. This is because even for ordinary commutative Hausdorff topological groups there is no well-defined operation of taking a quotient group that would leave us in the category of Hausdorff topological groups (for example, when a subgroup is dense in a group). Therefore, in the topological case it is natural to generalise the definition of the unipotent series $X_V^{\bu}\times W$ as follows.

We will need the following category~$\CP$, which we will call \textit{category of pairs of Hausdorff topological Boolean groups}. The object of this category is the pair $(V,U)$, where
\begin{itemize}
\item $V$ and~$U$ are Hausdorff topological Boolean groups, with respect to some topologies~$\T_V$ and~$\T_U$ respectively (which are part of the structure);
\item $U$ is a subgroup of the group~$V$, and the inclusion $U\subset V$ is continuous with respect to the topologies~$\T_U$ and~$\T_V$.
\end{itemize}
Note that we do not require the topology~$\T_U$ to be the same as the restriction of topology~$\T_V$ to~$U$; the definition only implies that the topology of $\T_U$ either coincides with the restriction of the topology~$\T_V$ or is finer than it. A morphism in the category~$\CP$ between objects~$(V_1,U_1)$ and~$(V_2,U_2)$ is a commutative diagram
\begin{center}
\begin{tikzcd} [row sep=small]
U_1\arrow[r, "\alpha"] & U_2\\
V_1\arrow[r, "\beta"]\arrow[u, phantom, sloped, "\supset"] & V_2\arrow[u, phantom, sloped,"\supset"]
\end{tikzcd}
\end{center}
where $\alpha$ is a homomorphism continuous with respect to the topologies~$\T_{U_1}$ and~$\T_{U_2}$, and~$\beta$ is a homomorphism continuous with respect to the topologies~$\T_{V_1}$ and~$\T_{V_2}$. We will call the objects of the category~$\CP$ \textit{pairs of Hausdorff topological Boolean groups}.

Let $(V,U)$ be a pair from the category~$\CP$. Consider the continuous \textit{unipotent involution}
$\iota_{\bu}\colon U\times V\to U\times V$, determined by the formula
\begin{equation*}
\iota_{\bu}(u,v)=(u,uv).
\end{equation*}
Then the coset two-valued group
\begin{equation}\label{unip_top}
X^{\bu}_{V,U}=(U\times V)/\iota_{\bu}
\end{equation}
is an involutive commutative Hausdorff topological two-valued group.

In the particular case when the subgroup~$U$ is distinguished as a topological direct factor, $V=U\times W$, and the topology~$\T_V$ is the topology of a product, this construction yields topological two-valued groups $X^{\bu}_{V,U}\cong X^{\bu}_{U}\times W$.

Denote by~$\CC$ the category of non-special involutive commutative Hausdorff topological two-valued groups consisting of elements of orders~$1$, $2$ and~$4$ and their continuous homomorphisms.
A natural analogue of Theorem~\ref{thm_124} would be the statement that any topological two-valued group from the category~$\CC$ is isomorphic either to a two-valued group of the form~$X^{\ba}_A$, where $A$ is a Hausdorff topological group consisting of elements of orders~ $1$, $2$ and~$4$, or to a two-valued group of the form~$X^{\bu}_{V,U}$, where $(V,U)$ is a pair of Hausdorff topological Boolean groups. However, below we will give an example showing that this statement is not true.

Before building this example, let's discuss how the cohomological classification approach from sections~\ref{section_124} and~\ref{section_quasi} is adapted to the topological case under consideration and what problems arise along the way.

Let $X$ be a topological two-valued group from the category~$\CC$. Denote by~$V$ the Boolean group, consisting of all elements of orders~$1$ and~$2$ in~$X$, with the Hausdorff topology ~$\T_V$ induced by the inclusion $V\subset X$. Denote by~$U\subset V$ the subgroup consisting of all elements~$x^2$, where $x\in X.$ Similarly to Lemma~\ref{lem_124mult}, for every element $u\in U$ the set $X_u$ consisting of all $x$ such that $x^2=u$, is a $V$-orbit. Moreover, the stabilizer of the elements of this orbit is the subgroup~$\{e,u\}\subset V$ and, if $x\in X_u$, $y\in X_{w}$ and~$x*y=[z_1,z_2]$, then $z_1,z_2\in X_{uw}$ and $z_2=uz_1=wz_1$. Thus, there is a canonical bijection $X/V\leftrightarrow U$ and the (single-valued) multiplication in~$U$ is induced by the two-valued multiplication in~$X$. Let us supply $U$ with the quotient topology $\T_U.$ A priori it is not obvious that this topology is Hausdorff. We note, however, that the continuity of the operation of squaring in ~$X$ implies the continuity of the embedding $X/V=U\subset V$, where $U$ is endowed with the quotient topology~$\T_U$, and $V$ is endowed with the subset topology~$\T_V$. Therefore, the topology $\T_U$ is Hausdorff and the pair $(V,U)$ is an object of the category~$\CP$. Thus, the constructed correspondence
$$
X\mapsto (V,U)
$$
defines a functor from the category~$\CC$ to the category~$\CP$, which we will denote by~$\Phi$.

It can be directly verified that $\Phi(X_{V,U}^{\bu})=(V,U)$ for any pair $(V,U)$ from the category~$\CP$. This implies, in particular, that topological two-valued groups~$X^{\bu}_{V_1,U_1}$ and~$X^{\bu}_{V_2,U_2}$ are isomorphic if and only if pairs of topological Boolean groups~$(V_1,U_1)$ and~$(V_2,U_2)$ are isomorphic.

The functor~$\Phi$ allows splitting the problem of classifying topological two-valued groups~$X$ from the category~$\CC$ into separate problems of classifying all possible~$X$ with each given (up to isomorphism) pair~$\Phi(X)=(V,U)$. If we choose in each $V$-orbit~$X_u$, where~$u\in U$, a representative~$x_u$ (so that $x_e=e$), then multiplication in the two-valued group~$X$ will be uniquely given by the rule of multiplication of representatives~$x_u$, which will look like
\begin{equation}\label{eq_quasi_coc_top}
x_u*x_v=\varphi(u,v)[x_{uv},ux_{uv}],\qquad u,v\in U,
\end{equation}
for some involutive symmetric quasi-cocycle $\varphi\colon U\times U\to V$, that is, a map such that
\begin{gather}
\varphi(u,v)\equiv \varphi(v,u) \pmod{\langle u,v\rangle},\label{eq_cond1_top}\\
\varphi(u,u)\equiv\varphi(e,u)\equiv\varphi(u,e)\equiv e\pmod{\langle u\rangle},\label{eq_cond2_top}\\
\varphi(u,v)\varphi(uv,w)\varphi(u,vw)\varphi(v,w)\in\langle u,v,w\rangle\label{eq_cond3_top}
\end{gather}
for all $u,v,w\in U$.

So far, the only difference from what happened in section~\ref{section_124} is that the group~$U$ is not necessarily the same as~$V$. However, now we need to ensure that the multiplication given by the formula~\eqref{eq_quasi_coc_top} is continuous. Naturally, in order to follow the continuity property of multiplication, we would like the map $U\to X$ of the choice of representatives of $V$-orbits $u\mapsto x_u$ to be continuous (with respect to the topology~$\T_U$ on~$U$). Let's call a topological group~$X$ from the category~$\CC$ \textit{tame} if it has a continuous mapping of the choice of representatives of $V$-orbits $u\mapsto x_u$, and \textit{wild} if there is no such continuous mapping. This is where the first fundamental problem arises.

\begin{question}
Does the category~$\CC$ contain wild topological two-valued groups?
\end{question}

The authors do not know the answer to this question at the moment. On the one hand, there are no reasonable arguments in favour of the negative answer, especially if we take into account that the topological space~$U$ is not simply connected. On the other hand, so far we failed to construct an example of a wild two-valued group.

In what follows, we focus on the study of tame two-valued groups from the category~$\CC$ and show that even among them there is a topological group that belongs neither to the principal, nor to the unipotent series. First of all, we prove that, as a topological space, the two-valued group~$X$ is homeomorphic to the unipotent two-valued group $X^{\bu}_{V,U}=(U\times V)/\iota_{\bu}$.

\begin{propos}
Let $X$ be a tame topological two-valued group from the category~$\CC$ and $\Phi(X)=(V,U)$. Then as the topological space the two-valued group $X$ is homeomorphic to the unipotent two-valued group $X^{\bu}_{V,U}=(U\times V)/\iota_{\bu}$. More precisely, if $u\mapsto x_u$ is a continuous mapping of the choice of representatives in the $V$-orbits of the two-valued group~$X$, then
the mapping $g\colon U\times V\to X$ defined by the formula $g(u,v)= vx_u$ induces a homeomorphism $\bar g\colon (U\times V)/\iota_{\bu}\approx X$.
\end{propos}

\begin{proof}
From the fact that the stabilizer of each element of the orbit~$Vx_{u}$ under the action of the group~$V$ is equal to~$\{e,u\}$, it follows that $\bar g$ is a bijection. The continuity of the map $g$ (and hence, of the map $\bar g$) follows from the continuity of the mapping $u\mapsto x_u$ and the continuity of multiplication in the two-valued group~$X$. Denote by $\pi$ the (continuous) projection of $X\to X/V=U$. The mapping $g$ admits a two-valued continuous inverse mapping $X\to \Sym^2(U\times V)$ defined by formula $x\mapsto [\pi(x),x*x_{\pi(x)}]$. (Thus, $g$ is a two-fold Smith-Dold ramified covering.) This immediately implies the continuity of the mapping~$\bar g^{-1}$.
\end{proof}

Thus, an arbitrary tame two-valued group~$X$ with $\Phi(X)=(V,U)$ can be considered as a two-valued group of the unipotent series~$X_{V,U}^{\bu}$ with multiplication deformed by an involutive symmetric quasi-cocycle~$\varphi$ according to the formula ~\eqref{eq_quasi_coc_top}. (Recall that, as was shown in section ~\ref{section_124}, the trivial quasi-cocycle corresponds to the two-valued group~$X_{V,U}^{\bu}$ itself.) We will use the notation~ $X_{V,U,\varphi}$ for the two-valued group with multiplication ~\eqref{eq_quasi_coc_top}.

It is easy to check that for the multiplication in the two-valued group~$X_{V,U,\varphi}$ to be continuous, it is necessary and sufficient that the quasi-cocycle $\varphi$ is quasi-continuous in the following sense. Recall that, in fact, each of the values~$\varphi(u,v)$ is not uniquely defined, but up to multiplication by an element from~$\langle u,v\rangle$. We say that a quasi-cocycle $\varphi\colon U\times U\to V$ is \textit{quasi-continuous} if the four-valued mapping $U\times U\to \Sym^4(V)$ defined by the formula
$$
(u,v)\mapsto \bigl[\varphi(u,v),u\varphi(u,v),v\varphi(u,v),uv\varphi(u,v)\bigr],
$$
is continuous.

If instead of a set of representatives of $V$-orbits~$x_u$ we choose another set of representatives~$x_u'$, where again $x_e'=e$, then $x_u'=\chi(u)x_u$ for some mapping~$\chi\colon U\to V$ such that $\chi(e)=e$. Then the quasi-cocycle $\varphi$ is replaced by a quasi-cocycle~$\varphi'$ such that
\begin{equation}\label{eq_homological_top}
\varphi'(u,v)\equiv\varphi(u,v)\chi(u)\chi(v)\chi(uv)\pmod{\langle u,v\rangle}.
\end{equation}
Moreover, the transition to a new set of representatives~$x_u'$ will continuously depend on~$u$ if and only if the mapping~$\chi$ is quasi-continuous in the sense that the two-valued mapping $U\to \Sym^2(V)$ given by the formula
$$
u\mapsto \bigl[\chi(u),u\chi(u)\bigr]
$$
 is continuous.
 
Thus, it is natural to call two quasi-continuous involutive symmetric quasi-cocycles~$\varphi$ and~$\varphi'$ \textit{cohomologous} if they are related by the relation ~\eqref{eq_homological_top} for some quasi-continuous mapping~$\chi\colon U\to V$ such that $\chi(e)=e$. Then the following analogue of the Proposition ~\ref{propos_2gqc} holds.

\begin{propos}\label{propos_2gqc_top}
Topological two-valued groups~$X_{V,U,\varphi}$ and~$X_{V',U',\varphi'}$ are isomorphic if and only if there is an isomorphism of pairs of topological Boolean groups~$(V,U)$ and~$(V',U')$ taking a quasi-continuous quasi-cocycle~$\varphi$ into a quasi-continuous quasi-cocycle cohomologous to~$\varphi'$. In particular, $X_{V,U,\varphi}$ is isomorphic to~$X_{V,U}^{\bu}$ if and only if the quasi-continuous quasi-cocycle~$\varphi$ is cohomologically trivial.
\end{propos}

Denote by $\CH_{\mathrm{qc}}(U,V)$ the Boolean group of cohomology classes of quasi-continuous quasi-cocycles $U\times U\to V$.
Thus, the classification of the tame topological two-valued groups~$X$ such that $\Phi(X)=(V,U)$ reduces to the following problem.

\begin{prob}
Describe the group~$\CH_{\mathrm{qc}}(U,V)$.
\end{prob}

Naturally, there is a topology-forgetting homomorphism $$f\colon\CH_{\mathrm{qc}}(U,V)\to\CH(U,V),$$ where $\CH(U,V)$ is the Boolean group of all (without the requirement of quasi-continuity) quasi-cocycles up to the cohomology defined by an arbitrary (again without the requirement of quasi-continuity) mapping~$\chi$. Using Lemma~\ref{lem_cocycle}, it is easy to verify that the isomorphism $\CH(U,V)\cong \CH(U)$ always holds, and hence the group $\CH(U,V)$ is trivial when $\dim U\le 2$ and $\CH(U,V)\cong C_2$ when $\dim U\ge 3$. Apparently, in the general situation, the homomorphism $f$ is neither injective, nor surjective. 

We now give an example of an element in the kernel of the homomorphism~$f$, which provides us with the construction of a tame topological two-valued group from the category~$\CC$, which does not belong neither to the antipodal family~$X_A^{\ba}$, nor to the unipotent family~$X_{V,U}^{\bu}$. In our example, $U=V$, so from now on we will write $\CH_{\mathrm{qc}}(V)$ instead of $\CH_{\mathrm{qc}}(V,V)$ and $X_{V,\varphi}$ instead of $X_{V,V,\varphi}$.

To construct the above example, we need the notion of \textit{free topological Boolean group}. Let $Y$ be a completely regular Hausdorff topological space. A.\,A.~Markov~\cite{Mar45} introduced the concepts of a free topological group and a free abelian topological group with basis~$Y$. We will need a similar concept of a free Boolean group ~$B(Y)$ with basis~$Y$, see ~\cite{Sip}. Algebraically, $B(Y)$ is just a Boolean group~$C_2^Y$, that is, a vector space over a field~$\F_2$ whose basis is the set~$Y$. As a topology on~$B(Y)$, we take a topology that is consistent with the group structure and satisfies the following two conditions:
\begin{enumerate}
\item the embedding of a basis $Y\hookrightarrow B(Y)$ is a homeomorphism onto the image;
\item any continuous mapping of the space~$Y$ into an arbitrary Boolean topological group extends to a continuous homomorphism of the group~$B(Y)$ into this topological group.
\end{enumerate}
Such a topology exists, is unique, and is simultaneously the finest among all topologies compatible with the group structure and satisfying condition~(1), and the coarsest among all topologies compatible with the group structure and satisfying condition~(2). Proofs of analogues of these statements for free topological groups and free abelian topological groups are contained in ~\cite{Mar45} and ~\cite{Gra48}; the necessary extension of these results to the Boolean case can be found in ~\cite{Sip}.

The elements of the group~$B(Y)$ can be interpreted as finite subsets of the space~$Y$ (in the usual sense, that is, without multiplicities). In this case, the operation in the group is the symmetric difference of subsets, and the identity is the empty set. For convenience, we will denote the symmetric difference of subsets~$u$ and~$v$ through~$uv$.

We are interested in the case when $Y$ is a compact metric space with metric~$d$. Then the topology of a free topological Boolean group on~$B(Y)$ can be described as follows, see ~\cite{Gra48}, \cite{Sip}.

For each subset $v\in B(Y)$ define its \textit{norm} $\| v\|$ with values in $\R_{\ge 0}\cup\{+\infty\}$ as follows:
\begin{itemize}
\item if the set~$v$ consists of an even number of elements, then $\| v\|$ is the minimum of the sums
$$
\sum_{i=1}^{k}d(p_i,q_i)
$$
over all partitions of the set~$v$ into pairwise disjoint pairs $\{p_1,q_1\},\ldots,\{p_k,q_k\}$;
\item if the set~$v$ consists of an odd number of elements, then $\| v\|=+\infty$.
\end{itemize}

The constructed function is a norm in the following sense:
\begin{itemize}
\item $\|v\|\ge 0$ for all~$v$ and $\|v\|=0$ if and only if $v=e$;
\item $\|uv\|\le \|u\|+\|v\|$.
\end{itemize}

The introduced norm is analogous to the seminorms from A.\,A.~Markov~\cite{Mar45} and M.\,I.~Graev~\cite{Gra48} for topological Boolean groups; see ~\cite[Section~2]{Sip} for specific formulae.

So the function
$$
\rho(u,v)=\|uv\|
$$
defines a well-defined $\infty$-metric on~$B(Y)$, that is, a metric with values in~$\R_{\ge 0}\cup\{+\infty\}$. We could look at the corresponding metric topology at ~$B(Y)$, but that would not be the topology we want. The topology we need is finer, it is obtained as follows. 

For each~$k\ge 0$ we denote by~$B_k(Y)$ the subset of the group~$B(Y)$ consisting of all subsets of $v\subset Y$ with at most~$k$ elements. Endow each set~$B_k(Y)$ with the metric topology induced by the metric~$\rho$. The desired topology of the free topological Boolean group on~$B(Y)$ is the topology of the direct limit of these metric topologies. This means that the set~$U$ is open in~$B(Y)$ if and only if the intersection~$U\cap B_k(Y)$ is open in~$\bigl(B_k(Y),\rho\bigr)$ for every~$k$.

Informally, $B(Y)$ can be thought of as a configuration space of finite collections of points on~$Y$ with the rule that when a pair of points collides, they both vanish. In particular, it is easy to see that if $Y$ is path-connected, then $B(Y)$ consists of two path-connected components: $B_{even}(Y)$ and~$B_{odd}(Y)$, consisting of all subsets of even and odd cardinality, respectively.

\begin{example}
Let $K\subset \R^d$ be a compact convex body (i.e. a convex set with non-empty interior) in a finite-dimensional Euclidean space of dimension $d\ge 1$. Consider a free topological Boolean group~$B(K)$.

We define a quasi-cocycle $\varphi\colon B(K)\times B(K)\to B(K)$ by the formula
\begin{equation}\label{eq_nice_cocycle}
\varphi(u,v)=\prod_{p\in u}\prod_{q\in v}\left\{\frac{p+q}2\right\},
\end{equation}
where $\{\cdot\}$ denotes a one-point subset. It can be directly verified that the map~$\varphi$ satisfies the conditions~\eqref{eq_cond1_top}--\eqref{eq_cond3_top}, that is, it is indeed an involutive symmetric quasi-cocycle. Moreover, it is obvious that the map $\varphi$ is continuous and, moreover, quasi-continuous.
\end{example}

\begin{propos}\label{propos_strange_example}
Let $\varphi\colon B(K)\times B(K)\to B(K)$ be a quasi-cocycle given by the formula ~\eqref{eq_nice_cocycle}.
Then
\begin{itemize}
\item[\textnormal{(a)}] the quasi-cocycle~$\varphi$ becomes cohomologically trivial if we forget about the topology on the Boolean group~$B(K)$; in other words, $\varphi$ represents a trivial class in the group~$\CH(B(K))$;

\item[\textnormal{(b)}] quasi-cocycle~$\varphi$ is cohomologically non-trivial as a quasi-continuous quasi-cocycle, that is, it represents a non-trivial class in the group~$\CH_{\mathrm{qc}}(B(K))$.
\end{itemize}
\end{propos}

\begin{proof}
Let us prove statement~(a). Define a discontinuous map $\chi_0\colon B(K)\to B(K)$ by the formula
\begin{equation}\label{eq_chi_particular}
\chi_0(v)=\prod_{\{p,q\}\subset v,\ p\ne q}\left\{\frac{p+q}2\right\},
\end{equation}
where the product is taken over all unordered subsets~$\{p,q\}$ of the set~$v$. It is easy to see that
\begin{equation}\label{eq_chi_particular_homology}
\varphi(u,v)=\chi_0(u)\chi_0(v)\chi_0(uv)
\end{equation}
for all $u,v\in B(K)$, from which statement~(a) immediately follows. Note that the map~$\chi_0$ defined by formula~\eqref{eq_chi_particular} is not quasi-continuous, so the equality ~\eqref{eq_chi_particular_homology} does not imply the cohomological triviality of the quasi-cocycle ~$\varphi$ in the group~$\CH_{\mathrm{qc}}(B(K))$.

Let us prove statement~(b) by contradiction. Suppose that $\varphi$ is cohomologically trivial as a quasi-continuous quasi-cocycle. Then there is a quasi-continuous mapping $\chi\colon B(K)\to B(K)$ such that $\chi(e)=e$ and
$$
\varphi(u,v)\equiv \chi(u)\chi(v)\chi(uv)\pmod{\langle u,v\rangle}
$$
for all $u,v\in B(K)$.

Restricting~$\chi$ to one-element subsets of the set~$K$, we get the mapping $K\to B(K)$, $q\mapsto\chi(\{q\})$. This mapping is quasi-continuous in the sense that the two-valued mapping
\begin{equation}\label{eq_chi_2val}
q\mapsto\bigl[\chi(\{q\}),\{q\}\chi(\{q\})\bigr]
\end{equation}
is continuous.
Now note that of the two elements~$\chi(\{q\})$ and~$\{q\}\chi(\{q\})$ exactly one always lies in the connected component~$B_{even}(K)$ and the other one lies in~$B_{odd}(K)$. Therefore, if from these two elements we always choose the one that lies in~$B_{even}(K)$, we get a continuous one-valued branch of the two-valued mapping ~\eqref{eq_chi_2val}.
$$
\psi(\{q\})=\left\{
\begin{aligned}
&\chi(\{q\})&\text{if }\chi(\{q\})&\in B_{even}(K),\\
&\{q\}\chi(\{q\})&\text{if }\chi(\{q\})&\in B_{odd}(K).
\end{aligned}
\right.
$$
Then the map $q\mapsto\psi(\{q\})$ is continuous. We extend the map~$\psi$ to a continuous homomorphism $\psi\colon B(K)\to B_{even}(K)$ by setting
$$
\psi(v) = \prod_{q\in v}\psi(\{q\}).
$$

Let us replace the map~$\chi$ with a new map $\widetilde{\chi}\colon B(K)\to B(K)$ such that
$$
\widetilde\chi(v)=\chi(v)\psi(v)
$$
for all~$v\in B(K)$. Since $\psi$ is a homomorphism, we still have $\widetilde\chi(e)=e$ and
\begin{equation}\label{eq_chi_tilde}
\varphi(u,v)\equiv \widetilde\chi(u)\widetilde\chi(v)\widetilde\chi(uv)\pmod{\langle u,v\rangle}
\end{equation}
for all $u,v\in B(K)$. Moreover, it follows from the continuity of the homomorphism~$\psi$ that the map~$\widetilde\chi$ is still quasi-continuous.

By construction for any point $q\in K$ we have either $\widetilde\chi(\{q\})=e$ or $\widetilde\chi(\{q\})=\{q\}$. Let $p$ and~$q$ be two different points of the set~$K$. Congruence~\eqref{eq_chi_tilde} for the pair $(u,v)=(\{p\},\{q\})$ implies that
$$
\left\{\frac{p+q}2\right\}\subseteq \widetilde\chi(\{p,q\})\subseteq \left\{\frac{p+q}2,p,q\right\}.
$$
In addition, note that the set $\{p,q\}$ belongs to the same connected component~$B_{even}(K)$ of the space~$B(K)$ as the empty set~$e$. The quasi-continuity of the mapping~$\widetilde\chi$ means that the unordered pair of sets $\bigl[\widetilde\chi(w),w\widetilde\chi(w)\bigr]$ depends continuously on~$w$. If $w=e$, then both sets in this pair are empty and, therefore, lie in~$B_{even}(K)$. Therefore, both sets $\widetilde\chi(\{p,q\})$ and $\{p,q\}\widetilde\chi(\{p,q\})$ also lie in~$B_{even}(K)$. So $\widetilde\chi(\{p,q\})$ is one of the two sets $\{(p+q)/2,p\}$ and $\{(p+q)/2,q\}$.

Let now $p$, $q$ and~$r$ be three points in the set~$K$ such that $6$ points $p$, $q$, $r$, $(p+q)/2$, $(q+r)/2$, $(r+p)/2$ are pairwise distinct. Then it follows from the congruence ~\eqref{eq_chi_tilde} for the pair $(u,v)=(\{p,q\},\{r\})$ that the set~$\widetilde\chi(\{p,q,r\})$ contains exactly one of the two points~$p$ and~$q$. However, similarly, this set contains exactly one of the two points~$p$ and~$r$ and exactly one of the two points~$q$ and~$r$. Together, these three statements lead us to a contradiction, which completes the proof of the Proposition.
\end{proof}

\begin{cor}\label{cor_strange_example}
The tame topological two-valued group~$X_{B(K),\varphi}$ is not isomorphic (as a topological group) to any two-valued group~$X_A^{\ba}$ of the principal series and to any two-valued group~$X_{V,U}^{\bu}$ of the unipotent series.
\end{cor}

\begin{proof}
It follows from statement~(a) of Proposition~\ref{propos_strange_example} that as an abstract two-valued group, that is, after forgetting the topology, the two-valued group~$X_{B(K),\varphi}$ is isomorphic to the two-valued group~$X^{\bu}_{B(K)}$ of the unipotent series. From here it immediately follows (by the theorem ~\ref{theorem_non_fg}) that~$X_{B(K),\varphi}$ is not isomorphic to any group~$X_A^{\ba}$ of the principal series even without topology.

We have $\Phi\bigl(X_{B(K),\varphi}\bigr) = \bigl(B(K),B(K)\bigr)$. Therefore, the topological two-valued group~$X_{B(K),\varphi}$ certainly cannot be isomorphic to any topological two-valued group of the unipotent series, except for the two-valued group $X_{B(K),B(K)}^{\bu}= X_{B(K)}^{\bu}$. The absence of isomorphism between the topological groups~$X_{B(K)}^{\bu}$ and~$X_{B(K),\varphi}$ follows from the cohomological nontriviality of the quasi-cocycle~$\varphi$ in the group~$\CH_{\mathrm{qc}}(B(K))$ and Proposition~\ref{propos_2gqc_top}.
\end{proof}

\begin{remark}
The construction of the quasi-cocycle~$\varphi$ can be varied quite a lot so that Proposition~\ref{propos_strange_example} and Corollary~\ref{cor_strange_example} remain true (and their proofs are verbatim the same). Namely, we can consider any quasi-cocycle of the form
$$
\varphi(u,v)=\prod_{p\in u}\prod_{q\in v}F(p,q),
$$
where $F\colon K\times K\to B_{odd}(K)$ is an arbitrary symmetric continuous mapping such that $F(q,q)=\{q\}$ for all $q\in K$. In this case, only pathwise connectivity of the set~$K$ is necessary, and convexity is not important: we needed this only for our chosen specific map $F(p,q)=\{(p+q)/2\}$ to be well-defined.
\end{remark}

\section{Locally compact two-valued groups}\label{section_lc}

The most important class of topological groups is the class of locally compact (Hausdorff) topological groups. There is a rich structural theory for this class of groups, the development of which is primarily associated with the solution of Hilbert's fifth problem. The formulation of this problem in its most common interpretation is as follows: \textit{is this true that every topological group which is a topological manifold is isomorphic to a Lie group?} 

In connection with the questions considered in this paper, we are primarily interested in commutative locally compact groups. The foundations of their structural theory were laid in the works of L.S.~Pontryagin~\cite{Pon34},~\cite{Pon36}; in particular, he gave a positive solution to Hilbert's fifth problem for commutative groups. A positive solution to Hilbert's fifth problem in general case was obtained by E.~Gleason~\cite{Gle52} and D.~Montgomery and L.~Zippin~\cite{MoZi52}. Later, even stronger structure theorems were obtained for locally compact groups. 

The most important for us is the following theorem on locally compact groups without small subgroups. By definition, a topological group \textit{has no small subgroups} if it contains a neighbourhood of identity that does not contain nontrivial subgroups.

\begin{theorem}[Gleason--Yamabe theorem]
Let $G$ be a locally compact Hausdorff topological group without small subgroups. Then $G$ is isomorphic to a Lie group.
\end{theorem}

\begin{remark}
The statement of this theorem was proved by E.~Gleason~\cite{Gle52} under the additional assumption that the group is finite-dimensional and by H.~Yamabe~\cite{Yam53} in the general case, see also the monograph~\cite[Corollary~1.5.8]{Tao14}.
\end{remark}

It is well known that any commutative Lie group can be represented as a direct product of a compact torus~$(S^1)^m$, a vector group~$\R^n$, and a discrete abelian group. Thus, in the commutative case the Gleason--Yamabe theorem takes the following form.

\begin{cor}
Let $G$ be a commutative locally compact Hausdorff topological group without small subgroups. Then $G\cong (S^1)^m\times\R^n\times A$, where $m\ge 0$, $n\ge 0$ and $A$ is a discrete abelian group.
\end{cor}

The natural question is how this result can be carried over to the case of involutive commutative two-valued groups. As in the case of ordinary groups, we say that a topological two-valued group \textit{has no small subgroups} if it contains a neighbourhood of identity that does not contain non-trivial two-valued subgroups.

\begin{theorem}\label{theorem_lc}
Let $X$ be an involutive commutative locally compact Hausdorff topological two-valued group without small subgroups. Then one of two possibilities holds:
\begin{enumerate}
\item $X\cong \bigl((S^1)^m\times\R^n\times A\bigr)/\iota_{\ba}$, where $m\ge 0$, $n\ge 0$, $m+n>0$, $A$ is a discrete abelian group and $\iota_{\ba}$ is the antipodal involution;
\item $X$ is discrete.
\end{enumerate}
\end{theorem}

\begin{remark}
Discrete two-valued groups are, of course, always locally compact and do not contain small subgroups. The classification of discrete involutive commutative two-valued groups is given by the theorem ~\ref{theorem_non_fg}.
\end{remark}

To prove Theorem ~\ref{theorem_lc}, we need two simple Lemmas.

\begin{lem}\label{lem_open_e}
Let $X$ be a Hausdorff topological two-valued group with identity~$e$. Assume that the one-element set~$\{e\}$ is open. Then $X$ is discrete.
\end{lem}

\begin{proof}
Consider in $\Sym^2(X)$ the subset~$W$ consisting of all multisets of the form~$[e,x]$, where $x\in X$. The fact that the set~$\{e\}$ is open in~$X$ immediately implies that the set $W$ is open in~$\Sym^2(X)$. For any element $y\in X$, multiplication by ~$y$ defines a continuous map $\mu_y\colon X\to\Sym^2(X)$. Since $e\in x*y$ if and only if $x=y$, then $\mu_y^{-1}(W)=\{y\}$. Thus, all one-element sets~$\{y\}$ are open, hence the topology on~$X$ is discrete.
\end{proof}

\begin{lem}\label{lem_neigh_wo_24}
Let $X$ satisfy the conditions of theorem ~\ref{theorem_lc}. Then in~$X$ there is a neighbourhood of the identity that does not contain elements of orders~$2$ and~$4$.
\end{lem}

\begin{proof}
Let $U$ be an open neighbourhood of the identity of a two-valued group~$X$ that does not contain nontrivial two-valued subgroups. Let $U'$ be the set consisting of all elements of $x\in X$ such that $x^2\in U$. Since the map $x\mapsto x^2$ is continuous, $U'$ is also an open neighbourhood of the identity.

Note that the set~$U$ does not contain elements of order~$2$. Indeed, if in~$U$ there were an element of order~$2$, that is, an element~$x\ne e$ such that $x*x=[e,e]$, then the set~$\{e,x\}$ would be a non-trivial two-valued subgroup contained in~$U$. The square of any element of order~$4$ is an element of order~$2$. Moreover, the squares of elements from~$U'$ lie in~$U$. Therefore, the set~$U'$ does not contain elements of order~$4$. Thus, as the desired neighbourhood, we can take the set~$U\cap U'$.
\end{proof}

\begin{proof}[Proof of Theorem~\ref{theorem_lc}]
Let $V$ be an open neighbourhood of the identity without elements of order~$2$ and~$4$ given by Lemma~\ref{lem_neigh_wo_24}. Assume that the two-valued group~$X$ is not discrete. Then by Lemma~\ref{lem_open_e} the one-element set~$\{e\}$ is not open. Hence, the neighbourhood~$V$ contains at least one element $t$ different from~$e$. Then $\ord t\notin \{1,2,4\}$. By Theorem~\ref{theorem_nonspecial_top} there is an isomorphism $X\cong G/\iota_{\ba}$, where $G$ is a commutative Hausdorff topological group and $\iota_{\ba}$ is the antipodal involution. Using the fact that the projection $G\to X$ is a two-fold Smith-Dold ramified covering, it is easy to deduce from the local compactness of the two-valued group~$X$ the local compactness of the group~$G.$ Now, by the Gleason--Yamabe theorem there is an isomorphism $G\cong (S^1)^m\times\R^n\times A$ with $m+n>0$, since $X$ is not discrete.
\end{proof}

For compact two-valued groups without small subgroups, Theorem \ref{theorem_lc} leads to a complete classification similar to Theorem ~\ref{theorem_main_fg}.

\begin{theorem}\label{theorem_compact}
Let $X$ be an involutive commutative compact Hausdorff topological two-valued group without small subgroups. Then $X$ is isomorphic to one of the following two-valued groups:
\begin{itemize}
\item[\textnormal(1)] $X^{\ba}_{d_1,\ldots,d_k}(m)=\bigl((S^1)^m\times C_{d_1}\times\cdots\times C_{d_k}\bigr)/\iota_{\ba}$, where $2\le d_i<\infty$ and $d_1|d_2|\cdots|d_k$;
\item[\textnormal(2)] $X^{\bu}_n\times C_2^m$, where $n\ge 1$, $m\ge 0$;
\item[\textnormal(3)] $Y_n\times C_2^m$, where $n\ge 1$, $m\ge 0$.
\end{itemize}
All isomorphisms between the listed two-valued groups are exhausted by the isomorphisms~\eqref{eq_main_iso1},~\eqref{eq_main_iso2} from Theorem ~\ref{theorem_main_fg}.
\end{theorem}

\begin{proof}
The theorem follows almost immediately from Theorems~\ref{theorem_lc} and ~\ref{theorem_main_fg}.
We only need to verify that the topological two-valued groups~$X^{\ba}_{d_1,\ldots,d_k}(m)$ corresponding to different collections $(m;d_1,\ldots,d_k)$ are pairwise non-isomorphic and, for~$m\ge 1$, are not isomorphic to two-valued groups of the unipotent and special series. To do this, note that the connected component of the identity in a two-valued group $X=X^{\ba}_{d_1,\ldots,d_k}(m)$ is the group~$X_0=X^{\ba}_{\varnothing}(m) = (S^1)^m/\iota_{\ba}$, and the quotient group~$X/X_0$ is isomorphic to the two-valued group $X^{\ba}_{d_1,\ldots,d_k}$. Moreover, the topological two-valued groups~$X^{\ba}_{\varnothing}(m)$ are pairwise non-isomorphic, since they have different dimensions, and the two-valued groups $X^{\ba}_{d_1,\ldots,d_k}$ are pairwise non-isomorphic by Lemma ~\ref{theorem_main_fg}. This completes the proof of the theorem.
\end{proof}

\section{Algebraic two-valued groups}\label{section_alg}

We now discuss the algebraic version of commutative involutive groups following \cite{B-90,Buc06, BuVe19}. For simplicity, we will work over the field of complex numbers~$\mathbb C.$

In the simplest setting, the algebraic two-valued multiplication law $(x,y)\mapsto z=x*y$ is given in local coordinates by the equation $F(x,y,z)=0$, where $F(x,y,z)$ is a polynomial, which is symmetric in the variables $x,y,z$ and of degree 2 in each of them.
The two-valued multiplication law must satisfy the associativity condition
$$
(x*y)*z=x*(y*z)
$$
in the sense of the correspondence theory: the systems of equations
$$
\left\{
\begin{aligned}
 F(x,y,u)&=0,\\ F(u,z,w)&=0
\end{aligned}
\right.
$$
and
$$
\left\{
\begin{aligned}
F(y,z,v)&=0, \\ F(x,v,w)&=0
\end{aligned}
\right.
$$
after the elimination of $u$ and $v$ define the same set $(x,y,z,w) \in \mathbb C^4.$
We also assume that $$F(0,y,z)=(z-y)^2,$$ which means that 0 is the (strong) identity of the group. Note that due to symmetry we have the equality $F(x,y,0)=(x-y)^2,$ which means that the corresponding group is automatically involutive. And vice versa, the addition law of an involutive two-valued group must be symmetric in all three variables, see Lemma ~\ref{lem_basic}.

Such groups were classified in \cite{B-90,BuVe19}, where it was shown that the corresponding polynomial must be of the form
\begin{equation}
\label{univ}
F(x,y,z)=(x+y+z-a_2 xyz)^2 -4(1 + a_3 xyz) (xy+xz+yz+a_1 xyz),
\end{equation}
where $a_1,a_2,a_3$ are arbitrary parameters. In general, the equation $F=0$ for such a polynomial defines the coset group $\mathcal E/\iota_{\ba}$, where $\mathcal E$ is an elliptic curve (considered as an abelian group) and $\iota_{\ba}\colon \mathcal E \to \mathcal E$ is the antipodal involution.

More precisely, consider an elliptic curve in the standard Weierstrass form $$v^2=4u^3-g_2u-g_3$$ and a point $\alpha$ on it, and set the corresponding parameters as
\begin{equation}
\label{param}
a_1=3\wp(\alpha),\quad a_2=3\wp(\alpha)^2-\frac{g_2}{4},\quad a_3=\frac{1}{4}(4\wp(\alpha)^3-g_2\wp(\alpha)-g_3),
\end{equation}
where $\wp$ is the classical Weierstrass elliptic function satisfying the equation
$$
(\wp')^2=4\wp^3-g_2\wp-g_3.
$$

\begin{theorem}[\cite{B-90,BuVe19}]
The equation $F(x,y,z)=0$ with the symmetric polynomial $F(x,y,z),$ satisfying the condition $F(0,y,z)=(z-y)^2,$ defines the structure of an algebraic two-valued group if and only if $F$ has the form (\ref{univ}).

The multiplication law $F(x,y,z)=0$ with $F$ of the form (\ref{univ}) with parameters (\ref{param}) is reduced to the form $X\pm Y\pm Z=0$ by a change of variables
\begin{equation}\label{conju}
x=(\wp(X)+\wp(\alpha))^{-1},\quad y=(\wp(Y)+\wp(\alpha))^{-1},\quad z=(\wp(Z)+\wp(\alpha))^{-1}.
\end{equation}
The corresponding two-valued group is isomorphic to the involutive coset two-valued group $\mathcal E/\iota_{\ba}=\mathbb{CP}^1$, where $\mathcal E$ is the elliptic curve given by the equation
\begin{equation}\label{curve}
v^2=u^3+a_1u^2+a_2u+a_3,
\end{equation}
and $\iota_{\ba}$ is its involution $v\mapsto -v.$

In the case where the roots of the polynomial on the right side coincide, the elliptic curve degenerates into a rational one, and the corresponding two-valued involutive groups are isomorphic to the coset groups $\mathbb C^*/\iota_{\ba}$ and $\mathbb C/\iota_{\ba}$, where $\iota_{\ba}(z)=z^{-1}$ and $\iota_{\ba}(z)=-z$ respectively.
\end{theorem}

Interestingly, the general family (\ref{univ}) includes all cases of the formal two-valued groups that first appeared in topology in the works of V.\,M.~Buchstaber and S.\,P.~Novikov \cite{BuNo71,B-73,B-75}: the case with $a_1=0$ corresponds to elliptic cohomology, the case $a_2=a_3=0$ corresponds to $K$-theory, and the most degenerate case $a_1=a_2=a_3=0$ corresponds to the ordinary cohomology (for details see \cite{BuVe19}).

Note that, as it follows from the parametrization (\ref{conju}), from the algebro-geometric point of view, the general equation (\ref{univ}) defines the (affine part) of the special Kummer surface $\mathcal E\times\mathcal E/{\pm I}$. General Kummer surfaces corresponding to Jacobi varieties of curves of genus 2 were discussed in V.\,M.~Buchstaber and V.~Dragovic \cite{BD} in connection with integrable billiards. The corresponding two-valued groups on the Kummer surfaces embedded in $\mathbb{CP}^3$ are explicitly described in terms of the addition laws for the Klein $\wp$-functions \cite{BL}.

It is natural to assume that general algebraic two-valued commutative involutive groups are obtained by the coset construction from abelian varieties and their degenerations, but the details of a precise formulation need to be clarified. The results of A.~Pillay~\cite{Pil09}, using the Hrushovski theorem from model theory, may be useful for the proof.

The general theory of the one-dimensional two-valued formal groups was actually completed by V.\,M.~Buchstaber in \cite{B-75},~\cite{B-75-2}.
In particular, the concept of canonical invariant operators $d_n$ was introduced there and it was shown that for two-valued formal addition laws given by the relation
$$
z^2-\Theta_1(x,y)z+\Theta_2(x,y)=0,
$$
the canonical operator $d_1$ has the form
\begin{equation}\label{canon}
d_1=\frac{1}{2}\varphi_1(x)\frac{d}{dx}+\frac{1}{8}\varphi_2(x) \frac{d^2}{dx^2},
\end{equation}
where
\begin{gather*}
\varphi_1(x)=\frac{\partial \Theta_1(x,y)}{\partial y}|_{y=0}, \qquad \varphi_2(x)=\frac{\partial \sigma(x,y)}{\partial y}|_{y=0}, \\ \sigma(x,y)=\Theta_1^2(x,y)-4\Theta_2(x,y).
\end{gather*}
It was also shown there that the algebra of invariant operators on the group is freely generated by the canonical operator $d_1$ and a general formula was found for the exponential of such groups.

It is easy to check that for algebraic two-valued groups (\ref{univ}) the canonical operator (\ref{canon}) in the parametrization (\ref{conju}) is simply the second derivative operator $\frac{d^2}{dX^2}$ in the elliptic coordinate $X$ (or, the special case of the Lamé operator $\mathcal L=\frac{d^2}{dX^2}-m(m+1)\wp(X)$ with $m=0$), which generates an algebra of differential operators that are invariant under shifts on an elliptic curve (\ref{curve}) and involution $\iota_{\ba}$. The corresponding exponential is written explicitly in elliptic functions in \cite{B-90}.

\section{Concluding remarks}\label{section_conclusion}

The structural theory of $n$-valued groups is at the very beginning of its development and most of the questions are open.

In our paper, we completed the classification of the two-valued commutative involutive groups in finitely generated case and in some important topological classes.

In the general case, the most promising direction seems to be the development of the theory of the
$n$-valued Lie groups. In particular, in the case of the $n$-valued group $G/H$, $H \subset \Aut G$ defined by the coset construction from the Lie group $G$, it is natural to study the algebra of differential operators on $G$, which are invariant under (say, left) shifts on $G$ and actions of the automorphism group $H.$
The corresponding algebra can be considered as an analogue of the universal enveloping algebra for the $n$-valued Lie group, the study of which in the general case seems to be an important and interesting problem. 

The example of an elliptic two-valued group shows the existence of a structure of a two-valued Lie group on a two-dimensional sphere, which does not admit the usual group structure. The question of topological obstructions to a $n$-valued group structure is of great interest.
As shown by H.~Hopf ~\cite{Hopf41}, if a topological space admits a (single-valued) multiplication with identity, then in the cohomology ring of this space there is a structure of the Hopf algebra. The nonexistence of such a structure is an obstruction to the existence of multiplication on a topological space. It was shown by V.\,M.~Buchstaber and E.~Rees~\cite{BR-96, BR-97} (see also \cite{BR-08}) that the existence of an $n$-valued multiplication on a space leads to a similar structure, which they called \textit{Hopf $n$-algebra}. In the case of four-dimensional manifolds, the corresponding obstructions were studied by T.E.~Panov~\cite{Panov}. He obtained an explicit list of homotopy types of four-dimensional manifolds whose cohomology admits the structure of a Hopf $2$-algebra. As a result, only manifolds from this list may admit two-valued multiplication. Note that the simply-connected singular Kummer variety $\mathcal{E}\times\mathcal{E}/\pm I\approx T^4/\iota_{\ba}$ is a two-valued group, which realizes the case of the second Betti number $6$ pointed out by Panov.  A number of results on multivalued multiplications on spheres were obtained by D.V.~Gugnin~\cite{Gugnin}.

In the theory of formal two-valued groups, significant advances were made in the 1970s in the works of V.M.~Buchstaber \cite{B-73,B-75}, who was motivated by the applications to the cobordism theory \cite{BuNo71}. These works, in particular, revealed a deep connection with the Delsarte-Levitan theory of generalized shift operators \cite{Lev73} (see also in this direction the paper~\cite{Gur}).
The methods of papers \cite{B-73,B-75}, as well as connections with the theory of quantum groups and Hopf algebras \cite{BR-96, BR-97}, may be useful in the more general context of our problems.

In the theory of $n$-valued groups there are important constructions arising from various branches of mathematics. One of such constructions is the  $n$-valued group structure on the set of the irreducible representations of a noncommutative group, see ~\cite{Buc06}. It is natural to consider this construction as an interpretation in group-theoretic terms of the Tanaka--Krein duality (see ~\cite{BaIt87}), which is an analogue of Pontryagin duality for non-commutative groups. The duality theory for the $n$-valued groups is discussed in ~\cite{BuMo03}, ~\cite{Yag09},~\cite{Yag10}.

Let us mention one more construction. The group algebra of a coset $n$-valued group is a subalgebra of the group algebra of the original group generated by the elements given by the sums of the points of the orbits of the action of the automorphism group. A more general example of group origin is given by the so-called Schur rings (see ~\cite{Wie}). In connection with the problem of Schur rings some examples of non-coset $n$-valued groups were constructed using number-theoretic methods in ~\cite{BVEP96}.

Finally, there are important connections with the theory of integrable systems ~\cite{BuVe96,Ves91, Ves92}, where also many questions remain open (see, for example, the discussion and references in the review \cite{BuVe19}).
In particular, in the works of A.P.~Veselov \cite{Ves91MZ, Ves92} it was proposed an approach to the integrability of multi-valued mappings in terms of the growth of the number of images of their iterations. The connection of this approach with the theory of multivalued groups is discussed in \cite{BuVe96}.
It is interesting to note that the original proof of Gromov's remarkable theorem that finitely generated groups of polynomial growth are almost nilpotent, uses the results of E.~Gleason, D.~Montgomery and L.~Zippin, see ~\cite{Gromov}.

\end{document}